\documentclass[11pt]{article}
\usepackage{latexsym,amssymb,amsmath} % for \Box, \mathbb, split, etc.
\usepackage{path}
\usepackage{verbatim}
\usepackage{amsfonts}
\usepackage{amsfonts}
\usepackage{array}
\usepackage{mathrsfs}
\usepackage{amsfonts}
\usepackage{amsmath}
\usepackage{amssymb}
\usepackage{pifont}
\usepackage{amsthm}
\usepackage{bm}
\usepackage{appendix}
\usepackage{multirow}
\usepackage{pifont}
\usepackage{graphicx} % more modern
\usepackage{subfigure}
\usepackage{color}
\usepackage[table]{xcolor}
\usepackage{thmtools}
\usepackage{thm-restate}
\usepackage{enumitem}

\usepackage{graphicx}
\usepackage{epsfig}
\usepackage{subfigure}
\usepackage{epstopdf}

\definecolor{gray}{rgb}{0.85,0.85,0.85}
\definecolor{yama}{rgb}{0.98, 0.87, 0.68}
\definecolor{lightskyblue}{rgb}{0.53, 0.81, 0.98}

\graphicspath{{figures/}}

\usepackage[lined,boxed,ruled, commentsnumbered]{algorithm2e}

\newtheorem{lemma}{Lemma}

\newtheorem{definition}{Definition}
\newtheorem{remark}{Remark}
\newtheorem{assumption}{Assumption}

\newcommand{\supp}{\text{supp}}

\newcommand{\argmin}{\mathop{\arg\min}}

\newcommand{\tabincell}[2]{\begin{tabular}{@{}#1@{}}#2\end{tabular}}

\newcommand{\st}{\text{subject to }}

\newcommand{\junk}[1]{{}}

\newlength{\fwtwo} 
\DeclareGraphicsExtensions{.jpg}

\title{Generalization Bounds for High-dimensional M-estimation under Sparsity Constraint}
\author{
  Xiao-Tong Yuan,\ \  Ping Li \\
  Cognitive Computing Lab\\
  Baidu Research \\
  E-mail: \{\texttt{xtyuan1980@gmail.com}, \texttt{pingli98@gmail.com}\}
  }
\date{}

\begin{document}

\maketitle

\begin{abstract}
The $\ell_0$-constrained empirical risk minimization ($\ell_0$-ERM) is a promising tool for high-dimensional statistical estimation. The existing analysis of $\ell_0$-ERM estimator is mostly on parameter estimation and support recovery consistency. From the perspective of statistical learning, another fundamental question is how well the $\ell_0$-ERM estimator would perform on unseen samples. The answer to this question is important for understanding the learnability of such a non-convex (and also NP-hard) M-estimator but still relatively under explored.

In this paper, we investigate this problem and develop a generalization theory for $\ell_0$-ERM. We establish, in both white-box and black-box statistical regimes, a set of generalization gap and excess risk bounds for $\ell_0$-ERM to characterize its sparse prediction and optimization capability. Our theory mainly reveals three findings: 1) tighter generalization bounds can be attained by $\ell_0$-ERM than those of $\ell_2$-ERM  if the risk function is (with high probability) restricted strongly convex; 2) tighter uniform generalization bounds can be established for $\ell_0$-ERM than the conventional dense ERM; and 3) sparsity level invariant bounds can be established by imposing additional strong-signal conditions to ensure the stability of $\ell_0$-ERM. In light of these results, we further provide generalization guarantees for the Iterative Hard Thresholding (IHT) algorithm which serves as one of the most popular greedy pursuit methods for approximately solving $\ell_0$-ERM. Numerical evidence is provided to confirm our theoretical predictions when implied to sparsity-constrained linear regression and logistic regression models.
\end{abstract}


%\subparagraph{AMS subject classifications (2010).} Provide up to
%five subject classification codes here; search for the string
%``MSC'' at \url"www.ams.org".

\section{Introduction}

We are interested in developing sparse learning theory for the problem of high-dimensional stochastic risk minimization~\cite{shalev2010learnability,vapnik2013nature}
\[
\min_{w \in \mathcal{W}} F(w):= \mathbb{E}_{\xi \sim D} [\ell(w;\xi)],
\]
where $w\in \mathcal{W}\subseteq \mathbb{R}^p$ is the model parameter vector, $\ell(w;\xi)$ is a non-negative convex function that measures the loss of $w$ at data sample $\xi \in \mathcal{X}$, $D$ represents a random distribution over $\mathcal{X}$. In realistic problems, the mathematical formulation of $D$ is typically unknown and thus it is hopeless to directly optimize such a population form. Alternatively, given a set of i.i.d. training samples $S=\{\xi_i\}_{i=1}^n \in \mathcal{X}^n$ drawn from $D$, the following sparsity-constrained empirical risk minimization problem is often considered for learning sparse models in high-dimensional settings~\cite{bach2012optimization,donoho2006compressed,hastie2015statistical}:
\begin{equation}\label{prob:general}
w_{S,k} = \argmin_{w \in \mathcal{W}} F_S(w):= \frac{1}{n}\sum\limits_{i=1}^n \ell(w;\xi_i) \quad \st \|w\|_0\le k,
\end{equation}
where the model cardinality constraint $\|w\|_0\le k$ is imposed for enhancing learnability and interpretability of model when $p\gg n$ which might usually be the case in the era of big data. We refer the above sparse M-estimation problem as $\ell_0$-ERM.

In this paper, we address the fundamental question of how well the population risk $F(w_{S,k})= \mathbb{E}_{\xi \sim D} [\ell(w_{S,k};\xi)]$ is approximated by the empirical risk $F_S(w_{S,k})=\frac{1}{n}\sum_{i=1}^n \ell(w_{S,k};\xi_i)$ at the $\ell_0$-ERM estimator $w_{S,k}$. The value $\Delta_{S,k}:=F(w_{S,k}) - F_S(w_{S,k})$ is referred to as \emph{generalization gap} of $\ell_0$-ERM at sample $S$. Our primary goal is to find a sample size vanishing bound $\delta_n$, as tight as possible, such that $\Delta_{S,k} \le \delta_n$ holds in expectation or with high probability. Denote $F(\bar w)= \min_{\|w\|_0\le k} F(w)$ as the population optimal value under sparsity constraint. The convergence of $\Delta_{S,k}$ consequently implies that $F(w_{S,k}) - F(\bar w)$, the \emph{excess risk} (a.k.a. \emph{population sub-optimality}~\cite{shalev2009stochastic}), also converges with respect to sample size~\footnote{For example, $\mathbb{E}_S\left[F(w_{S,k}) - F(\bar w)\right] = \mathbb{E}_S\left[\Delta_{S,k} + F_S(w_{S,k}) -F_S(\bar w) + F_S(\bar w) - F(\bar w)\right] \le \mathbb{E}_S \left[\Delta_{S,k}\right]$.}.

Due to the presence of cardinality constraint, $\ell_0$-ERM is simultaneously non-convex and NP-hard even when the loss function is convex~\cite{natarajan1995sparse}, which makes it computationally intractable to solve the problem exactly in general cases. Therefore, one must instead seek approximate solutions instead of combinatorial search over all possible models. Among others, the Iterative Hard Thresholding (IHT)~\cite{blumensath2009iterative} is a family of first-order greedy selection methods popularly used and studied for solving $\ell_0$-ERM with outstanding practical efficiency and scalability witnessed in many applications~\cite{jain2014iterative,jin2016training,yuan2018gradient,zhou2018efficient}. The common theme of IHT-style algorithms is to iterate between gradient descent and hard thresholding to maintain sparsity of solution while minimizing the objective value. In the considered problem setting, a plain IHT iteration is given by
\begin{equation}\label{prob:iht}
w_{S,k}^{(t)} = \mathrm{H}_k \left(w_{S,k}^{(t-1)} - \eta\nabla F_S(w_{S,k}^{(t-1)})\right),
\end{equation}
where $\mathrm{H}_k(\cdot)$ is the truncation operator that preserves the top $k$ (in magnitude) entries of input and sets the remaining to be zero, with ties broken arbitrarily. The procedure is typically initialized with all-zero vector $w^{(0)}$. In this paper, we are also interested in understanding the generalization performance of the estimate $w_{S,k}^{(t)}$ output by IHT after sufficient rounds of iteration.

\begin{table}
%\renewcommand{\arraystretch}{1.2}
\centering
\begin{tabular}{|c|c|c|c|}
\hline
Result & Measurement & High Prob. Bound & Exp. Bound  \\
\hline
\multirow{2}{*}{\tabincell{c}{Theorem~\ref{thrm:generalization_barw} \\ ($L_0$-ERM, white-box) \\ }} & $F( w_{S,k}) - F_S( w_{S,k})$  & $\mathcal{\tilde O} \left(\frac{k\log (p)}{n}+\frac{1}{\sqrt{n}}\right)$ & $\mathcal{O} \left(\frac{k\log (p)}{n}\right)$\\
&$F( w_{S,k}) - F(\bar w)$  & $\mathcal{\tilde O} \left(\frac{k\log (p)}{n}\right)$ & $\mathcal{O} \left(\frac{k\log (p)}{n}\right)$ \\
\hline
%\tabincell{c}{Proposition~\ref{thrm:uniform_stability} \\ (uniform)} &$\sup_{\|w\|_0\le k}|F(w) - F_S(w)|$ &  $\mathcal{\tilde O}\left(\sqrt{\frac{k^2\log(n)\log(p)}{n}}\right)$ & --- \\
\multirow{2}{*}{\tabincell{c}{ Theorem~\ref{thrm:universe_generalizaion} \\ ($L_0$-ERM, uniform) }} & \multirow{2}{*}{\tabincell{c}{$\sup_{\|w\|_0\le k}\left|F(w)  - F_S(w)\right| $}} & \multirow{2}{*}{\tabincell{c}{$\mathcal{\tilde O}\left( \sqrt{\frac{k\log(p)}{n}} \right)$}} & \multirow{2}{*}{\tabincell{c}{---}} \\
& &  & \\
%\hline
%\multirow{2}{*}{\tabincell{c}{ Proposition~\ref{thrm:uniform_stability} \\ (Black-box) }} &$F(w_{S,k}) - F_S( w_{S,k})$ & $\mathcal{\tilde O}\left(\sqrt{\frac{\log^2(n) + k\log(n)\log(p/k)}{n}} \right)$ & --- \\
%&$F( w_{S,k}) - F(\bar w)$ & $\mathcal{\tilde O}\left(\sqrt{\frac{\log^2(n) + k\log(n)\log(p/k)}{n}} \right)$ & --- \\
\hline
\multirow{2}{*}{\tabincell{c}{ Theorem~\ref{thrm:generalizaion_support_recovery_high} \\ ($L_0$-ERM, black-box) }} &$F(w_{S,k}) - F_S( w_{S,k})$ & $\mathcal{\tilde O}\left(\frac{\log(n)}{\sqrt{n}} \right)$ & $\mathcal{O}\left(\frac{1}{n}\right)$ \\
&$F( w_{S,k}) - F(\bar w)$ & $\mathcal{\tilde O}\left(\frac{\log(n)}{\sqrt{n}} \right)$ & $\mathcal{O}\left(\frac{1}{n}\right)$ \\
\hline
\hline
\multirow{2}{*}{\tabincell{c}{ Theorem~\ref{thrm:uniform_stability_strong_iht} \\ (IHT, black-box) }} &$F(w^{(t)}_{S,k}) - F_S( w^{(t)}_{S,k})$ & $\mathcal{\tilde O}\left(\frac{\log(n)}{\sqrt{n}} \right)$ &--- \\
&\tabincell{c}{$F( w^{(t)}_{S,k}) - F(\bar w)$ \\ ($t \ge \mathcal{\tilde O}(\log(n))$) } & $\mathcal{\tilde O}\left(\frac{\log(n)}{\sqrt{n}} \right)$ & ---\\
\hline
\end{tabular}
\caption{Overview of our main results on the generalization gap and excess risk bounds for $\ell_0$-ERM and IHT in statistical white-box (upper panel) and black-box (lower panel) regimes respectively. For $\ell_0$-ERM, the target solution is $\bar w=\argmin_{\|w\|_0\le k}F(w)$. For IHT, the target solution is $\bar w =\argmin_{\|w\|_0\le \bar k} F(w)$ for proper $\bar k < k$. \label{tab:results}}
\end{table}

\textbf{Our results:} The main result of this work is a set of novel generalization bounds, in expectation and/or with high probability, for the $\ell_0$-ERM estimator and IHT. Our analysis simultaneously covers a statistical \emph{white-box} setting where the data is assumed to be generated according to a sub-Gaussian model with sparse parameters, and a \emph{black-box} but more realistic setting where the prior information of the data generation process is unknown.
%We substantialize the bounds obtained in the white-box setting to sparse linear regression and logistic regression models to demonstrate the applicability of our results.
Table~\ref{tab:results} summarizes our main results, which are highlighted in details below:
\begin{itemize}
  \item \emph{White-box generalization results.} We first consider a white-box regime in which we assume that there exists a nominal sparse model $\bar w$ such that the sample-wise loss gradient $\nabla \ell(\bar w;\xi)$ is element-wise $\sigma^2$-sub-Gaussian with zero-mean. This assumption implies that $\nabla F(\bar w) =0$, i.e., $\bar w$ attains a global minimizer of the population risk $F$. In this well-specified setting, on top of the standard $\ell_2$-norm estimation error bounds of $\ell_0$-ERM, we show in Theorem~\ref{thrm:generalization_barw} that with high probability, the generalization gap is bounded by $\Delta_{S,k} \le \mathcal{\tilde O} \left(k\log(p)/n+1/\sqrt{n}\right)$ while the excess risk by $F(w_{S,k}) - F(\bar w)\le \mathcal{\tilde O} \left(k\log(p)/n\right)$. Here we use the big o notation $\mathcal{\tilde O}$ to hide the logarithmic factors other than $n,p,k$. The corresponding bounds in expectation are both of order $\mathcal{O}(k\log(p)/n)$. These bounds have been substantialized to sparse linear regression and logistic regression models to demonstrate their applicability. We notice that up to logarithmic factors, the excess risk bounds established in this setting are minimax optimal over the cardinality constraint and they match the results in~\cite{abramovich2018high} for sparsity-penalized logistic regression in misclassification excess risk.
  \item \emph{Black-box generalization results.} We then turn to a more realistic black-box regime in which we do not impose any distribution-specific assumptions on the data generative model. Under proper regularization conditions on sample size, we first establish in Theorem~\ref{thrm:universe_generalizaion} a uniform convergence bound $\sup_{\|w\|_0\le k}|F(w) - F_S(w)|\le\mathcal{\tilde O}\left(\sqrt{k\log(p)/n}\right)$ over a bounded domain of interest, which extends the uniform bound for conventional dense ERM~\cite{shalev2009stochastic} to $\ell_0$-ERM. For restricted strongly convex risk functions, based on the stability arguments and a recent uniform stability theory~\cite{feldman2019high}, we show in Proposition~\ref{thrm:uniform_stability} that the generalization gap and excess risk of $\ell_0$-ERM can be upper bounded by $\mathcal{\tilde O}\left(\sqrt{(\log^2(n)+k\log(n)\log(ep/k))/n}\right)$ without assuming bounding conditions on sample size and domain of interest. By imposing some additional strong-signal conditions, we further establish in Theorem~\ref{thrm:generalizaion_support_recovery_high} the $\mathcal{\tilde O}(\log(n)/\sqrt{n})$ (with high probability) and $\mathcal{O}(1/n)$ (in expectation) generalization bounds for $\ell_0$-ERM which are known to be rate-optimal even for the $\ell_2$-ERM.
  \item \emph{Generalization results for IHT}. In light of the derived generalization bounds for $\ell_0$-ERM and the computational complexity bounds of IHT that have been established in~\cite{jain2014iterative,yuan2018gradient}, we are able to analyze the generalization bounds of IHT after sufficient rounds of iteration. In the white-box statistical setting, our results in this line basically reveal that the generalization bounds of IHT are determined by those of $\ell_0$-ERM. While in the black-box setting, given that the population risk function $F$ is \emph{stable with respect to IHT} (see Definition~\ref{def:iht_stability} for a formal definition) up to the desired rounds of iteration, we further prove in Theorem~\ref{thrm:uniform_stability_strong_iht} the $\mathcal{\tilde O}(\log(n)/\sqrt{n})$ high probability bounds for IHT.
\end{itemize}

\textbf{Paper organization:} The paper proceeds with the material organized as follows: In Section~\ref{sect:related_work} we briefly review the related literature. In Section~\ref{sect:generalization_l0_erm} and Section~\ref{sect:generalization_iht} we respectively present the generalization bounds for the $\ell_0$-ERM estimator and the IHT algorithm. The numerical study for theory verification is provided in Section~\ref{sect:experiment}. The concluding remarks are made in Section~\ref{sect:conclusion}. All the technical proofs are relegated to the appendix sections.

\section{Related Work}
\label{sect:related_work}

The problem regime considered in this paper lies at the intersection of high-dimensional sparse M-estimation and statistical learning theory, both of which have long been studied with a vast body of beautiful and deep theoretical results established in literature. Next we will incompletely connect our research to several closely relevant lines of study in this context. We refer the interested readers to~\cite{cesa2006prediction,hastie2015statistical,rigollet201518} and the references there in for a more comprehensive coverage of the related topics.

\textbf{Consistency and generalization of M-estimation with sparsity.} Statistical consistency of $\ell_0$-ERM~\eqref{prob:general} for estimating an underlying true sparse model is now well understood for some popular statistical M-estimation models including linear regression, logistic regression and principle component analysis~\cite{foucart2017mathematical,rigollet201518,yuan2016exact}. To avoid the hardness of $\ell_0$-constraint in terms of global guarantees and computation complexity, convex relaxation-based methods such as those $\ell_1$-penalized estimations (Lasso)~\cite{tibshirani1996regression} were alternatively extensively studied with strong consistency guarantees obtained on parameter estimation and variable selection~\cite{bellec2018slope,li2015sparsistency,loh2012high,meinshausen2009lasso,ravikumar2011high,wainwright2009sharp}. The folded concave penalization, such as the SCAD~\cite{fan2001variable} and the MCP~\cite{zhang2010nearly}, can correct the intrinsic estimation bias of the Lasso and can achieve variable selection consistency under substantially weaker conditions than those of the Lasso~\cite{fan2011nonconcave,fan2004nonconcave,fan2014strong,zhang2012general}. The generalization ability of sparsity-inducing learning models is relatively less understood but has received recent attention. The out-of-sample predictive risk of least squares Lasso estimator was analyzed in~\cite{chatterjee2013assumptionless}. The misclassification excess risk of sparsity-penalized binary logistic regression was investigated in~\cite{abramovich2018high} with near optimal bounds established. For linear prediction models, a data dependent generalization error bound was derived for a class of risk minimization algorithms with structured sparsity constraints~\cite{maurer2012structured}. In the context of deep learning, it was justified in~\cite{arora2018stronger} that sparsity benefits considerably to the generalization performance of deep neural networks. Despite the remarkable success achieved in understanding sparsity models, the generalization theory of $\ell_0$-ERM still remains far less studied. To our knowledge, the most closely related results to ours are those misclassification excess risk bounds developed in~\cite{chen2018best,chen2018high} for $\ell_0$-ERM based binary classification problems. In that regime, they proved an $ \mathcal{\tilde O} \left(\sqrt{k\log(p)/n}\right)$ high probability excess risk bound, and under additional margin conditions an $ \mathcal{O} \left(k\log(p)/n\right)$ in expectation bound. Comparing to those results, our bounds as summarized in Table~\ref{tab:results} are valid for a broader range of loss functions beyond binary loss and in some cases are substantially tighter (see, e.g., Theorem~\ref{thrm:generalization_barw} and Theorem~\ref{thrm:generalizaion_support_recovery_high}). More thorough comparison to the $\ell_0$-ERM generalization results in~\cite{chen2018best,chen2018high} will be carried out in the main text.

\textbf{Uniform convergence and stability of ERM.} There is a rich literature on uniform convergence bounds for the difference between the empirical risk $F_S(w)$ and the population risk $F(w)$~\cite{bartlett2006convexity,bottou2008tradeoffs,shalev2009stochastic}. Although showing to be more general (e.g., applicable to non-convex problems) and lead to tight generalization in some restricted settings~\cite{kakade2009complexity}, uniform convergence bounds tend to suffer from the polynomial dependence on dimensionality and thus are not satisfactory for high-dimensional learning algorithms. To handle this deficiency, for a class of $\ell_1$-penalized high dimensional M-estimators, uniform convergence bounds with polynomial dependence on the sparsity level of certain nominal model were established in~\cite{mei2018landscape}. Also for $\ell_0$-ERM with binary loss function, a uniform excess risk bound of order $ \mathcal{\tilde O} \left(\sqrt{k\log(p)/n}\right)$ was derived in~\cite{chen2018high} under proper regularity conditions. Alternatively, a useful proxy for analyzing the generalization performance is the \emph{stability} of learning algorithms to changes in the training dataset. Since the seminal work of Bousquet and Elisseeff~\cite{bousquet2002stability}, stability has been extensively demonstrated to beget strong generalization bounds for ERM solutions with convex loss~\cite{mukherjee2006learning,shalev2010learnability} and more recently for iterative learning algorithms (such as SGD) as well~\cite{charles2018stability,hardt2016train,kuzborskij2018data}. Specially, the state-of-the-art generalization results are offered by approaches based on the notion of uniform stability~\cite{feldman2018generalization,feldman2019high}. In light of these prior results, we aim to analyze the generalization performance of $\ell_0$-ERM with popularly used convex or non-convex loss functions based on uniform convergence and stability arguments, which to our knowledge has not been systematically treated elsewhere in literature.

\textbf{Statistical guarantees on IHT-style algorithms.} The IHT-style algorithms are popularly used and studied in compressed sensing and sparse learning~\cite{blumensath2009iterative,foucart2011hard,garg2009gradient}. Recent works have demonstrated that by imposing certain assumptions such as restricted strong convexity/smothness and restricted isometry property (RIP) over the risk function, IHT and its variants converge linearly towards certain nominal sparse model with high estimation accuracy~\cite{bahmani2013greedy,yuan2014gradient}. It was later shown in~\cite{jain2014iterative} that with proper relaxation of sparsity level, high-dimensional estimation consistency can be established for IHT without assuming RIP conditions. The sparsity recovery performance of IHT-style methods was investigated in~\cite{shen2017iteration,yuan2016exact} to understand when the algorithm can exactly recover the support of a sparse signal from its compressed measurements. The generalization performance of IHT yet still remains an open problem that we aim to address in this work.

\section{Generalization Bounds for $\ell_0$-ERM}
\label{sect:generalization_l0_erm}

In this section, we present a set of generalization gap and excess risk bounds for the $\ell_0$-ERM estimator. We distinguish our analysis in two regimes: the first is a white-box statistical setting where the data is assumed to be generated according to a truly sparse model, while the second is a black-box but more realistic setting where the data generation process is presumed unknown.

\subsection{White-box statistical analysis}
\label{ssect:whitebox}

We begin by considering an ideal setting where the underlying statistical model for generating the data samples is truly sparse. Such a statistical treatment is conventional in the theoretical analysis of high-dimensional sparsity recovery approaches~\cite{agarwal2012fast,mei2018landscape,yuan2018gradient}. More specifically, we assume that there exists a $k$-sparse parameter vector $\bar w$ such that, roughly speaking, the population risk function is minimized exactly at $\bar w$ with $\nabla F(\bar w)=0$.

\subsubsection{Preliminaries}
We impose the following assumption on the loss function which basically requires the gradient of loss at $\bar w$ obeys a light tailed distribution.

\begin{assumption}[Sub-Gaussian gradient at the true model]\label{assump:gradient_sub_gaussian}
For each $j\in\{1,...,p\}$, we assume that $\nabla_j \ell(\bar w;\xi)$ is $\sigma^2$-sub-Gaussian with zero mean, namely, $\mathbb{E}_\xi[\nabla_j \ell(\bar w;\xi)]=0$ and there exists a constant $\sigma>0$ such that for any real number $t$,
\[
\mathbb{E}_\xi\left[\exp\left\{t(\nabla_j \ell(\bar w;\xi))\right\}\right] \le \exp\left\{\frac{\sigma^2t^2}{2}\right\}.
\]
\end{assumption}
\begin{remark}
The zero-mean assumption implies $\nabla F(\bar w) = 0$. As we will show shortly that this assumption is satisfied by the widely used linear regression and logistic
regression models.
\end{remark}
Our analysis also relies on the conditions of Restricted Strong Convexity/Smoothness (RSC/RSS) which are conventionally used in the analysis of sparsity methods~\cite{bahmani2013greedy,blumensath2009iterative,jain2014iterative,yuan2018gradient}.
\begin{definition}[Restricted Strong Convexity/Smoothness]\label{def:strong_smooth}
For any sparsity level $1\le s \le p$, we say a function $f$ is restricted $\mu_s$-strongly convex and $L_s$-strongly smooth if there exist $\mu_s, L_s > 0$ such that
\[
\frac{\mu_s}{2}\|w-w'\|^2 \le f(w) - f(w') - \langle \nabla f(w'), w -w'\rangle \le \frac{L_s}{2}\|w - w'\|^2, \quad \forall\|w - w'\|_0\le s.
\]
Particularly, we say $f$ is $L$-strongly smooth if $\forall w,w'$,
\[
f(w) - f(w') - \langle \nabla f(w'), w -w'\rangle \le \frac{L}{2}\|w - w'\|^2.
\]
\end{definition}
The ratio number $L_s/\mu_s$, which measures the curvature of the loss function over sparse subspaces, will be referred to as \emph{restricted strong condition number} in this paper.

\subsubsection{Main results}

When analyzing $\ell_0$-ERM in the considered white-box setting, there are three sources of uncertainty at play: the sparse pattern of the unknown $\bar w$, the RSC/RSS conditions of the empirical risk $F_S$ and the statistical noise encoded in the gradient of loss function $\nabla \ell(\bar w;\xi)$. By simultaneously taking all these three factors into account,  we establish in the following theorem a set of generalization bounds for $\ell_0$-ERM. A proof of this result is provided in Appendix~\ref{apdsect:proof_generalization_barw}.

\begin{restatable}{theorem}{ERMWhiteBoxBounds}\label{thrm:generalization_barw}
Assume that there exists a $k$-sparse vector $\bar w$ such that Assumption~\ref{assump:gradient_sub_gaussian} holds. Suppose that $F_S$ is $\mu_{2k}$-strongly convex with probability at least $1- \delta'_n$. Assume that the loss function $\ell$ is $L$-strongly smooth with respect to its first argument and $0\le\ell(w;\xi)\le M$ for all $w,\xi$. Then for any $\delta \in (0, 1- \delta'_n)$, with probability at least $1 - \delta - \delta'_n$ the  generalization gap and excess risk are (separately) upper bounded by
\[
F(w_{S,k}) - F_S(w_{S,k}) \le \mathcal{O}\left(\frac{L}{\mu_{2k}^2} \left(\frac{k\sigma^2\log (p/\delta)}{n}\right)+ M\sqrt{\frac{\log(1/\delta)}{n}}\right),
\]
and
\[
F(w_{S,k}) - F(\bar w) \le \mathcal{O}\left(\frac{L}{\mu_{2k}^2}\left(\frac{k\sigma^2\log (p/\delta)}{n}\right)\right).
\]
Moreover, assume the domain of interest $\mathcal{W}\subset \mathbb{R}^p$ is bounded by $R$. Let $\delta_n = \sigma^2(72 + 16\log p)/n$. If $\delta'_n \le \min\left\{0.5, \frac{\delta_n}{4R^2}\right\}$, then the generalization gap and excess risk in expectation are upper bounded by
\[
\mathbb{E}_S \left[F(w_{S,k}) - F(\bar w)\right] \le \mathbb{E}_S \left[F(w_{S,k}) - F_S(w_{S,k})\right] \le \mathcal{O} \left(\frac{L}{\mu_{2k}^2} \left(\frac{k\sigma^2\log(p)}{n}\right)\right).
\]
\end{restatable}
\begin{remark}
Theorem~\ref{thrm:generalization_barw} basically reveals that with high probability, the generalization gap is bounded by $\mathcal{\tilde O} \left(k\sigma^2\log(p)/n+1/\sqrt{n}\right)$ and the excess risk by $\mathcal{\tilde O} \left(k\sigma^2\log(p)/n\right)$. In expectation, both bounds are of the order $\mathcal{\tilde O} \left(k\sigma^2\log(p)/n\right)$. We comment on the tightness of the excess risk bounds in the minimax sense. It is well known (see, e.g.,~\cite{rigollet201518}) that, up to logarithmic factors, the high probability bound $\mathcal{\tilde O} \left(k\sigma^2\log(p)/n\right)$ is minimax optimal for the squared estimation error $\|w_{S,k} - \bar w\|^2$, which immediately implies that the same bound is also minimax optimal to $F(w_{S,k}) - F(\bar w)$ provided that $F$ has restricted strong convexity.
\end{remark}

\subsubsection{Comparison to $\ell_2$-regularized ERM}
\label{sssect:comparsion_l2ERM}

We compare the generalization gap bounds of $\ell_0$-ERM in Theorem~\ref{thrm:generalization_barw} to those available for the following $\ell_2$-regularized ERM ($\ell_2$-ERM) with convex loss function:
\[
w_{S,\lambda}= \argmin_{w \in \mathcal{W}} F_S(w)+\frac{\lambda}{2}\|w\|^2,
\]
where $\lambda>0$ is the regularization strength parameter. Based on the uniform stability arguments~\cite{bousquet2002stability}, the in expectation generalization gap of $w_{S,\lambda}$ is upper bounded as
\[
\mathbb{E}_S\left[F(w_{S,\lambda}) - F_S(w_{S,\lambda})\right]\le \mathcal{O}\left(\frac{1}{\lambda n}\right).
\]
By setting $\lambda=\mathcal{O}(1/\sqrt{n})$ with balanced impact against the guarantees on excess risk, the above generalization gap bound of $w_{S,\lambda}$ is of order $\mathcal{O}(1/\sqrt{n})$, which would be slower than the corresponding $\mathcal{O}\left(\frac{kL\sigma^2\log (p/\delta)}{n\mu_{2k}^2}\right)$ bound of $w_{S,k}$ for sufficiently large $n = \Omega\left(\frac{k^2 L^2 \sigma^4\log^2(p/\delta)}{\mu^4_{2k}}\right)$. We comment that such a gain in tightness for $\ell_0$-ERM mainly attributes to the (high probability) restricted strong convexity of the empirical risk function.

To our knowledge, the best known high probability (with tail bound $\delta$) generalization gap bond for $\ell_2$-ERM was derived in~\cite[Theorem 1.1]{feldman2019high} as follows:
\[
F(w_{S,\lambda}) - F_S(w_{S,\lambda}) \le \mathcal{O}\left(\frac{\log(n)\log(n/\delta)}{\lambda n} + \sqrt{\frac{\log(1/\delta)}{n}}\right).
\]
With the choice of $\lambda=\mathcal{O}(\log(n)/\sqrt{n})$, the high probability bound of $w_{S,\lambda}$ is dominated by $\mathcal{O}(\log(n/\delta)/\sqrt{n})$. As a comparison, when $n$ is sufficiently large, the bound of $w_{S,k}$ in Theorem~\ref{thrm:generalization_barw} is dominated by $\mathcal{O}(\sqrt{\log(1/\delta)/n})$ which is comparable (up to logarithmic factors on $n,\delta$) to that of $w_{S,\lambda}$.
%For the comparison of the in expectation generalization gap bound, we can see from Theorem~\ref{thrm:generalization_barw} that the $\mathcal{O}\left(\frac{L}{\mu_{2k}^2}\left(\frac{k\sigma^2\log (p/\delta)}{n}\right)\right)$ bound of $w_{S,k}$ is superior to that of $w_{S,\lambda}$ for sufficiently large $n \ge \Omega\left(\frac{L k^2}{\mu^2_{2k}}\right) $. As for the high probability bound, the bound of $w_{S,k}$ in Theorem~\ref{thrm:generalization_barw} is dominated by $\mathcal{O}(\sqrt{\log(1/\delta)/n})$ when $n$ is sufficiently large, which is comparable to that of $w_{S,\lambda}$.

Concerning the comparison in excess risk, we remark that the $\mathcal{O}\left(\frac{k\sigma^2\log (p/\delta)}{n}\right)$ bounds of $L_0$-ERM established in Theorem~\ref{thrm:generalization_barw} are not directly comparable to those $\mathcal{O}(1/\sqrt{n})$ ones of $L_2$-ERM (see, e.g.,~\cite[Corollary 4.2]{feldman2019high}) as the former is derived for sparsity-constrained minimization while the latter is for unconstrained minimization. Nevertheless, we can still see that $w_{S,k}$ achieves tighter excess risk bounds than $w_{S,\lambda}$ does provided that $n$ dominates $k^2$ as discussed in the previous comparison.

\subsubsection{Examples}

We further show how to apply the bounds in Theorem~\ref{thrm:generalization_barw} to the widely used sparse linear regression and logistic regression models.

\textbf{Example I: Sparse linear regression.} We assume the samples $S=\{x_i, y_i\}$ obey the linear model $y_i = \bar w^\top x_i + \varepsilon_i$ where $\bar w$ is a $k$-sparse parameter vector, $x_i$ are drawn i.i.d. from a zero-mean sub-Gaussian distribution with covariance matrix $\Sigma\succ 0$, and $\varepsilon_i$ are $n$ i.i.d. zero-mean sub-Gaussian random variables with parameter $\sigma^2$. The sparsity-constrained least squares regression model is then written by $\min_{\|w\|_0\le k} F_S(w)= \frac{1}{2n}\sum_{i=1}^n\|y_i - w^\top x_i\|^2$. Based on the result in~\cite[Lemma 6]{agarwal2012fast} (as restated in Lemma~\ref{lemma:strong_convexity_subgaussian} in Appendix~\ref{apdsect:proof_corollary_linear}) we can verify that there exists $c_0>0$ such that with probability at least $1- \exp\{-c_0 n\}$, $F_S$ is $\mu_{2k}$-strongly convex with $\mu_{2k}= \frac{1}{2}\lambda_{\min}(\Sigma) - k c_1\log (p)/n$. Provided that $n \ge 4kc_1 \log (p)/\lambda_{\min}(\Sigma)$, we have $\mu_{2k}\ge \frac{1}{4}\lambda_{\min}(\Sigma)$ holds with probability at least $1- \exp\{-c_0 n\}$. Now we are ready to present the following corollary as an application of Theorem~\ref{thrm:generalization_barw} to the linear regression with sub-Gaussian noise and bounded design. See Appendix~\ref{apdsect:proof_corollary_linear} for its proof.

\begin{restatable}{corollary}{ERMWhiteBoxBoundsLinaer}\label{corol:generalization_barw_linearreg}
Assume that $\varepsilon_i$ are i.i.d. zero-mean $\sigma^2$-sub-Gaussian and $x_i$ are i.i.d. zero-mean sub-Gaussian distribution with covariance matrix $\Sigma\succ 0$ and $\Sigma_{jj}\le 1$. Let $\delta_n = \sigma^2(72 + 16\log p)/n$. Then there exist universal constants $c_0, c_1>0$ such that when $n \ge 4kc_1 \log (p)/\lambda_{\min}(\Sigma)$, for any $\delta \in (0, 1- \exp\{-c_0 n\})$, with probability at least $1 - \delta - \exp\{-c_0 n\}$ the excess risk can be bounded as
\[
F(w_{S,k}) - F(\bar w) \le \mathcal{O}\left(\frac{k\sigma^2\log (p/\delta)}{n\lambda^2_{\min}(\Sigma)}\right).
\]
Moreover, suppose that the domain of interest in bounded by $R$ and $\|x_i\|\le 1$. If $n$ is sufficiently large such that $\exp\{-c_0 n\}\le \min\left\{0.5, \frac{\delta_n}{4R^2}\right\}$, then the expected generalization gap and excess risk are bounded by
\[
\begin{aligned}
\mathbb{E}_S \left[F(w_{S,k}) - F(\bar w)\right]\le \mathbb{E}_S \left[F(w_{S,k}) - F_S(w_{S,k}) \right] \le \mathcal{O} \left(\frac{k\sigma^2\log(p)}{n\lambda^2_{\min}(\Sigma)}\right).
\end{aligned}
\]
\end{restatable}
\begin{remark}
This condition $\exp\{-c_0 n\}\le \min\left\{0.5, \frac{\delta_n}{4R^2}\right\}$ is always satisfiable for sufficiently large $n$ because the left hand side approaches to zero exponentially with respect to $n$ while the right hand side sub-linearly for $\delta_n$.
\end{remark}

\textbf{Example II: Sparse Logistic Regression.} Let us further consider the binary logistic regression model in which the relation between the random feature vector $x \in \mathbb{R}^p$ and its associated random binary label $y \in \{-1,+1\}$ is determined by the conditional probability $\mathbb{P}(y|x; \bar w) = \frac{\exp (2y\bar w^\top x)}{1+\exp (2y\bar w^\top x)}$, where $\bar w$ is a $k$-sparse parameter vector. Given a set of $n$ independently drawn data samples $\{(x_i,y_i)\}_{i=1}^n$, sparse logistic regression learns the parameters so as to minimize the logistic loss function under sparsity constraint: $\min_{\|w\|_0\le k} F_S(w)= \frac{1}{n}\sum_{i=1}^n \log(1+\exp(-2y_i w^\top x_i))$.
Let $X=[x_1,...,x_n]\in\mathbb{R}^{d\times n}$ be the design matrix and $s(z)=1/(1+\exp(-z))$ be the sigmoid function. It can be shown that $\nabla F_S(w) = X a(w)/n$ in which the vector $a(w)\in \mathbb{R}^n$ is given by $[a(w)]_i=-2y_i(1-\sigma(2y_iw^\top x_i))$, and the Hessian $\nabla^2 F_S(w) = X \Lambda(w) X^\top / n$ where $\Lambda(w)$ is an $n\times n$ diagonal matrix whose diagonal entries are $[\Lambda(w)]_{ii}=4s(2y_iw^\top x_i)(1-s(2y_iw^\top x_i))$. Then we have the following corollary as an application of Theorem~\ref{thrm:generalization_barw} to $\ell_0$-constrained logistic regression.

\begin{restatable}{corollary}{ERMWhiteBoxBoundsLogistic}\label{corol:generalization_barw_logisticreg}
Assume that $x_i$ are i.i.d. zero-mean sub-Gaussian distribution with covariance matrix $\Sigma\succ 0$ and $\Sigma_{jj}\le \sigma^2/32$. Suppose that $\|x_i\|\le 1$ for all $i$ and $\mathcal{W}\subset \mathbb{R}^p$ is bounded by $R$. Let $\delta_n = \sigma^2(72 + 16\log p)/n$. Then there exist universal constants $c_0, c_1>0$ such that when $n \ge 4kc_1 \log (p)/\lambda_{\min}(\Sigma)$, for any $\delta \in (0, 1- \exp\{-c_0 n\})$, with probability at least $1 - \delta - \exp\{-c_0 n\}$ the excess risk can be bounded as
\[
F(w_{S,k}) - F(\bar w) \le \mathcal{O} \left(\frac{k\exp(R)\sigma^2\log(p/\delta)}{n\lambda^2_{\min}(\Sigma)}\right).
\]
Moreover, if $n$ is sufficiently large such that $\exp\{-c_0 n\}\le \min\left\{0.5, \frac{\delta_n}{4R^2}\right\}$, then the generalization gap and excess risk are bounded in expectation by
\[
\begin{aligned}
\mathbb{E}_S \left[F(w_{S,k}) - F(\bar w)\right] \le \mathbb{E}_S \left[F(w_{S,k}) - F_S(w_{S,k})\right] \le \mathcal{O} \left(\frac{k\exp(R)\sigma^2\log(p)}{n\lambda^2_{\min}(\Sigma)}\right).
\end{aligned}
\]
\end{restatable}
A proof of Corollary~\ref{corol:generalization_barw_logisticreg} is given in Appendix~\ref{apdsect:proof_corollary_logistic}. The excess risk bound in this corollary matches the minimax optimal excess risk bound derived in~\cite{abramovich2018high} for model-size-penalized logistic regression.

\subsection{Black-box analysis: uniform convergence and stability}
\label{ssect:blackbox}
We now turn to analyze the generalization ability of $\ell_0$-ERM without explicitly accessing the underlying statistical generative models. Such a black-box setting is of special interest for understanding the model generalization behavior in a broader context of high-dimensional M-estimation beyond sparsity recovery. We study two types of bounds separately in this regime: the uniform convergence bounds which are simpler and more general, and the algorithm stability implied bounds which are tighter but at the price of imposing more stringent regularization conditions. Recollect that $\bar w= \argmin_{\|w\|_0\le k} F(w)$ denotes the population sparse minimizer.

\subsubsection{A uniform convergence bound}

Under mild conditions on the loss function and the domain of interest, it has been known that the uniform convergence bound $\sup_{w\in \mathcal{W}}\left|F(w) - F_S(w)\right| \le \mathcal{\tilde O}\left(\sqrt{p/n}\right)$ holds with high probability~\cite{bottou2008tradeoffs,shalev2009stochastic}. In the following result, we extend such a uniform convergence result to the $k$-sparse subspace. See Appendix~\ref{apdsect:proof_universe_generalizaion} for a proof of this result.

\begin{restatable}{theorem}{ERMBlackBoxBoundsUniform}\label{thrm:universe_generalizaion}
Assume that the domain of interest $\mathcal{W}\subset \mathbb{R}^p$ is bounded by $R$ and the loss function $\ell$ is $G$-Lipschitz continuous with respect to its first argument. Then for any $\delta\in (0,1)$, there exists a universal constant $c_0>0$ such that with probability at least $1-\delta$ over the random draw of $S$,
\[
\sup_{ w\in\mathcal{W}, \|w\|_0\le k}\left|F(w)  - F_S(w)\right| \le \mathcal{O}\left(\sqrt{\frac{GR(k\log(p)+ \log(1/\delta))}{n}}\right),
\]
provided that
\[
\frac{4G(\log(c_0/\delta)+10k\log(k))}{R} \le n \le \frac{k^{10}(\log(c_0/\delta) + 10k\log(k))}{GR}.
\]
\end{restatable}

Comparing to the white-box results in Theorem~\ref{thrm:generalization_barw}, on one hand the uniform bounds established here are more general in the sense that they are applicable to both convex and non-convex problems without imposing any strong statistical conditions on the underlying nominal model. On the other hand, the $\mathcal{\tilde O}\left( \sqrt{k\log(p)/n} \right)$ uniform convergence rate is substantially slower than those white-box ones at the rate of $\mathcal{\tilde O}\left( k\log(p)/n \right)$.

We comment on the difference between the bound in Theorem~\ref{thrm:generalization_barw} and the essentially $ \mathcal{\tilde O} \left(\sqrt{k\log(p)/n}\right)$ uniform convergence bound established in~\cite[Theorem 1]{chen2018best} for sparsity constrained binary classification problems. First, our bound holds for real-valued Lipschitz continuous loss functions while that bound was tailored for binary loss functions based on Talagrand's concentration inequality~\cite{talagrand1994sharper}. Second, regarding the regularization condition, one sufficient condition to warrant the result in~\cite[Theorem 1]{chen2018best} is $p\vee n=\Omega\left( k^8\right)$ which could be fairly unrealistic even when $k$ is moderate (say, $k=10^3$) in practical problems. In contrast, our analysis basically requires that the sample size should be neither ``too small'' (i.e., $\Omega(k\log (k))$) nor ``too big'' (i.e., $\mathcal{O}(k^{10})$) with respect to sparsity level, which is expected to be more realistic in applications.

The derivation of the uniform  bounds does not hinge the optimality of the $\ell_0$-ERM estimator $w_{S,k}$ for risk minimization. We next show how to improve upon the above uniform convergence bounds via uniform stability arguments which can explicitly benefit from the optimality of the estimator when the loss function is convex.

\subsubsection{Uniform stability bounds}

Uniform stability, as formally defined in below, is a powerful toolbox for analyzing generalization bounds of M-estimators and their learning algorithms as well~\cite{bousquet2002stability,feldman2019high,hardt2016train,shalev2010learnability}.
\begin{definition}[Uniform Stability]\label{def:uniform_stability}
Let $A: \mathcal{X}^n \mapsto \mathcal{W}$ be a learning algorithm that maps a data set $S \in \mathcal{X}^n$ to a model $A(S) \in \mathcal{W}$. $A$ is said to have uniformly stability $\gamma$ with respect to a loss function $\ell: \mathcal{W} \times \mathcal{X}\mapsto \mathbb{R}$ if for any pair of datasets $S, S' \in \mathcal{X}^n$ that differ in a single element and every $x \in \mathcal{X}$, $\left|\ell(A(S); x) - \ell(A(S');x)\right| \le  \gamma$.
\end{definition}
For an instance, conventional ERM estimators with $\lambda$-strongly convex loss functions have uniform stability of order $\mathcal{O}(1/(\lambda n))$~\cite{bousquet2002stability}. This fundamental result then gives rise to the $\ell_2$-norm regularized ERM which introduces a penalty term $\frac{\lambda}{2}\|w\|^2$ to the convex loss with optimal choice $\lambda=\mathcal{O}(1/\sqrt{n})$ to balance empirical loss and generalization gap~\cite{shalev2009stochastic,feldman2018generalization,feldman2019high}.

\textbf{Challenge.} The existing stability analysis of ERM, however, seems not readily extendable to $\ell_0$-ERM as considered in this work. The reason behind this challenge is mainly because the stability of $\ell_0$-ERM relies heavily on the stability of its recovered supporting set $\supp(w_{S,k})$ which could be highly non-trivial to guarantee even when the loss function is strongly convex. In contrast, the convectional dense ERM is supported over the entire range of feature dimension and thus its supporting set is by nature unique and stable.

\textbf{Our ideas.} We provide two solutions to address the above sparsity pattern stability issue of $\ell_0$-ERM. The basic idea of our first solution is similar to that of the previous uniform convergence analysis: we first analyze the uniform stability of ERM restricted over any fixed feature index set of cardinality $k$, and then establish a high probability generalization gap bound for $\ell_0$-ERM via applying union probability arguments to all the possible $k$-sparse supporting sets. Although fairly simple in principle, one technical obstacle we need to overcome is that in many statistical learning problems the restricted strong convexity of loss function usually holds with high probability over data sample rather than uniformly. In order to handle the small failure probability of strong convexity, we propose to analyze a regularized variant of $\ell_0$-ERM with an unknown penalty term $\frac{\lambda}{2}\|w\|^2$ added to guarantee restricted uniform stability, and consequently show that the stability-induced generalization bound of the regularized estimator can be inherited by $\ell_0$-ERM with high chance.

Our second attempt is to directly analyze the stability of $\ell_0$-ERM with respect to its recovered supporting set. The key component is to show that under certain additional strong-signal condition, the support recovery of $\ell_0$-ERM is stable with high probability. More concretely, we can show that if there exists an underlying $k$-sparse vector $\tilde w$ with sufficiently strong signal-noise-ratio, then $\supp(w_{S,k})=\supp(w_{S'})=\supp(\tilde w)$ holds with high probability for any pair of datasets $S, S'$ that differ in a single element. Based on this key observation, we can further prove a set of $\mathcal{O}(\log^2(n)/\sqrt{n})$ (with high probability) and $\mathcal{O}(1/n)$ (in expectation) generalization bounds for $\ell_0$-ERM which are known to be nearly optimal up to logarithmic factors even for the conventional dense ERM.

\textbf{Main results.} In the following proposition, we establish a set of generalization bounds of $\ell_0$-ERM induced by the uniform stability of ERM restricted over any feature set of cardinality $k$.

\begin{restatable}{proposition}{ERMBlackBoxBoundsStability}\label{thrm:uniform_stability}
Assume that the loss function $\ell$ is $G$-Lipschitz continuous with respect to its first argument and $0\le\ell(w;\xi)\le M$ for all $w,\xi$. Suppose that $F_S$ is $\mu_{k}$-strongly convex with probability at least $1- \delta'_n$ over the random draw of $S$. For any $\delta\in (0,1)$, if $\delta'_n \le \frac{\delta}{2}\left(\frac{k}{ep}\right)^k$, then with probability at least $1-\delta$ over the random draw of sample set $S$ the generalization gap $F(w_{S,k}) - F_S(w_{S,k})$ is upper bounded by
\[
\mathcal{O}\left(\frac{G^{3/2}M^{1/4}}{\mu_k^{3/4}}\sqrt{\frac{\log(n)(\log(n/\delta)+k\log(ep/k))}{n}}\right),
\]
and (separately) the excess risk $F(w_{S,k})- F(\bar w)$ is upper bounded by
\[
\mathcal{O}\left(\frac{G^{3/2}M^{1/4}}{\mu_k^{3/4}}\sqrt{\frac{\log(n)(\log(n/\delta)+k\log(ep/k))}{n}} + M\sqrt{\frac{\log(1/\delta)}{n}}\right).
\]
\end{restatable}
\begin{proof}[Proof sketch]
For a given feature index set $J\subseteq \{1,...,p\}$ with $|J|=k$, we first establish the generalization bounds for the restrictive estimator over $J$ defined by $w_{S\mid J} = \argmin_{\supp(w)\subseteq J} F_S(w)$. Since $F_S$ is only assumed to have strong convexity over $J$ with high probability, we propose to alternatively study an $\ell_2$-regularized variant of $w_{S\mid J}$ defined by
\[
w_{\lambda,S\mid J} := \argmin_{\supp(w)\subseteq J} \left\{F_{\lambda,S}(w):= F_S(w) + \frac{\lambda}{2}\|w\|^2\right\},
\]
which can be easily shown to be  uniformly stable, and thus according to the result from~\cite[Theorem 1.1]{feldman2019high} its generalization gap is upper bounded by $\mathcal{O}\left(G^2\log(n)\log(n/\delta)/(\lambda n) + \sqrt{\log(1/\delta)/n}\right)$. The next key step is to bound with high probability the discrepancy between $w_{S\mid J}$ and $w_{\lambda, S\mid J}$ as $\|w_{S\mid J} - w_{\lambda, S\mid J}\| \le \mathcal{O}(\lambda /(\mu_k + \lambda))$ in light of the (high probability) restricted strong convexity of $F_S$, which consequently indicate that the generalization guarantee of $w_{\lambda, S\mid J}$ can be passed over to $w_{S\mid J}$ with a small overhead of $\mathcal{O}(\lambda/(\mu_k + \lambda))$. The final step is to apply union probability arguments over all the possible $J$, of which the size is no more than $(ep/k)^k$, to obtain the desired bounds under proper selection of $\lambda$. Note that the $\ell_2$-regularized estimator $w_{\lambda, S\mid J}$ is purely introduced as a hypothetical tool for analysis and it is not involved in the actual computation of $\ell_0$-ERM. A full proof of this result is provided in Appendix~\ref{apdsect:proof_uniform_stability}.
\end{proof}
\begin{remark}
For restricted strongly convex problems, the bounds in Proposition~\ref{thrm:uniform_stability} are generally comparable to the corresponding uniform convergence bounds in Theorem~\ref{thrm:universe_generalizaion} but without having to require the domain of interest and the sample features to be bounded. It is noteworthy that for linear prediction models, the bound in Proposition~\ref{thrm:uniform_stability} can also be established in view of the covering number arguments in~\cite{kakade2009complexity}.
\end{remark}

Let us substantialize Proposition~\ref{thrm:uniform_stability} in the context of sparse logistic regression as discussed in Example II. Assume that $x_i$ are i.i.d. zero-mean sub-Gaussian distribution with covariance matrix $\Sigma\succ 0$ and $\Sigma_{jj}\le \sigma^2/32$. Suppose that $\|x_i\|\le 1$ for all $i$ and $\mathcal{W}\subset \mathbb{R}^p$ is bounded by $R$. Then it can be verified that the logistic loss $\ell(w;\xi_i) = \log(1+\exp(-2y_i w^\top x_i))$ satisfies $\ell(w;\xi_i)= \mathcal{O}(R)$ and it is $\mathcal{O}(1)$-Lipschitz continuous. Based on the proof arguments of Corollary~\ref{corol:generalization_barw_logisticreg} we know that there exist universal constants $c_0, c_1$ such that if $n \ge \frac{4c_1\sigma^2k \log (p)}{\lambda_{\min}(\Sigma)}$, then with probability at least $1- \exp\{-c_0 n\}$, $F_S(w)$ is $\mu_k$-strongly convex with $\mu_k\ge \frac{\lambda_{\min}(\Sigma)}{(1+\exp(2R))^2}$. It holds that $\delta'_n=\exp\{-c_0 n\}\le \frac{\delta}{2}\left(\frac{k}{ep}\right)^k$ as long as $n \ge \max\left\{ \frac{4kc_1 \log (p)}{\lambda_{\min}(\Sigma)}, \frac{k\log\left(ep/k\right)+ \log(2/\delta)}{c_0}\right\}$. Finally, by invoking Proposition~\ref{thrm:uniform_stability} we obtain that with probability at least $1 - \delta $ the generalization gap $\Delta_{S,k}$ and excess risk $F(w_{S,k})- F(\bar w)$ are (separately) upper bounded by
\[
\mathcal{O}\left(\frac{\exp(3R)R^{1/4}}{\lambda^{3/4}_{\min}(\Sigma)}\sqrt{\frac{\log(n)(\log(n/\delta)+k\log(p/k))}{n}} + R\sqrt{\frac{\log(1/\delta)}{n}}\right).
\]
%In the following corollary we show the implication of Proposition~\ref{thrm:uniform_stability} to the sparse logistic regression model as discussed in Example II.
%\begin{restatable}{corollary}{ERMStabilityBoundsLogistic}\label{corol:generalization_stability_logisticreg}
%Consider the logistic loss function. Then for any $\delta \in (0,1)$ there exist universal constants $c_0, c_1>0$ such that if
%\[
%n \ge \max\left\{ \frac{4kc_1 \log (p)}{\lambda_{\min}(\Sigma)}, \frac{k\log\left(ep/k\right)+ \log(1/\delta)}{c_0}\right\},
%\]
%
%\end{restatable}

Like in the uniform convergence bounds appeared in Theorem~\ref{thrm:universe_generalizaion}, the leading term $\sqrt{k\log(n)\log(ep/k)/n}$ in the above generalization bounds is resulted from the worst case uncertainty of the sparsity pattern for $\ell_0$-ERM. In the following main theorem, we further show that such an overhead can actually be removed by imposing additional strong-signal conditions to ensure the stability of support recovery of $\ell_0$-ERM. Here we denote $w_{\min}:=\min_{i \in \supp(w)} |w_i|$ as the smallest (in magnitude) non-zero entry of a sparse vector $w$.

\begin{restatable}{theorem}{ERMBlackBoxBoundsUniformSupport}\label{thrm:generalizaion_support_recovery_high}
Assume that the loss function $\ell$ is $G$-Lipschitz continuous with respect to its first argument and $0\le\ell(w;\xi)\le M$ for all $w,\xi$. Suppose that $F_S$ is $\mu_{2k}$-strongly convex with probability at least $1- \delta'_n$ over the random draw of $S$. Suppose that there exits a $k$-sparse vector $\tilde w$ satisfying
\[
\tilde w_{\min}> \frac{2\sqrt{2k}\|\nabla F(\tilde w)\|_\infty}{\mu_{2k}} + \frac{2G}{\mu_{2k}} \sqrt{\frac{2k\log(2p/\delta)}{2n}}
\]
for some $\delta\in (0, 1-\delta'_n)$. Then
\begin{itemize}[leftmargin=*]
\item[(a)]the generalization gap and excess risk in expectation are upper bounded by
\[
\mathbb{E}_{S}\left[F(w_{S,k}) - F(\bar w)\right] \le\left|\mathbb{E}_{S}\left[F(w_{S,k}) - F_S(w_{S,k}) \right] \right| \le \frac{4G^2}{\mu_{2k} n} + 2M(\delta + 2\delta'_n).
\]
\item[(b)] Moreover, for any $\lambda>0$, with probability at least $1-\delta-\delta'_n$ over the random draw of $S$, the generalization gap $\Delta_{S,k}$ and excess risk $F(w_{S,k}) - F(\bar w)$ are upper bounded (separately) by
\[
\mathcal{O}\left(\frac{G^{3/2}M^{1/4}}{\mu_{2k}^{3/4}}\sqrt{\frac{\log(n)\log(n/\delta)}{n}} + M\sqrt{\frac{\log(1/\delta)}{n}}\right).
\]
\end{itemize}
\end{restatable}
\begin{proof}[Proof sketch]
The starting point of the proof is to show via a technical lemma (Lemma~\ref{lemma:support_recovery_bar_w}) that the imposed strong-signal condition on $\bar w$ leads to $\supp(w_{S,k})=\supp(\tilde w)$ with high probability. Based on such a simple observation we can further show that the support recovery stability $\supp(w_{S,k})=\supp(w_{S',k})$ holds with high probability for any pair of datasets $S, S' \in \mathcal{X}^n$ that differ in a single element, which in turn implies the high probability loss value stability bound $|\ell(w_{S,k};\xi) - \ell(w_{S',k};\xi) |\le \mathcal{O}(G^2/(\mu_{2k} n))$. Then the in expectation generalization bounds in part(a) can be proved by applying the standard unform stability arguments with proper handling of the small failure probability of loss value stability. Finally, the desired high probability generalization bounds in part (b) can be proved by invoking the technical Lemma~\ref{lemma:support_stability} developed for Proposition~\ref{thrm:uniform_stability} along with a careful manipulation of the involved tail bounds. A full proof of this theorem is provided in Appendix~\ref{apdsect:proof_theorem_generalizaion_support_recovery_high}.
\end{proof}
\begin{remark}
Consider the sparse linear and logistic regression problems as presented in Example I \& II with the nominal model $\bar w$ satisfying $\nabla F(\bar w)=0$. In these two cases, provided that $n$ is sufficiently large, we know from Corollary~\ref{corol:generalization_barw_linearreg} and~\ref{corol:generalization_barw_logisticreg} that there exists some universal constant $c_0>0$ such that $\delta'_n=\exp\{-c_0 n\}$. Choose $\delta=\mathcal{O}\left(G^2/(M\mu_{2k}n)\right)$ and suppose that $n$ is sufficiently large such that $\delta'_n \le \mathcal{O} \left(G^2/(M\mu_{2k}n)\right)$. If we assume $\bar w_{\min}= \Omega\left(\sqrt{k\log(np)/n}\right)$, then the in expectation generalization bounds in the part (a) of Theorem~\ref{thrm:generalizaion_support_recovery_high} scale as $\mathcal{O} \left(G^2/(\mu_{2k}n)\right)$ which is not relying on the sparsity level $k$. To compare with $\ell_2$-ERM, based on the discussions in Section~\ref{sssect:comparsion_l2ERM} we can see that such a $\mathcal{O} \left(G^2/(\mu_{2k}n)\right)$ of $\ell_0$-ERM is substantially tighter than the $\mathcal{O}(1/\sqrt{n})$ in expectation bound of $\ell_2$-ERM in the considered setting of Theorem~\ref{thrm:generalizaion_support_recovery_high}. We remark that such a benefit mainly attributes to the restricted strong convexity of the empirical risk function.
\end{remark}
\begin{remark}
The high probability generalization bounds in the part (b) of Theorem~\ref{thrm:generalizaion_support_recovery_high} scales as $\mathcal{\tilde O}\left(\log(n)/\sqrt{n} \right)$. Specially, for the sparse linear and logistic regression problems considered in the previous remark, if assuming $\bar w_{\min} = \Omega (\sqrt{k\log(np)/n})$, then with high probability the generalization gap and excess risk are bounded by $\mathcal{\tilde O}(\log(n)/\sqrt{n})$, which up to logarithmic factors are tighter than those in Theorem~\ref{thrm:generalization_barw}.
%The result in Theorem~\ref{thrm:generalizaion_support_recovery_high} suggests that the higher the singal/noise ratio is, the easier the generalization.
%Consider the sparse linear and logistic regression problems as presented in Example I \& II, we know from Corollary~\ref{corol:generalization_barw_linearreg} and~\ref{corol:generalization_barw_logisticreg} that there exists some universal constant $c_0>0$ such that $\delta'_n=\exp\{-c_0 n\}$.
We remark that the term $\mathcal{\tilde O}(\log(n)/\sqrt{n})$ is nearly tight because even for an algorithm that outputs a fixed function the sampling error term $\mathcal{\tilde O}(1/\sqrt{n})$ is necessary.
\end{remark}

\section{Generalization Bounds for Iterative Hard Thresholding}
\label{sect:generalization_iht}

In this section, we demonstrate the applications of our sparse generalization theory to deriving the generalization bounds of the IHT algorithm, as summarized in~\eqref{prob:iht}, for approximately solving $\ell_0$-ERM. The rate of convergence and parameter estimation error of IHT have been extensively analyzed under RIP (or restricted strong condition number) bounding conditions~\cite{bahmani2013greedy,yuan2014gradient}. The RIP-type conditions, however, are unrealistic in many applications. To remedy this deficiency, sparsity-level relaxation strategy was considered in~\cite{jain2014iterative,yuan2018gradient} with which the high-dimensional estimation consistency of IHT can be established under arbitrary restricted strong condition number. In order to make our analysis more realistic for high-dimensional problems, we choose to work with the following RIP-condition-free convergence rate bound, which is essentially from~\cite{jain2014iterative}, for IHT invoking on the empirical risk $F_S$.
\begin{lemma}[Convergence rate of IHT~\cite{jain2014iterative}]\label{lemma:convergence_iht}
Assume that $F_S$ is $L$-strongly smooth and $\mu_{3k}$-strongly convex. Set $\eta = \frac{2}{3L}$. Consider $\bar k$ such that $k\ge \frac{32 L^2}{\mu_{3k}^2}\bar k$. Let $\bar w_{S,k} =\argmin_{\|w\|_0\le \bar k} F_S(w)$. Then  IHT outputs $w_{S,k}^{(t)}$ satisfying $F_S(w_{S,k}^{(t)}) \le F_S(\bar w_{S,k}) + \epsilon$, after $t\ge\mathcal{O}\left(\frac{L}{\mu_{3k}}\log \left(\frac{F_S(w^{(0)})}{\epsilon}\right)\right)$ steps of iteration.
\end{lemma}
In light of the established generalization bounds for $\ell_0$-ERM and Lemma~\ref{lemma:convergence_iht}, we are in the position to analyze the generalization performance of IHT. Similar to the analysis of $\ell_0$-ERM, we separately consider a white-box statistical regime where the nominal model is assumed to be truly sparse, and a black-box statistical regime where the underlying data generation model is presumed not explicitly accessible.

\subsection{White-box generalization results}

We consider the same white-box statistical regime as studied in Section~\ref{ssect:whitebox}. Denote $\Delta_{S,k}^{(t)}:=F(w^{(t)}_{S,k}) - F_S(w^{(t)}_{S,k})$. Our main results for this setting are summarized in the following corollary of Theorem~\ref{thrm:generalization_barw}. A proof of this corollary can be found in Appendix~\ref{apdsect:proof_generalization_barw_iht}.
\begin{restatable}{corollary}{IHTWhiteBoxBounds}\label{corol:generalization_barw_iht}
Assume that $\bar w$ is a $\bar k$-sparse vector satisfying Assumption~\ref{assump:gradient_sub_gaussian}. Suppose that $F_S$ is $\mu_{3k}$-strongly convex with probability at least $1- \delta'_n$ over sample $S$. Assume that the loss function $\ell$ is $L$-strongly smooth with respect to its first argument. Suppose that $k\ge \frac{32 L^2}{\mu_{3k}^2}\bar k$. Then for any $\delta \in (0, 1- \delta'_n)$ and any $\epsilon>0$, IHT invoking on $F_S(w)$ with step-size $\eta = \frac{2}{3L}$ and sufficiently large $t\ge\mathcal{O}\left(\frac{L}{\mu_{3k}}\log \left(\frac{n\mu_{3k}}{k\sigma^2\log(p/\delta)}\right)\right)$ will output $w^{(t)}_{S,k}$ such that the following excess risk bound holds with probability at least $1 - \delta - \delta'_n$ over $S$,
\[
F(w^{(t)}_{S,k}) - F(\bar w) \le \mathcal{O}\left(\frac{L}{\mu_{3k}^2}\left(\frac{k\sigma^2\log (p/\delta)}{n}\right) \right).
\]
Moreover, assume that the loss $\ell$ is $L$-strongly smooth with respect to its first argument and the domain of interest $\mathcal{W}\subset \mathbb{R}^p$ is bounded by $R$. Let $\delta_n = \sigma^2(72 + 16\log p)/n$. If $\delta'_n \le \min\left\{0.5, \frac{\delta_n}{4R^2}\right\}$, then
\[
\mathbb{E}_S \left[F(w^{(t)}_{S,k}) - F(\bar w)\right] \le \mathbb{E}_S \left[F(w^{(t)}_{S,k}) - F_S(w^{(t)}_{S,k})\right]\le  \mathcal{O}\left(\frac{L}{\mu_{3k}^2}\left(\frac{k\sigma^2\log (p/\delta)}{n}\right) \right).
\]
\end{restatable}
\begin{remark}
This result conveys a main message that the generalization bounds of IHT with sufficient iteration are controlled by those of $\ell_0$-ERM.
\end{remark}

\subsection{Black-box generalization results}

We further consider the black-box statistical setting as studied in Section~\ref{ssect:blackbox} which is of more interest in machine learning problems. It is straightforward to verify that the uniform convergence bounds in Theorem~\ref{thrm:universe_generalizaion} readily applies to the output $w^{(t)}_{S,k}$ of IHT. Next we derive a set of generalization bounds for IHT based on uniform stability arguments.
%To do so, we apply IHT to minimize the $\ell_2$-regularized objective function $F_{\lambda,S}(w)$:
%\[
%w_{\lambda,S}^{(t)} = \mathrm{H}_k \left(w_{\lambda,S}^{(t-1)} - \eta\nabla F_{\lambda,S}(w_{\lambda,S}^{(t-1)})\right).
%\]
In order to make sure that the output $w_{S,k}^{(t)}$ at the end of iteration has uniform stability, we propose to slightly modify it as $\tilde w_{S,k}^{(t)}$ which fully minimizes $F_{\lambda,S}$ over the support of $\supp(w_{S,k}^{(t)})$, i.e.,
\[
\tilde w_{S,k}^{(t)} =\argmin_{w\in \mathcal{W}} F_{S}(w), \quad \st \supp(w) =\supp(w_{S,k}^{(t)}).
\]
%The following is our main result on the high probability generalization performance of $\tilde w_{\lambda,S}^{(t)}$.
The following result is an application of Proposition~\ref{thrm:uniform_stability} to the IHT algorithm, which is proved in Appendix~\ref{apdsect:proof_uniform_stability_iht}.
\begin{restatable}{corollary}{IHTBlackBoxBoundsStability}\label{corol:uniform_stability_iht}
Assume that the loss function $\ell$ is $L$-strongly smooth and $G$-Lipschitz continuous with respect to its first argument and $0\le\ell(w;\xi)\le M$ for all $w,\xi$. Suppose that $F_S$ is $\mu_{3k}$-strongly convex with probability at least $1- \delta'_n$ over sample $S$. Let $\bar w =\argmin_{\|w\|_0\le \bar k} F(w)$ with $k\ge \frac{32 L^2}{\mu_{3k}^2}\bar k$. Set the step-size $\eta = \frac{2}{3L}$. For any $\delta \in (0, 1)$, if $\delta'_n \le \frac{\delta}{2}\left(\frac{k}{ep}\right)^k$, then with probability at least $1-\delta$ over the random draw of sample set $S$, after sufficiently large $t\ge\mathcal{O}\left(\frac{L}{\mu_{3k}}\log \left(\frac{n}{k\log(n)\log(p/k)}\right)\right)$ rounds of IHT iteration, the generalization gap $F(\tilde w^{(t)}_{S,k}) - F_S(\tilde w^{(t)}_{S,k})$ and the excess risk $F(\tilde w^{(t)}_{S,k}) - F(\bar w)$ is separately upper bounded by
\[
\mathcal{O}\left(\frac{G^{3/2}M^{1/4}}{\mu_k^{3/4}}\sqrt{\frac{\log(n)(\log(n/\delta)+k\log(p/k))}{n}} + M\sqrt{\frac{\log(1/\delta)}{n}}\right).
\]
\end{restatable}
\begin{remark}
As expected that in the considered black-box regime the generalization gap and excess risk of IHT with sufficient iteration are upper bounded by those of the $\ell_0$-ERM of rate $\mathcal{\tilde O}\left(\sqrt{k/n} \right)$.
\end{remark}
Finally, we provide a direct stability analysis of IHT aiming to improve upon the previous generalization gap bounds by removing their dependency on sparsity level $k$. Note that Theorem~\ref{thrm:generalizaion_support_recovery_high} is not directly applicable to handle this case mainly because the stability of $\ell_0$-ERM does not imply the stability of IHT due to its involved truncation operation at each iteration. We need to tailor some new stability theory for IHT as discussed in the following analysis.

For a vector $w\in \mathbb{R}^p$, let us denote $[w]_{(j)}$ the entry of $w$ with $j$-th largest absolute value such that $|[w]_{(1)}| \ge |[w]_{(2)}| \ge...\ge |[w]_{(p)}|$. We first introduce the following concept of \emph{hard-thresholding stability} which quantifies the stability of the hard-thresholding operation.
\begin{definition}[Hard-Thresholding Stability]
For a vector $w\in \mathbb{R}^p$ and given $k\in [p]$, we say $w$ is $\varepsilon_{k}$-hard-thresholding stable for some $\varepsilon_{k}>0$ if and only if $|[w]_{(k)}| \ge |[w]_{(k+1)}| + \varepsilon_{k}$.
\end{definition}
Clearly, if $w$ is $\varepsilon_{k}$-hard-thresholding stable, then $\mathrm{H}_k(w)$ should be unique and $\supp\left(\mathrm{H}_k(w)\right)=\supp\left(\mathrm{H}_k(w+\delta_w)\right)$ where $\|\delta_w\|_\infty<\varepsilon_k/2$. That is, the larger $\varepsilon_k$ is, the stabler the hard-thresholding operation will be with respect to the preserved top $k$ supporting set. Next, we introduce the concept of \emph{iterative-hard-thresholding stability} which basically characterizes the stability of the IHT algorithm when applied to a (deterministic) function.
\begin{definition}[Iterative-Hard-Thresholding Stability]\label{def:iht_stability}
For a given differentiable function $F$, $k \in [p]$, $T\in \mathbb{Z}^+$ and $k$-sparse vector $w^{(0)}\in \mathbb{R}^p$, let $\{w^{(t)}\}_{t=1}^T$ be the sequence generated by invoking IHT on $F$ with step-size $\eta$ and initialization $w^{(0)}$. Then we say $F$ is $(\varepsilon_{k},\eta, T, w^{(0)})$-IHT stable if $w^{(t-1)} - \eta \nabla F(w^{(t-1)})$, $\forall t\in [T]$ is $\varepsilon_{k}$-hard-thresholding stable.
\end{definition}
\begin{remark}
By definition, if $F$ is $(\varepsilon_{k},\eta, T, w^{(0)})$-IHT stable, then for each $t \in [T]$, $w^{(t)}=\mathrm{H}_k\left(w^{(t-1)}-\eta \nabla F(w^{(t-1)})\right)$ is unique. That is, the $k$-sparse solution sequence $\{w^{(t)}\}_{t=1}^T$ of IHT is unique.
\end{remark}

The following theorem is our main result on the generalization performance of IHT for $\ell_0$-ERM given that the population risk function $F$ has IHT stability up to the desired number of iteration.
\begin{restatable}{theorem}{IHTBlackBoxBoundsStabilityStrong}\label{thrm:uniform_stability_strong_iht}
Assume that the domain of interest $\mathcal{W}\subset \mathbb{R}^p$ is bounded by $R$, the loss function $\ell$ is $L$-strongly smooth and $G$-Lipschitz continuous with respect to its first argument and $0\le\ell(w;\xi)\le M$ for all $w,\xi$. Suppose that $F_S$ is $\mu_{4k}$-strongly convex with probability at least $1- \delta'_n$ over sample $S$. Consider running $T$ steps of IHT iteration over $F_S$ with step-size $\eta= \frac{2}{3L}$ and a fixed $k$-sparse initialization $w^{(0)}$. Assume that the population risk function $F$ is $(\varepsilon_{k},\eta, T, w^{(0)})$-IHT stable. For any $\delta \in (0, 1-\delta'_n)$, if $n\ge \frac{2G^2(L+\mu_{4k})^2\log(pT/\delta)}{L^2\mu^2_{4k}\varepsilon^2_k}$, then with probability at least $1-\delta-\delta'_n$ over the random draw of sample set $S$, the generalization gap is upper bounded by
\[
F(\tilde w^{(T)}_S) - F_S(\tilde w^{(T)}_S) \le \mathcal{O}\left(\frac{G^{3/2}M^{1/4}}{\mu_{4k}^{3/4}}\sqrt{\frac{\log(n)\log(n/\delta)}{n}} \right).
\]
Moreover, if $T\ge\mathcal{O}\left(\frac{L}{\mu_{4k}}\log \left(\frac{n}{\log(1/\delta)}\right)\right)$, then the excess risk with respect to $\bar w =\argmin_{\|w\|_0\le \bar k} F(w)$ with $k\ge \frac{32 L^2}{\mu_{4k}^2}\bar k$ is separately upper bounded by
\[
F(\tilde w^{(T)}_S) - F(\bar w) \le \mathcal{O}\left(\frac{G^{3/2}M^{1/4}}{\mu_{4k}^{3/4}}\sqrt{\frac{\log(n)\log(n/\delta)}{n}} + M\sqrt{\frac{\log(1/\delta)}{n}} \right).
\]
\end{restatable}
\begin{proof}[Proof Sketch]
The key proof idea is to construct an \emph{oracle sequence} $\{w^{(t)}\}_{t=1}^T$ generated by applying $T$ rounds of IHT iteration to (unknown) $F$ with the considered initialization $w^{(0)}$ and step-size $\eta$. Given that $F$ is $(\varepsilon_{k},\eta, T, w^{(0)})$-IHT stable, we can show in Lemma~\ref{lemma:emp_iht_stability} that the actual sequence $\{w_{S,k}^{(t)}\}_{t=1}^T$ generated by IHT invoked to the empirical risk $F_S$ (with any fixed $S$) satisfies $\supp\left(w^{(t)}_{S,k}\right) = \supp\left(w^{(t)}\right), \forall t\in [T]$ provided that $n$ is sufficiently large as assumed. Particularly, we have $\supp\left(w^{(T)}_S\right) = \supp\left(w^{(T)}\right)$ which is a fixed deterministic index set of size $k$. Then using Lemma~\ref{lemma:support_stability} yields the desired high probability generalization gap bound for $\tilde w_{S,k}^{(T)}$. The  excess risk bound can be proved in light of the just established generalization gap bound and the convergence result in Lemma~\ref{lemma:convergence_iht}. A full proof of this theorem is provided in Appendix~\ref{apdsect:proof_uniform_stability_strong_iht}.
\end{proof}
\begin{remark}
The $\mathcal{O}(\log(n)/\sqrt{n})$ bounds established in Theorem~\ref{thrm:uniform_stability_strong_iht} are not relying on sparsity level $k$ and in this sense they are substantially tighter than those in Corollary~\ref{corol:uniform_stability_iht}, yet under certain more stringent conditions on the IHT stability of the population risk $F$.
\end{remark}

\section{Numerical Experiments}
\label{sect:experiment}
In this section, we carry out numerical experiments on synthetic sparse linear regression and logistic regression tasks to verify the IHT generalization theory as presented in Section~\ref{sect:generalization_iht}, which is mostly implied by the theory established in Section~\ref{sect:generalization_l0_erm} for $\ell_0$-ERM. Throughout our numerical study, we initialize $w^{(0)} = 0$ for IHT and replicate each individual experiment $10$ times over the random generation of training data for generalization performance evaluation.

\subsection{Sparse linear regression}
\label{ssect:experiment_sparse_linear_regression}

We first consider the sparse linear regression model with quadratic loss function $\ell(w; x_i, y_i)=\frac{1}{2}(y_i-w^{\top}x_i)^2$. The feature points $\{x_i\}_{i=1}^n$ are sampled from standard multivariate Gaussian distribution. Given a model parameter $\bar w \in \mathbb{R}^p$, the responses $\{y_i\}_{i=1}^n$ are generated according to a linear model $y_i=\bar{w}^{\top}x_i+\varepsilon_i$ with a random Gaussian noise $\varepsilon_i\sim \mathcal{N}(0,\sigma^2)$. In this case, the population risk function can be expressed in close form as
\begin{equation}\label{equat:linear_regression_close_form}
F(w) = \frac{1}{2}\|w - \bar w\|^2 + \frac{\sigma^2}{2}.
\end{equation}
We fix the feature dimension $p=1000$ throughout this group of experiments.

\begin{figure}[h!]
\centering
\subfigure[Impact of samples size and sparsity level under fixed noise level $\sigma=1$.\label{fig:linear_white_box_sparsity}]{
\includegraphics[width=3in]{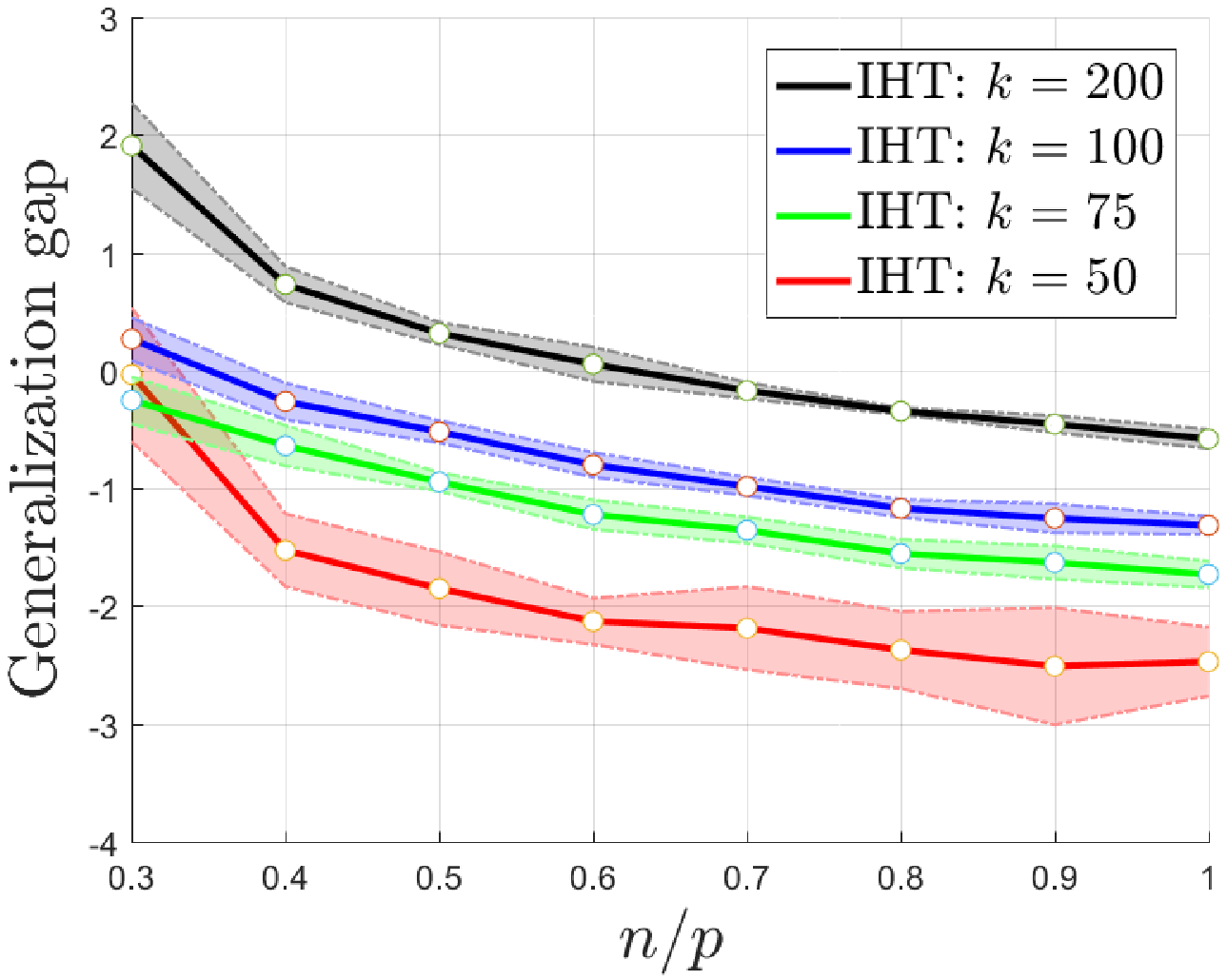}
\label{fig:generalizationgap_sparsity_white_linear}
\includegraphics[width=3in]{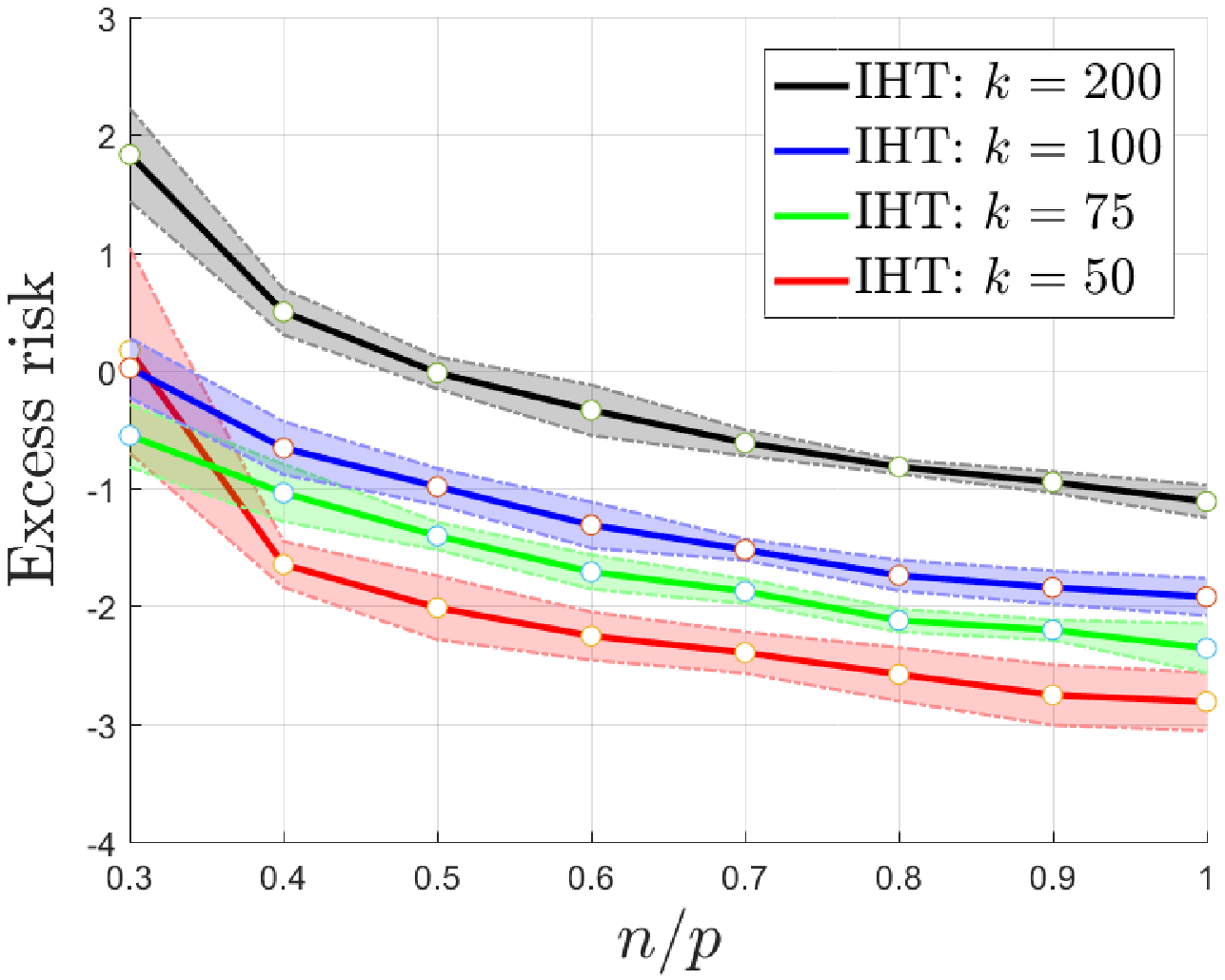}
\label{fig:excessrisk_sparsity_white_linear}
}
\subfigure[Impact of samples size and noise level under fixed sparsity level $k=100$.\label{fig:linear_white_box_noise}]{
\includegraphics[width=3in]{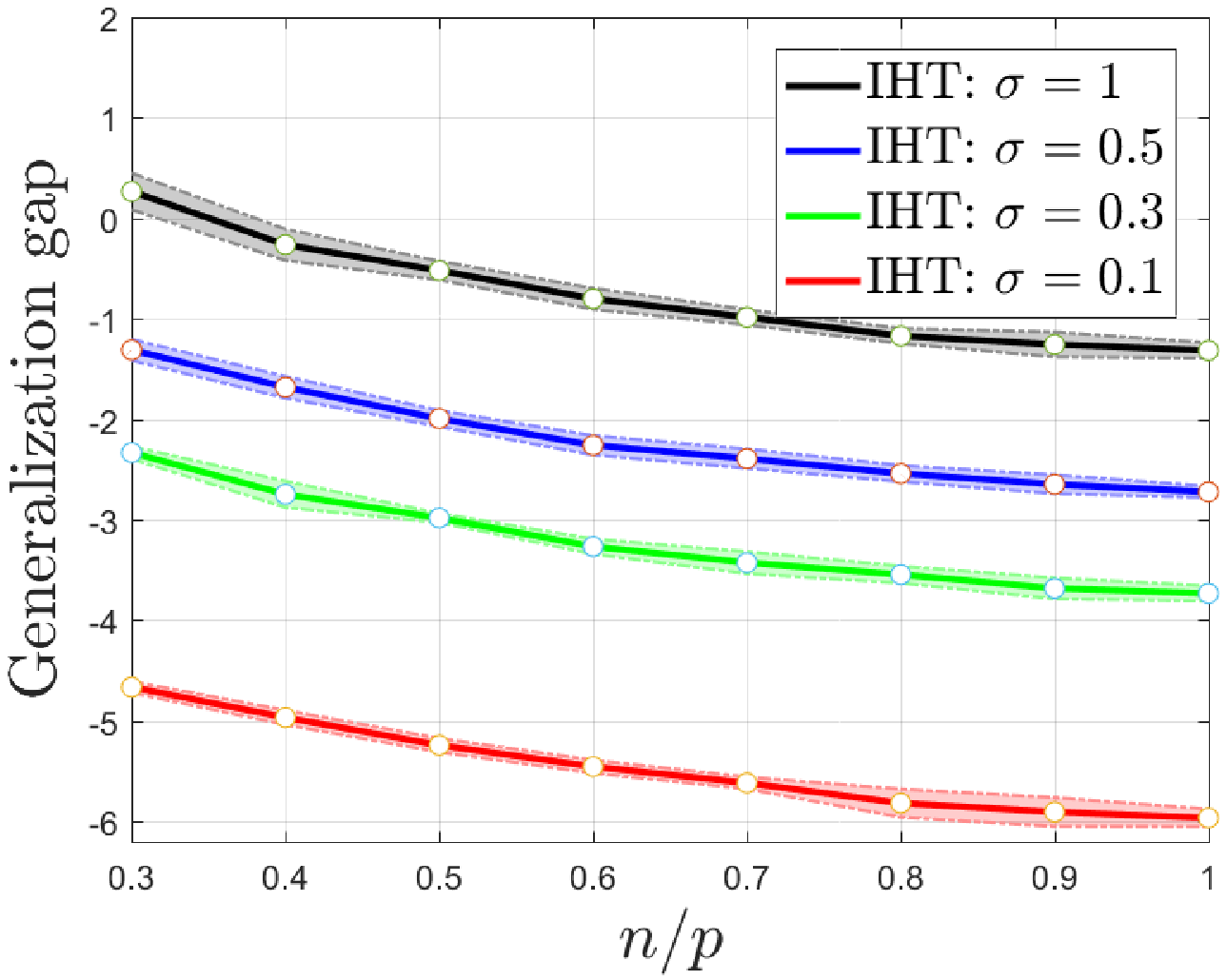}
\label{fig:generalizationgap_noise_white_linear}
\includegraphics[width=3in]{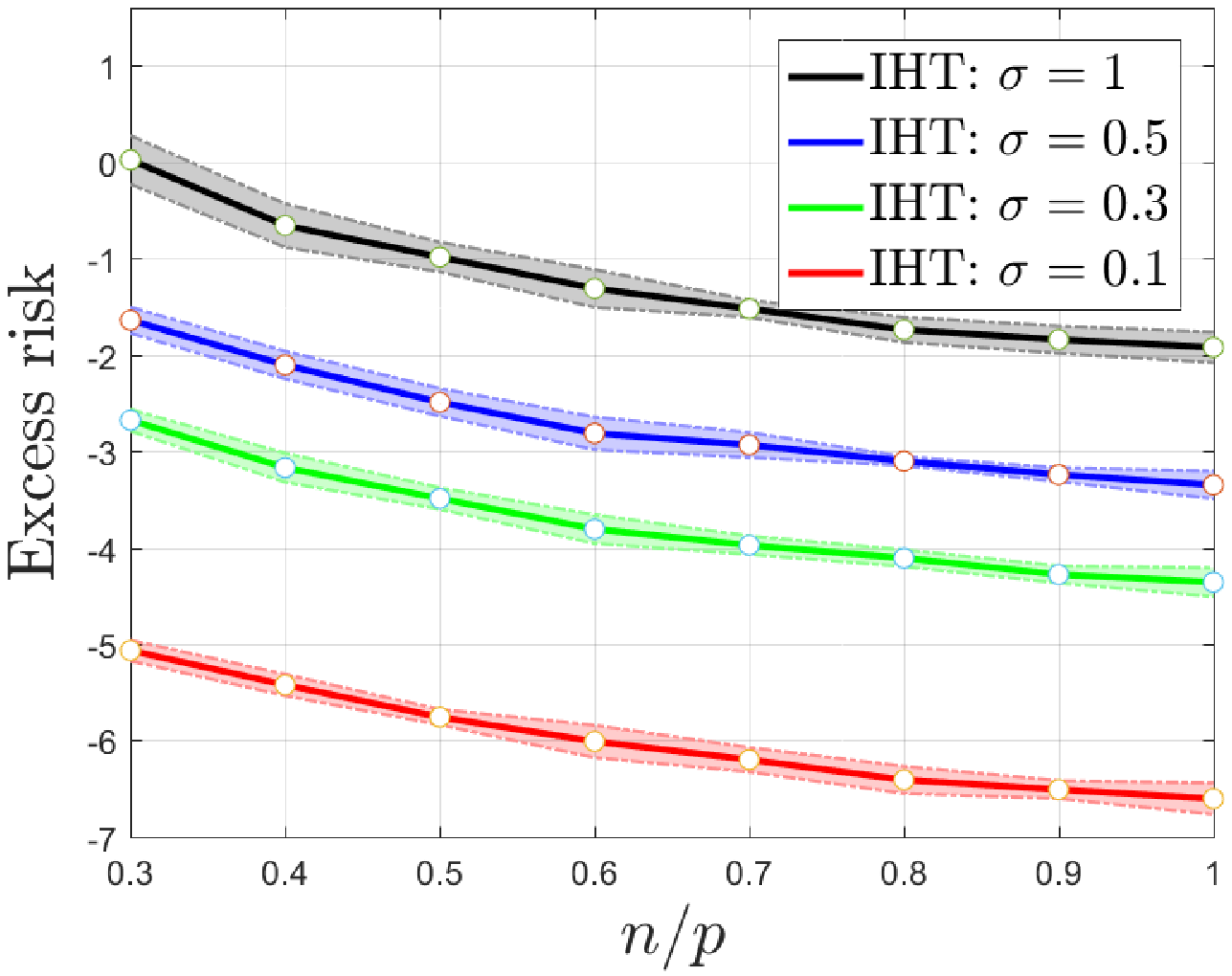}
\label{fig:excessrisk_noise_white_linear}
}
\caption{Generalization results of sparse linear regression with a white-box sparse model.}
\label{fig:linear_white_box}
\end{figure}

\newpage

\textbf{Setup and results in the white-box regime.} In this case, we set the true parameter vector $\bar{w}$ to be a $\bar k$-sparse vector with $\bar k =50$ and its non-zero entries are sampled from standard Gaussian distribution.  Corollary~\ref{corol:generalization_barw_iht} suggests that the generalization gap and excess risk bounds of IHT are controlled by the quantity $k\sigma^2\log p/n$. In order to verify this theory, we consider the following two experimental setups:
\begin{itemize}
  \item We fix $\sigma=1$ and study the impact of varying $n/p\in (0.3,1)$ and $k\in \{50, 75, 100, 200\}$ on the actual generalization performance. According to the close-form expression in~\eqref{equat:linear_regression_close_form}, we must have $\min_{\|w\|_0\le k} F(w)=\sigma^2/2$ for any $k \ge \bar k$ and thus the excess risk at any $k$-sparse $w$ can be exactly computed as $\frac{1}{2}\|w-\bar w\|^2$. Figure~\ref{fig:linear_white_box_sparsity} shows the evolving curves (error bar shaded in color) of generalization gap and excess risk as functions of sample size in the considered setting. From this set of curves we can make the following two observations: 1) for each fixed $k$, the generalization gap and excess risk of IHT decrease as $n$ increases, and 2) for each fixed $n$, these measurements grow larger as $k$ increases. These numerical evidences are consistent with the implication of Corollary~\ref{corol:generalization_barw_iht} in linear regression models.
  \item We fix $k=100$ and study the impact of varying $n/p\in (0.3,1)$ and $\sigma \in \{0.1,0.3,0.5,1\}$ on the actual generalization performance. The corresponding evolving curves of generalization gap and excess risk are shown in Figure~\ref{fig:linear_white_box_noise}. From this group of results we can see that for each fixed $n$, smaller noise level $\sigma$ leads to lower  generalization gap and excess risk, which again confirms the theoretical prediction of Corollary~\ref{corol:generalization_barw_iht}.
\end{itemize}

\textbf{Setup and results in the black-box regime.} We next verify the black-box generalization bounds in Corollary~\ref{corol:uniform_stability_iht} which are of the order $\mathcal{O}(\sqrt{k\log p/n})$, without assuming the nominal model to be sparse. To this end, we consider $\bar w = \bar w' + \varepsilon'$ where $\bar w'$ is a $\bar k$-sparse standard Gaussian vector with $\bar k =50$ and $\varepsilon'$ is a zero-mean Gaussian noise vector with sufficiently small variance such that $\bar w$ is dense but nearly sparse. We study the impact of varying $n/p\in (0.3,1)$ and $k\in \{50, 75, 100, 200\}$ on the actual generalization performance. Based on the close-form expression in~\eqref{equat:linear_regression_close_form}, the excess risk at any $k$-sparse $w$ can be evaluated as $\frac{1}{2}(\|w-\bar w\|^2 -\|\mathrm{H}_k(\bar w) - \bar w\|^2)$, keeping in mind that the optimal objective value is $\min_{\|w\|_0\le k} F(w)=\frac{1}{2}\|\mathrm{H}_k(\bar w) - \bar w\|^2 + \sigma^2/2$ for any $k \ge \bar k$. Figure~\ref{fig:linear_black_box} shows the evolving curves of generalization gap and excess risk as functions of sample size under varying sparsity level. These curves affirmatively confirm the theoretical bounds in Corollary~\ref{corol:uniform_stability_iht} which suggest that smaller generalization gap and excess risk of IHT can be attained at relatively larger $n$ and smaller $k$.
\begin{figure}[h]
\centering
\subfigure{
\includegraphics[width=3in]{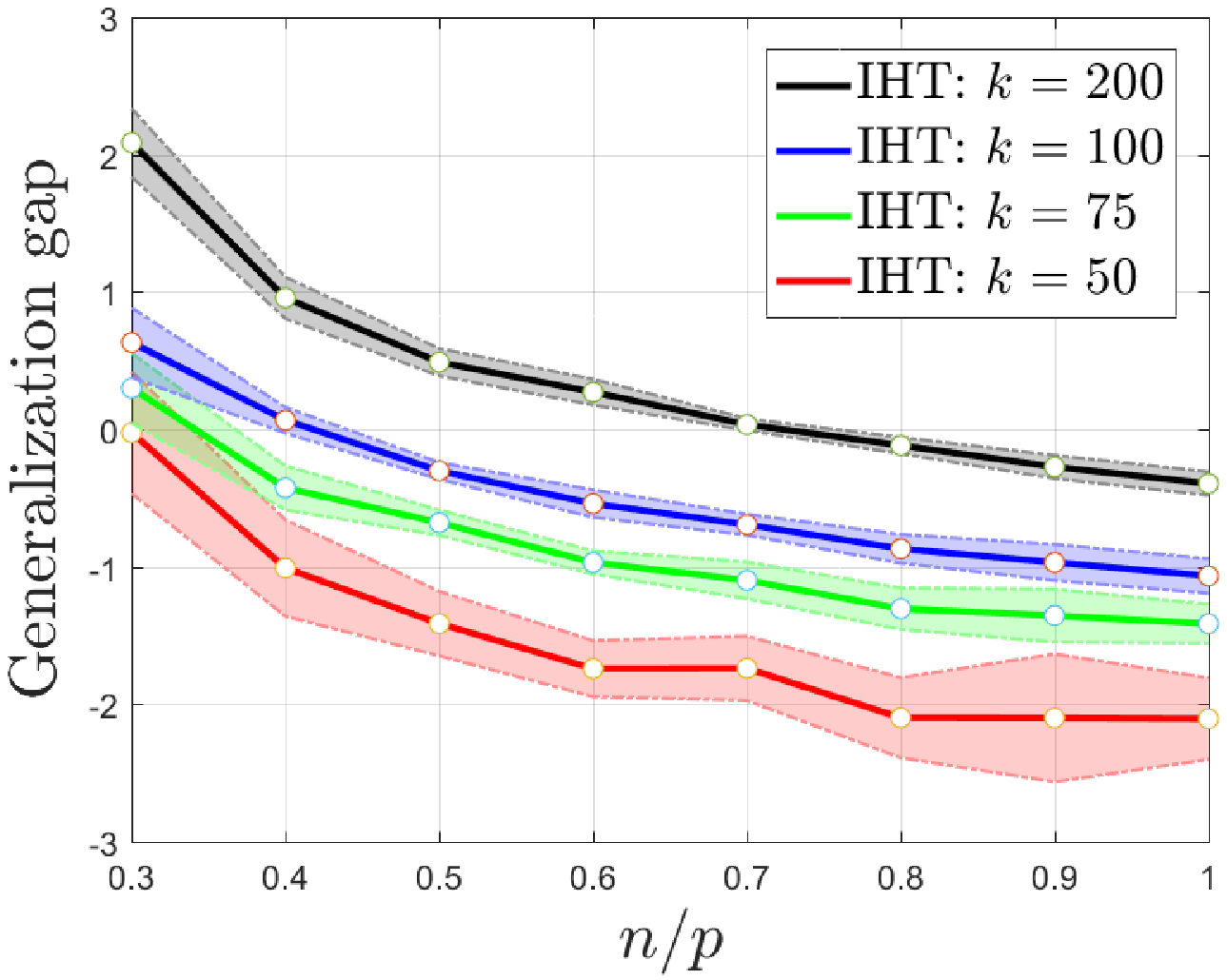}
\label{fig:generalizationgap_sparsity_black_linear}
\includegraphics[width=3in]{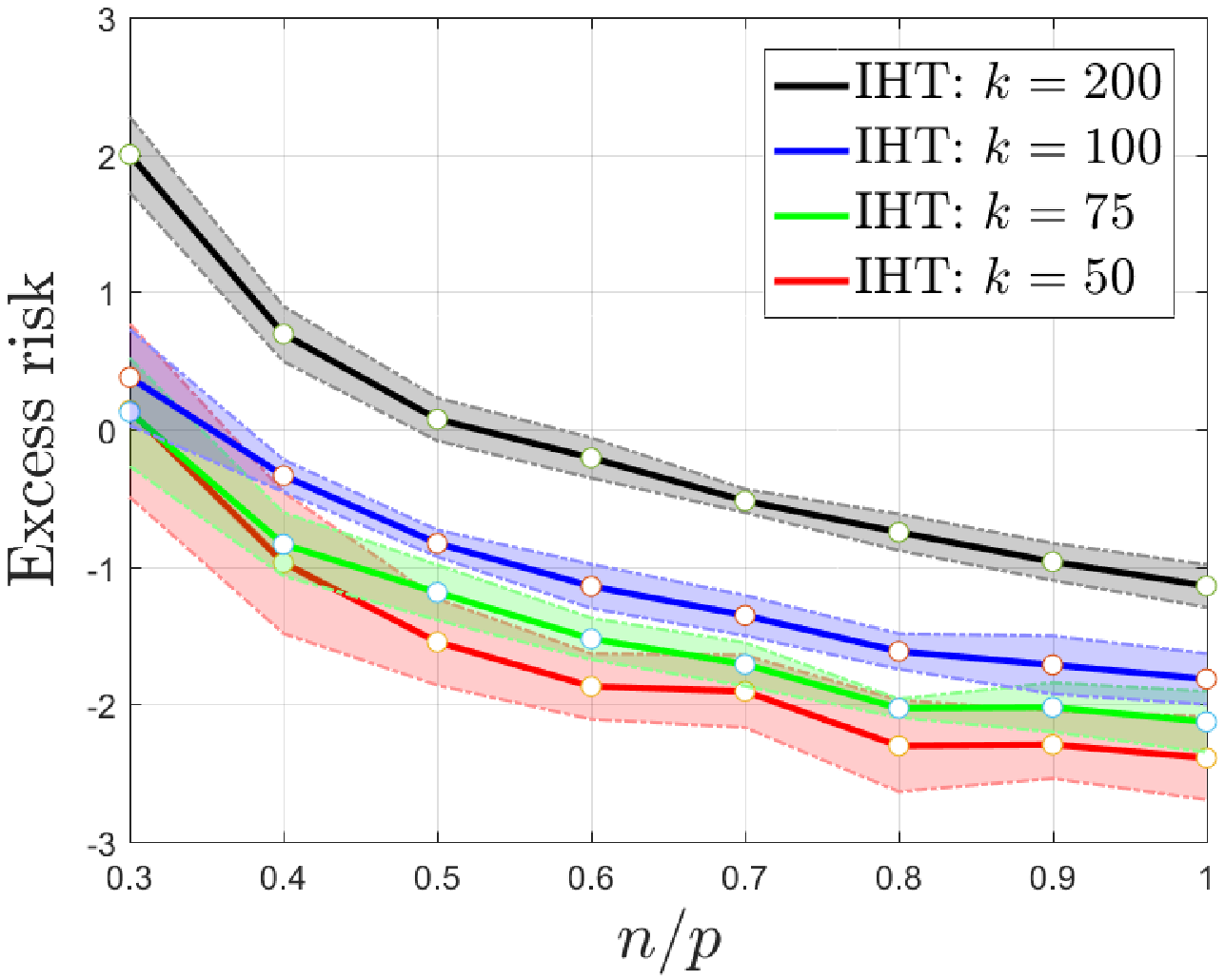}
\label{fig:excessrisk_sparsity_black_linear}
}
%\subfigure[Samples size versus noise level, Black box]{
%\includegraphics[width=2.25in]{figures/linear/iht_generalizationgap_samplesize_noise_blackbox.eps}
%\label{fig:generalizationgap_noise_black_linear}
%\includegraphics[width=2.25in]{figures/linear/iht_excessrisk_samplesize_noise_blackbox.eps}
%\label{fig:excessrisk_noise_black_linear}
%}
\caption{Generalization results of sparse linear regression with a black-box dense model.}
\label{fig:linear_black_box}
\end{figure}

\begin{figure}[h!]
\centering
\subfigure[White-box results for a sparse nominal model]{
\includegraphics[width=3in]{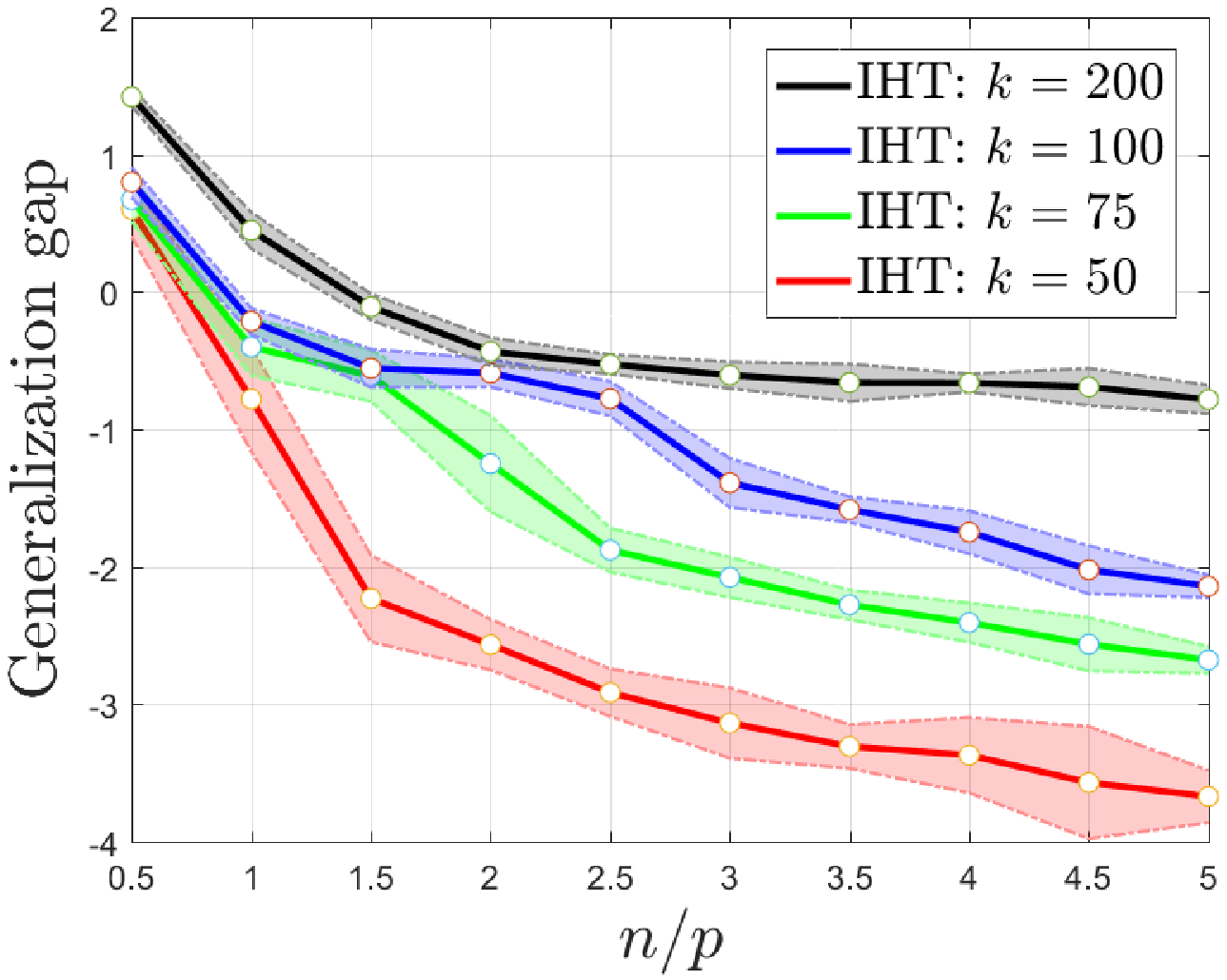}
\label{fig:generalizationgap_sparsity_white_logistic}
\includegraphics[width=3in]{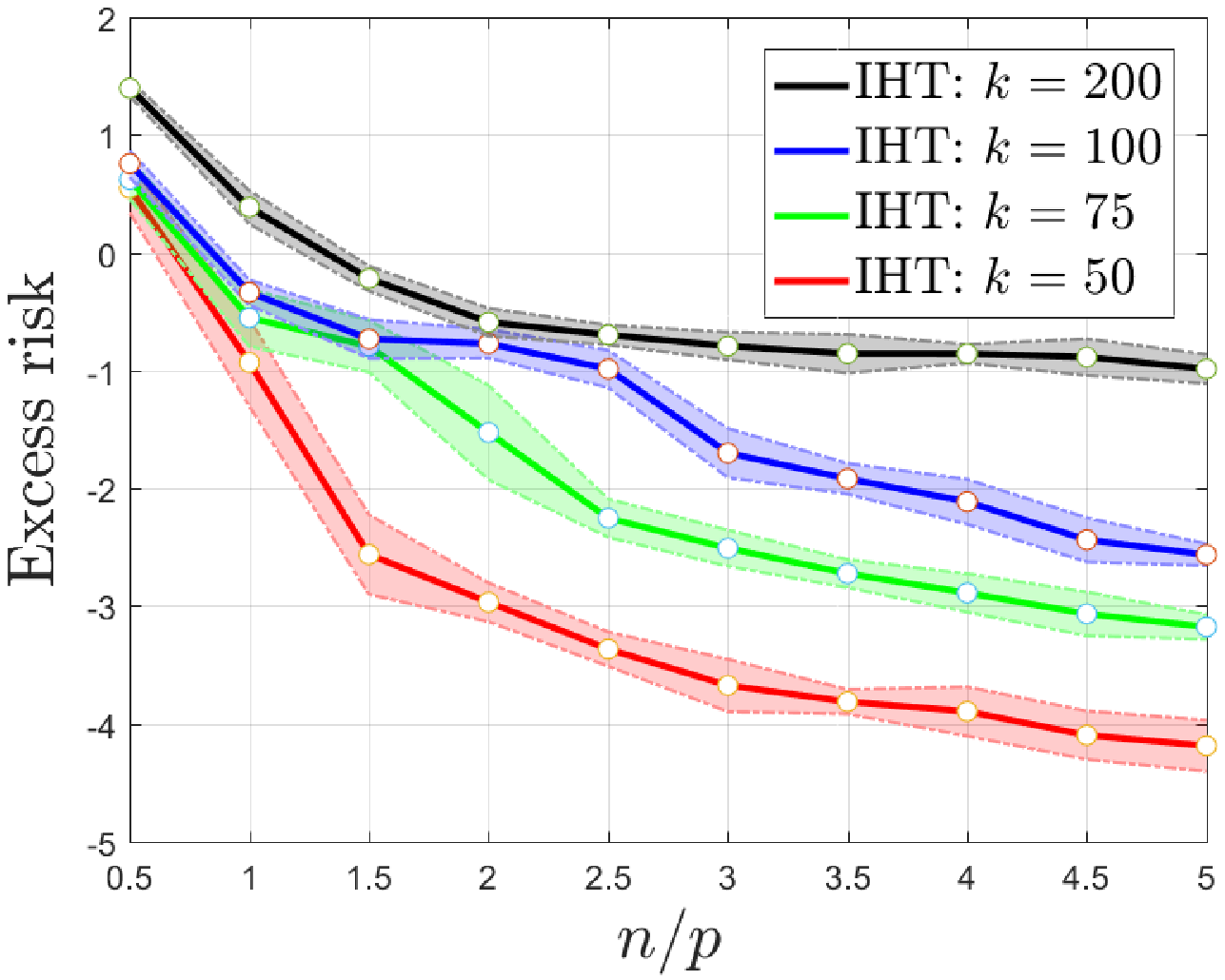}
\label{fig:excessrisk_sparsity_white_logistic}
}
\subfigure[Black-box results for a dense nominal model]{
\includegraphics[width=3in]{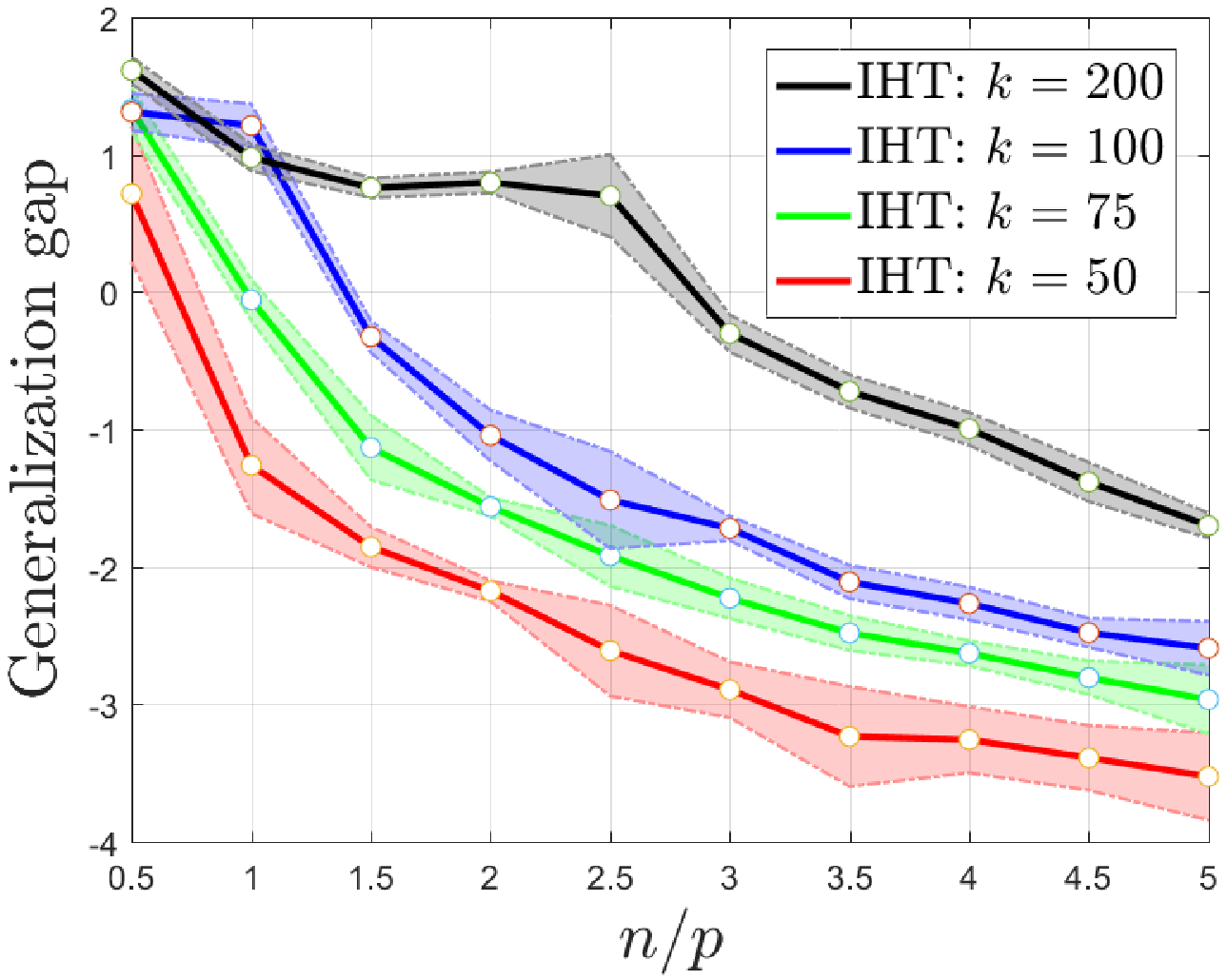}
\label{fig:generalizationgap_sparsity_black_logistic}
}
\caption{Generalization results of sparse logistic regression on the impact of sample size and sparsity level.}
\label{fig:TimeHuberTest}
\end{figure}

\subsection{Sparse logistic regression}

We further consider the binary logistic regression model with loss function $\ell(w; x_i, y_i) = \log\left(1+\exp(-y_{i}w^\top x_{i})\right)$. In this set of simulation study, each data feature $x_i$ is sampled from standard multivariate Gaussian distribution and its binary label $y_i \in \{-1,+1\}$ is determined by the conditional probability $\mathbb{P}(y_i|x_i; \bar w) = \exp (2y_i \bar w^\top x)/(1+\exp (2y_i \bar w^\top x_i))$ with a sparse parameter vector $\bar{w}$. We test with feature dimension $p=1000$ and aim to show the impact of varying $n/p\in (0.5,5)$ and $k\in \{50, 75, 100, 200\}$ on the actual generalization performance of IHT in both white-box and black-box regimes.
\begin{itemize}
  \item In the white-box setting, the true parameter vector $\bar{w}$ is set to be a $\bar k$-sparse vector with $\bar k =50$. Since for logistic loss the population risk function $F$ has no close-form expression, we approximate the population value $F(w)$ by its empirical version with sufficient sampling. In order to compute the excess risk, we need to estimate the optimal population risk which in view of the proof of Corollary~\ref{corol:generalization_barw_logisticreg} is given by $\min_{\|w\|_0\le k} F(w)=F(\bar w)$ for any $k \ge \bar k$. The evolving curves of generalization gap and excess risk as functions of sample size under different sparsity levels are shown in Figure~\ref{fig:excessrisk_sparsity_white_logistic}. For each fixed $k$, we can see that the generalization gap and excess risk of IHT decrease as $n$ increases, while for each fixed $n$, these generalization performance measurements increase as $k$ increases. These observations are consistent with the implication of Corollary~\ref{corol:generalization_barw_iht} in binary logistic regression.
  \item In the black-box setting, we also consider a nearly sparse vector $\bar w = \bar w' + \varepsilon'$ where $\bar w'$ is a $\bar k$-sparse standard Gaussian vector with $\bar k =50$ and $\varepsilon'$ is a zero-mean Gaussian noise vector with sufficiently small variance. In this case, the excess risk is computationally intractable as the optimal population risk $\min_{\|w\|_0\le k} F(w)$ is hard to compute exactly for logistic loss. Therefore, we only plot the evolving curves of generalization gap in Figure~\ref{fig:generalizationgap_sparsity_black_logistic}. In support of the theoretical bounds established in Corollary~\ref{corol:uniform_stability_iht}, these curves clearly show that the generalization gap of IHT decreases as sample size $n$ grows larger and sparsity level $k$ becomes smaller.
\end{itemize}

\begin{figure}[h!]
\centering
\subfigure[Impact of sample size and signal strength\label{fig:stability_white_linear}]{
\includegraphics[width=3in]{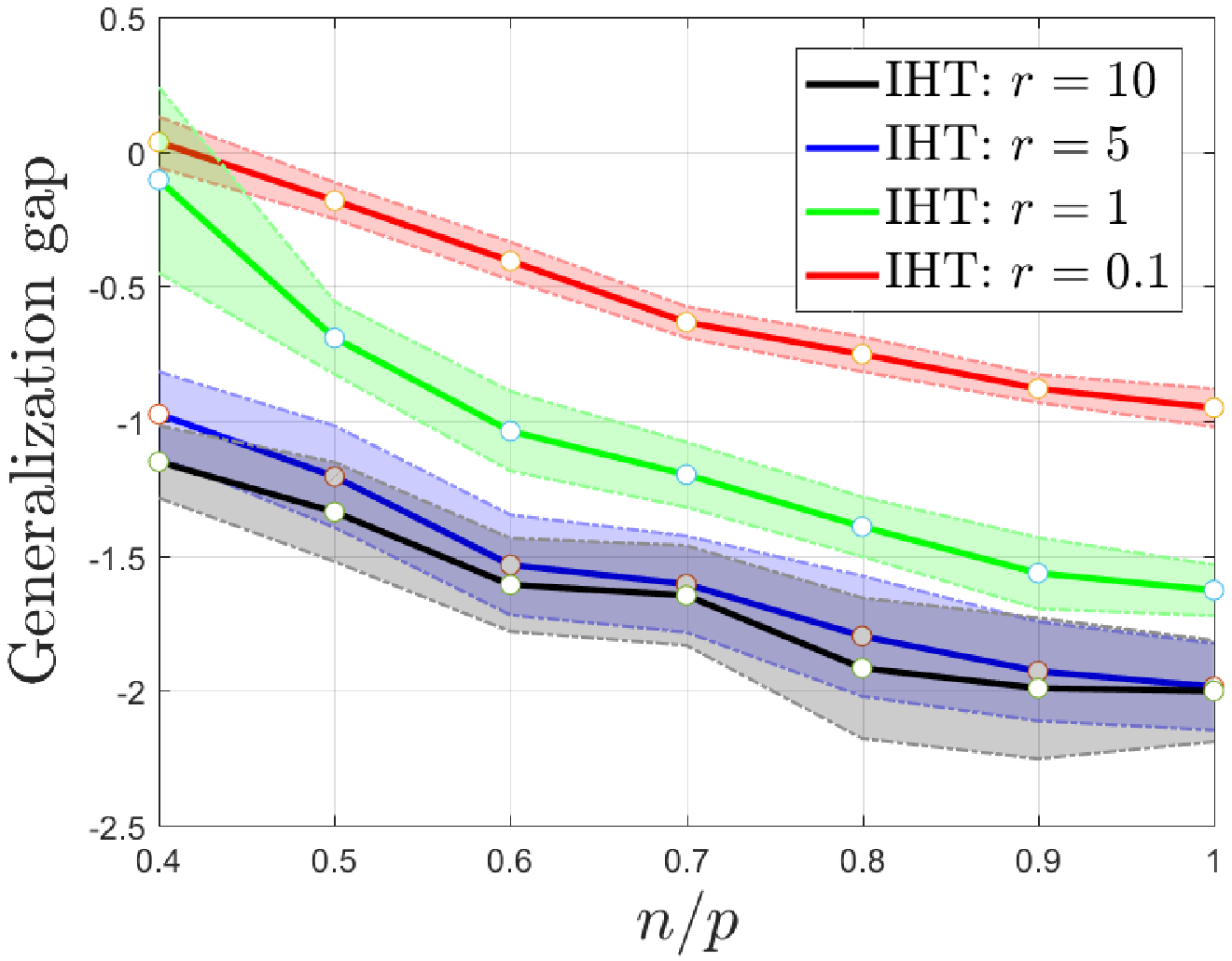}
\label{fig:stability_generalizationgap_strength_white_linear}
\includegraphics[width=3in]{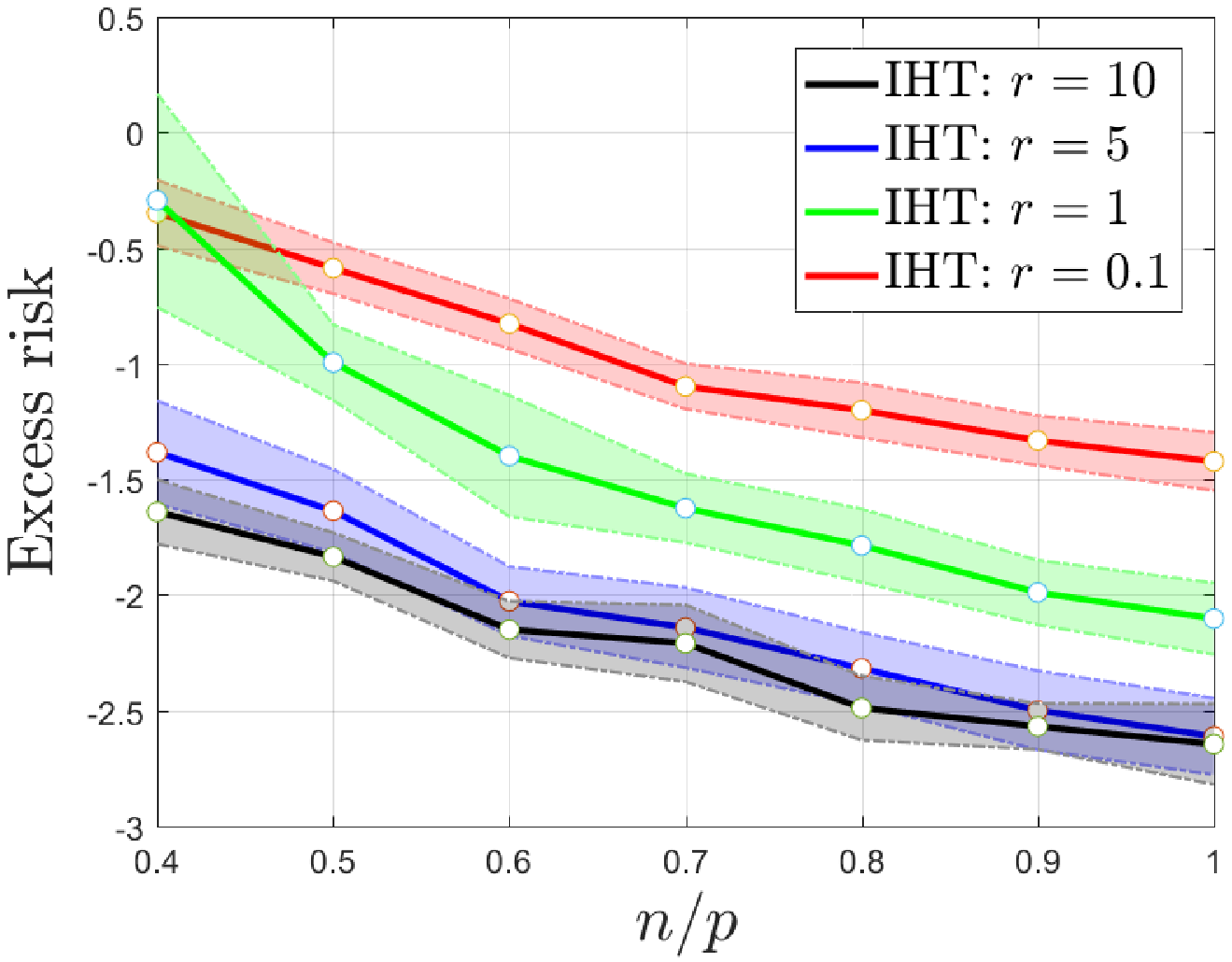}
\label{fig:stability_excessrisk_strength_white_linear}
}
\subfigure[Impact of sample size and sparsity level\label{fig:stability_black_linear}]{
\includegraphics[width=3in]{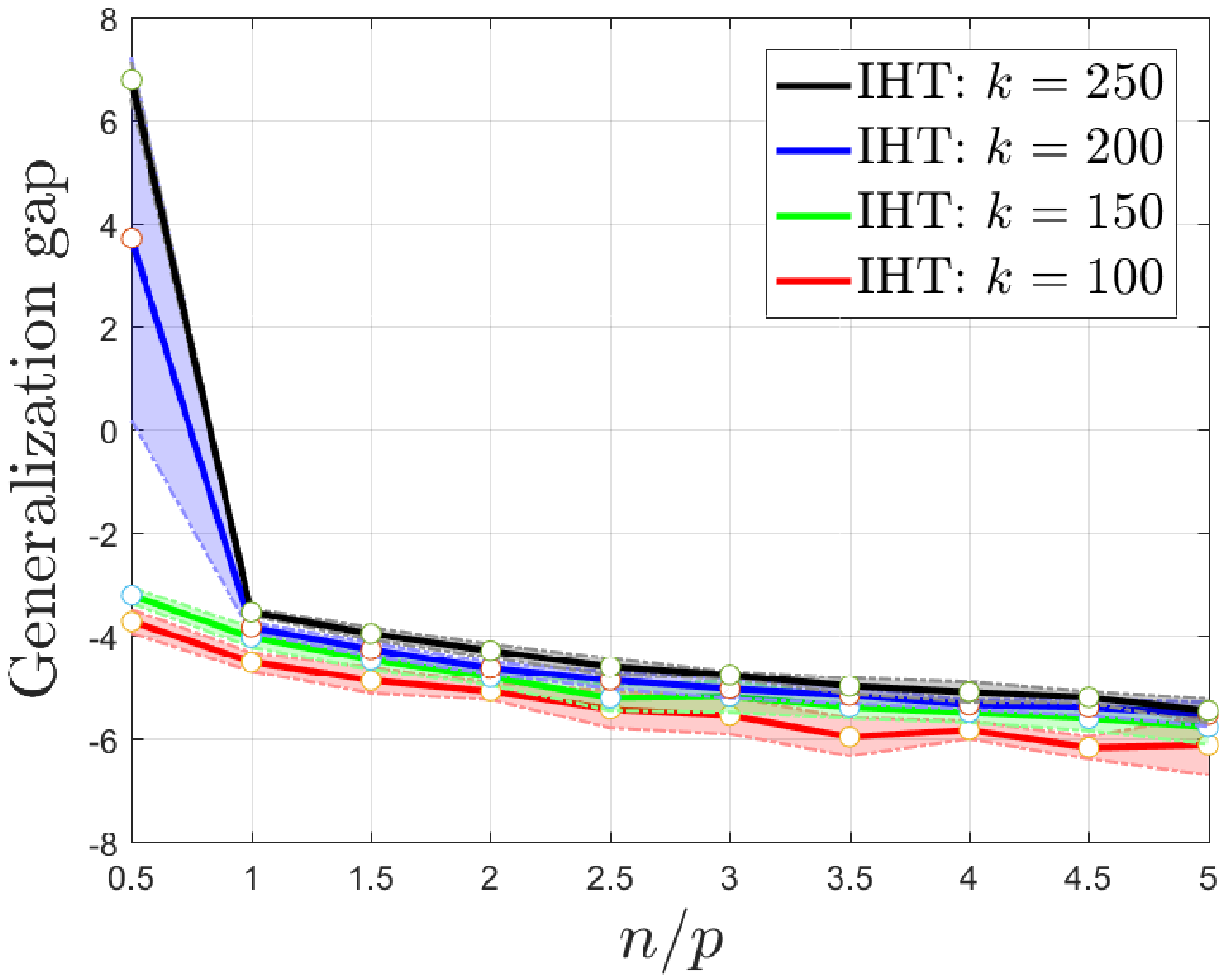}
\label{fig:stability_generalizationgap_sparsity_black_linear}
\includegraphics[width=3in]{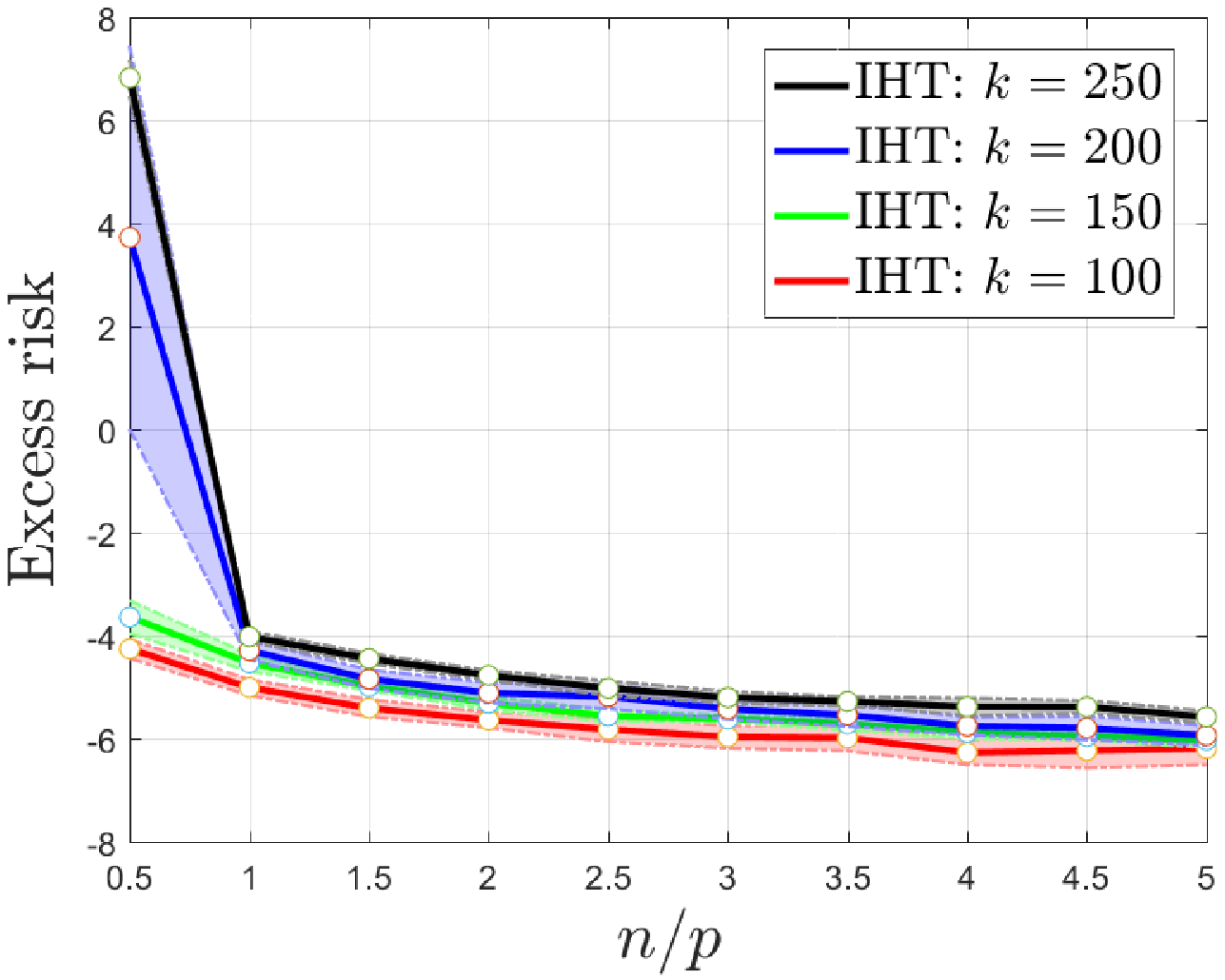}
\label{fig:stability_excessrisk_sparsity_black_linear}
}
\caption{IHT stability and generalization results of sparse linear regression.}
\label{fig:TimeHuberTest}
\end{figure}

\subsection{IHT stability and generalization theory verification}

Finally, we carry out a set of numerical experiments to verify the IHT stability and generalization theory as presented in Theorem~\ref{thrm:uniform_stability_strong_iht}, which mainly conveys that when sample size is sufficiently large, the IHT stability of the population risk $F$ plays an important role for obtaining tighter generalization bounds. For this experiment, we consider the sparse linear regression model as studied in Section~\ref{ssect:experiment_sparse_linear_regression} with $p=1000$.

We first study the case where the true parameter vector $\bar{w}$ is $\bar k$-sparse. In this special case, given $w^{(0)}=0$ and any step-size $\eta\in (0,1)$, it can be easily shown that the population risk $F$ is $(\varepsilon_{\bar k},\eta, T, w^{(0)})$-IHT stable with $\varepsilon_{\bar k}=\eta\bar w_{\min}$. Indeed, based on the close-form expression in~\eqref{equat:linear_regression_close_form} we can prove by induction that $w^{(t)}=\mathrm{H}_{\bar k} \left(w^{(t-1)}-\eta \nabla F(w^{(t-1)})\right)=(1- (1-\eta)^t) \bar w$ for all $t\ge 1$, which then implies the desired stability of IHT as $1- (1-\eta)^t \ge \eta$ when $t\ge 1$. Therefore, for a fixed $\eta\in (0,1)$, the stability strength $\varepsilon_{\bar k}=\eta\bar w_{\min}$ is controlled by the underlying signal strength $\bar w_{\min}$. In our experiment, we test with $\eta=0.5$ and set $\bar w=r \tilde w$ where $\tilde w$ is a fixed $\bar k$-sparse standard Gaussian vector with $\bar k = 100$ and $r>0$ controls the strength of signal. Figure~\ref{fig:stability_white_linear} shows convergence curves of generalization bounds under varying $n/p\in (0.4,1)$ and $r\in \{0.1, 1, 5, 10\}$. These results indicate that better generalization performance can be achieved under relatively larger $n$ and $r$, which supports the theoretical prediction by Theorem~\ref{thrm:uniform_stability_strong_iht}.

The generalization bounds in Theorem~\ref{thrm:uniform_stability_strong_iht} also suggest that the generalization performance should be invariable to sparsity level $k$ provided that IHT is stable and sample size is sufficiently large. In order to check this point, we further conduct an experiment with a nearly sparse model $\bar w$ as considered in the previous black-box regimes to show the sensitivity of the generalization performance of IHT to sparsity level. More specifically, we set $\bar w = \bar w' + \varepsilon'$ where $\bar w'$ is a $k$-sparse sparse vector whose non-zero entries are drawn from zero-mean Gaussian distribution with sufficiently large variance, while $\varepsilon'$ is a zero-mean Gaussian noise vector with sufficiently small variance such that IHT is ensured to be stable. Figure~\ref{fig:stability_black_linear} shows the convergence curves of generalization gap and excess risk under varying $n/p\in (0.5,5)$ and $k\in \{100, 150, 200, 250\}$. It can be clearly observed from this group of results that the generalization performance of IHT is sensitive to $k$ for relatively smaller $n$ but becomes much less sensitive to $k$ as $n$ grows larger.

\section{Conclusions}
\label{sect:conclusion}

In this paper, we studied the generalization theory for the $\ell_0$-ERM estimator which has long been applied with remarkable success in high-dimensional data analysis. Traditional generalization theory for convex ERM, however, does not readily extend to such a non-convex and NP-hard problem regime due to the presence of cardinality constraint. By assuming the unknown nominal model to be truly sparse, we established a set of generalization gap and excess risk bounds for $\ell_0$-ERM with restricted strongly convex risk function. Particularly, up to logarithmic factors, our high probability excess risk bound is minimax optimal over the cardinality constraint. In a more realistic setting where the generative model of data is not accessible, we further derived a set of black-box generalization results using uniform convergence and uniform stability arguments. Blessed with restricted strong convexity, our generalization bounds are substantially tighter or comparable to those of the widely studied $\ell_2$-ERM.

On top of the generalization theory for $\ell_0$-ERM, we further established several high probability generalization bounds for the IHT algorithm which serves as one of the most popular first-order greedy selection methods for solving $\ell_0$-ERM. Particularly, we have shown that IHT generalizes well provided that the sample size is sufficiently large and the population risk function $F$ is stable with respect to IHT iteration. We have substantialized our results to sparse linear regression and sparse logistic regression models to demonstrate the applicability of our theory.

We expect that the theory developed in this article will fuel future investigation on $\ell_0$-ERM and/or IHT with non-convex loss functions such as those used in the common practice of deep neural nets pruning~\cite{frankle2019lottery,han2016deep}, but rarely studied in theory~\cite{sun2019optimization}. As a starting point, we have shown in Theorem~\ref{thrm:universe_generalizaion} a unform convergence bound $\mathcal{\tilde O} \left(\sqrt{k\log(p)/n}\right)$ in terms of the risk function value which is applicable to the non-convex setting. We conjecture that similar uniform convergence bounds can be derived for the gradient and Hessian estimation to match those established for certain $\ell_1$-regularized non-convex M-estimators~\cite{mei2018landscape}, which would be beneficial for understanding the landscape of a sparse neural network. Also, it is interesting to further explore the structure information such as the deep and wide architectures to hopefully obtain stronger generalization bounds for deep learning with sparsity.

\appendix

\makeatletter
\renewcommand\theequation{A.\@arabic\c@equation }
\makeatother \setcounter{equation}{0}

\section{Some auxiliary lemmas}\label{append:lemmas}

In this appendix section we collect a number of auxiliary lemmas that will be used in our analysis. The proofs of these lemmas are deferred to Appendix~\ref{ProofforAuxiliaryLemmas}.

\begin{lemma}\label{lemma:estimation_error}
Assume that $f$ is $\mu_s$-strongly convex. Then for any $w,w'$ such that $\|w- w'\|_0 \le s $ and $f(w)\le f(w')+\epsilon$ for some $\epsilon\ge 0$, the following bound holds
\[
\|w - w'\| \le \frac{2\|\nabla_{I \cup I'} f(w')\|}{\mu_s} \le \frac{2\sqrt{s}\|\nabla f(w')\|_\infty}{\mu_s}+ \sqrt{\frac{2\epsilon}{\mu_s}},
\]
where $I = \supp(w)$ and $I'=\supp(w')$. Moreover, if $w'_{\min}> \frac{2\sqrt{s}\|\nabla f(w')\|_\infty}{\mu_s}+ \sqrt{\frac{2\epsilon}{\mu_s}}$, then it holds that $\supp(w') \subseteq \supp(w)$.
\end{lemma}

The following simple lemma that controls the maximum over a set of sub-Gaussian random variables is useful in our analysis.
\begin{lemma}\label{lemma:max_subgaussian_bound}
Let $X_1,..., X_p$ be $p\ge2$ zero-mean $\sigma^2$-sub-Gaussian random variables. Then
\[
\mathbb{E}\left[\max_{1\le j\le p }|X_j|\right] \le \sigma\sqrt{2\log(2p)} ,
\]
and
\[
\quad \mathbb{E}\left[\max_{1\le j\le p }X^2_j\right] \le \sigma^2(72 + 16\log p).
\]
\end{lemma}

\begin{lemma}\label{lemma:sub_Gaussian_bounds}
Under Assumption~\ref{assump:gradient_sub_gaussian}, for any $\delta\in(0,1)$ it holds with probability at least $1-\delta$ that
\[
\|\nabla F_S(\bar w)\|_\infty \le \sigma\sqrt{\frac{2\log(p/\delta)}{n}}.
\]
Moreover, the following expectation bounds hold:
\[
\mathbb{E}_S\left[\|\nabla F_S(\bar w)\|_\infty\right] \le \sigma\sqrt{\frac{2 \log (2p)}{n}},\quad \mathbb{E}_S\left[\|\nabla F_S(\bar w)\|^2_\infty\right] \le \frac{\sigma^2(72 + 16\log p)}{n}.
\]
\end{lemma}

\section{Proof of Main Results for $\ell_0$-ERM}
\label{append:proof_l0_erm}

\subsection{Proof of Theorem~\ref{thrm:generalization_barw}}
\label{apdsect:proof_generalization_barw}

This appendix subsection is devoted to providing a detailed proof of Theorem~\ref{thrm:generalization_barw} as restated below.
\ERMWhiteBoxBounds*

Before proving the main result, we first prove the following lemma which is key for deriving the in expectation bound in the main theorem.
\begin{lemma}\label{lemma:L2_error_expectation}
Assume that $\mathbb{E}_S\left[\|\nabla F_S(\bar w)\|^2_\infty\right]\le \delta_n$ for some $\delta_n\in (0,1)$. Assume that $F_S(w)$ is $\mu_{2k}$-strongly convex with probability at least $1- \delta'_n$. If $\delta'_n \le \min\left\{0.5, \frac{\delta_n}{4R^2}\right\}$, then
\[
\mathbb{E}_S\left[\|w_{S,k} - \bar w\|^2\right] \le \left(\frac{16k}{\mu^2_{2k}}+1\right)\delta_n.
\]
Moreover, let $\tilde w_{S,k}$ be an $k$-sparse solution such that $F_S(\tilde w_{S,k})\le F_S(\bar w) + \epsilon$. Then
\[
\mathbb{E}_S\left[\|\tilde w_{S,k} - \bar w\|^2\right] \le \left(\frac{32k}{\mu^2_{2k}}+1\right)\delta_n + \frac{4\epsilon}{\mu_{2k}}.
\]
\end{lemma}
\begin{proof}
Let us consider a random indication variable $Y_S$ defined by $Y_S=1$ if $F_S$ is $\mu_{2k}$-strongly convex, and $Y_S=0$ otherwise. Let $\delta = \min\left\{0.5, \frac{\delta_n}{4R^2}\right\}$. Then by assumption $\mathbb{P}(Y_S=1)\ge 1- \delta'_n \ge 1 - \delta $ and $\mathbb{P}(Y_S = 0) \le \delta'_n \le\delta$. Since $\mathbb{E}_S\left[\|\nabla F_S(\bar w)\|^2_\infty\right]\le \delta_n$, we have
\[
\begin{aligned}
\delta_n \ge& \mathbb{E}_S\left[\|\nabla F_S(\bar w)\|^2_\infty\right] = \mathbb{E}_{Y_S} \left[\mathbb{E}_{S\mid Y_S}\left[\|\nabla F_S(\bar w)\|^2_\infty \mid Y_S\right]\right]\\
=& \mathbb{P}(Y_S=1)\mathbb{E}_{S\mid Y_S=1}\left[\|\nabla F_S(\bar w)\|^2_\infty \mid Y_S=1\right] \\
&+ \mathbb{P}(Y_S=0)\mathbb{E}_{S\mid Y_S=0}\left[\|\nabla F_S(\bar w)\|^2_\infty \mid Y_S=0\right]\\
\ge& (1-\delta)\mathbb{E}_{S\mid Y_S=1}\left[\|\nabla F_S(\bar w)\|^2_\infty \mid Y_S=1\right]\\
\ge& 0.5\mathbb{E}_{S\mid Y_S=1}\left[\|\nabla F_S(\bar w)\|^2_\infty \mid Y_S=1\right],
\end{aligned}
\]
where we have used $\mathbb{P}(Y_S=1) \ge 1 - \delta\ge 0.5$ and $\|\nabla F_S(\bar w)\|^2_\infty\ge0$. It follows immediately from the above that
\[
\mathbb{E}_{S\mid Y_S=1}\left[\|\nabla F_S(\bar w)\|^2_\infty \mid Y_S=1\right] \le  2\delta_n.
\]
Therefore,
\[
\begin{aligned}
& \mathbb{E}_S\left[\|w_{S,k} - \bar w\|^2\right] = \mathbb{E}_{Y_S} \left[\mathbb{E}_{S\mid Y_S}\left[\| w_{S,k} - \bar w\|^2 \mid Y_S\right]\right]\\
=& \mathbb{P}(Y_S=1)\mathbb{E}_{S\mid Y_S=1}\left[\|w_{S,k} - \bar w\|^2 \mid Y_S=1\right] \\
& + \mathbb{P}(Y_S=0)\mathbb{E}_{S\mid Y_S=0}\left[\|w_{S,k} - \bar w\|^2 \mid Y_S=0\right] \\
\le& \mathbb{E}_{S\mid Y_S=1}\left[\|w_{S,k} - \bar w\|^2 \mid Y_S=1\right] + \delta\mathbb{E}_{S\mid Y_S=0}\left[\|w_{S,k} - \bar w\|^2 \mid Y_S=0\right]\\
\overset{\zeta_1}{\le}& \mathbb{E}_{S\mid Y_S=1}\left[\frac{8k\|\nabla F_S(\bar w)\|^2_\infty}{\mu_{2k}^2} \mid Y_S=1\right] + 4R^2\delta \le \left(\frac{16k}{\mu^2_{2k}}+1\right)\delta_n,
\end{aligned}
\]
where ``$\zeta_1$'' follows from the optimality of $w_{S,k}$ and applying Lemma~\ref{lemma:estimation_error} with $s=2k$ and $\epsilon = 0$. Similarly, for $\tilde w_{S,k}$ we can again use Lemma~\ref{lemma:estimation_error} to show that
\[
\begin{aligned}
&\mathbb{E}_S\left[\|\tilde w_{S,k} - \bar w\|^2\right] \\
\le& \mathbb{E}_{S\mid Y_S=1}\left[\|\tilde w_{S,k} - \bar w\|^2 \mid Y_S=1\right] + \delta\mathbb{E}_{S\mid Y_S=0}\left[\|\tilde w_{S,k} - \bar w\|^2 \mid Y_S=0\right]\\
\le& \mathbb{E}_{S\mid Y_S=1}\left[\frac{16k\|\nabla F_S(\bar w)\|^2_\infty}{\mu_{2k}^2} + \frac{4\epsilon}{\mu_{2k}} \mid Y_S=1\right] + 4R^2\delta \le \left(\frac{32k}{\mu^2_{2k}}+1\right)\delta_n + \frac{4\epsilon}{\mu_{2k}},
\end{aligned}
\]
This establishes the second desired bound.
\end{proof}

Now we are ready to prove the main theorem.
\begin{proof}[Proof of Theorem~\ref{thrm:generalization_barw}]
We first prove the high probability bound on excess risk. From Lemma~\ref{lemma:sub_Gaussian_bounds} we know that it holds with probability at least $1-\delta$ that
\[
\|\nabla F_S(\bar w)\|_\infty \le \sigma\sqrt{\frac{2\log(p/\delta)}{n}}.
\]
In the meanwhile, from the assumption we have that $F_S(w)$ is $\mu_{2k}$-strongly convex with probability at least $1- \delta'_n$, which according to Lemma~\ref{lemma:estimation_error} (with $\epsilon=0$) implies that with at least the same high probability
\[
\|w_{S,k} - \bar w\|^2 \le \frac{8k\|\nabla F_S(\bar w)\|^2_\infty}{\mu_{2k}^2}.
\]
The $L$-smoothness assumption of $\ell$ implies that $F$ is also $L$-strongly smooth. Since $\nabla F(\bar w)=0$, by union probability we get that with probability at least $1-\delta - \delta'_n$
\[
F(w_{S,k}) - F(\bar w) \le \frac{L}{2}\|w_{S,k} - \bar w\|^2 \le \frac{4kL\|\nabla F_S(\bar w)\|^2_\infty}{\mu_{2k}^2} \le \frac{8L}{\mu_{2k}^2}\left(\frac{k\sigma^2\log (p/\delta)}{n}\right),
\]
which is the desired excess risk bound. To prove the generalization gap bound, let us consider the following decomposition:
\[
F(w_{S,k}) - F_S(w_{S,k}) = \underbrace{F(w_{S,k}) - F(\bar w)}_{A_1} + \underbrace{F(\bar w) - F_S(\bar w)}_{A_2} + \underbrace{F_S(\bar w) - F_S(w_{S,k})}_{A_3}.
\]
The above argument has already shown that with probability at least $1-\delta'_n$ over $S$ the term $A_1$ can be bounded as
\[
A_1 \le \frac{L}{2}\|w_{S,k} - \bar w\|^2 \le \frac{4kL\|\nabla F_S(\bar w)\|^2_\infty}{\mu_{2k}^2} .
\]
Next we bound the term $A_3$ from above. Let $\bar I = \supp(\bar w)$ and $I_S=\supp(w_{S,k})$. By assumption $F_S(w)$ is $L$-strongly smooth and $\mu_{2k}$-strongly convex with probability at least $1- \delta'_n$, and thus the following holds with probability at least $1- \delta'_n$ over $S$:
\[
\begin{aligned}
A_3 =& F_S(\bar w) - F_S(w_{S,k}) \le \left|F_S(w_{S,k}) - F_S(\bar w) \right|\\
\le& \left|\langle \nabla F_S(\bar w), w_{S,k} - \bar w  \rangle\right| + \frac{L}{2}\|w_{S,k} - \bar w\|^2 \\
\le& \|\nabla_{\bar I \cup I_S} F_S(\bar w)\|\|w_{S,k} - \bar w\| + \frac{L}{2}\|w_{S,k} - \bar w\|^2 \\
\le& \sqrt{2k}\|\nabla F_S(\bar w)\|_\infty \|\|w_{S,k} - \bar w\| + \frac{L}{2}\|w_{S,k} - \bar w\|^2 \\
\le& \frac{k}{L}\|\nabla F_S(\bar w)\|^2_\infty + L \|w_{S,k} - \bar w\|^2 \le \left(\frac{1}{L} + \frac{8L}{\mu_{2k}^2}\right) k\|\nabla F_S(\bar w)\|^2_\infty,
\end{aligned}
\]
where in the last but one inequality we again have used fact $|ab| \le \frac{a^2}{2c} + \frac{cb^2}{2}$, $\forall c>0$. As a result, by union bound it holds with probability at least $1- \delta'_n - \delta/2$ over $S$ that
\[
\begin{aligned}
&A_1 + A_3 \le \left(\frac{1}{L} + \frac{12L}{\mu_{2k}^2}\right) k\|\nabla F_S(\bar w)\|^2_\infty \\
\le& 2 \left(\frac{1}{L} + \frac{12L}{\mu_{2k}^2}\right)\left(\frac{k\sigma^2\log (2p/\delta)}{n}\right)\le \frac{26L}{\mu_{2k}^2}\left(\frac{k\sigma^2\log (2p/\delta)}{n}\right),
\end{aligned}
\]
where in the last inequality we have used $L\ge \mu_{2k}$. Since $\ell(w;\xi)\le M$, by invoking Hoeffding inequality we know that with probability at least $1-\delta/2$ over $S$,
\[
A_2 = F(\bar w) - F_S(\bar w) \le M\sqrt{\frac{\log(2/\delta)}{2n}}.
\]
Finally, by union bound the following holds with probability at least $1- \delta'_n - \delta$ over $S$
\[
\begin{aligned}
F(w_{S,k}) - F_S(w_{S,k}) = A_1+A_2+ A_3 \le \frac{26L}{\mu_{2k}^2}\left(\frac{k\sigma^2\log (2p/\delta)}{n}\right) + M\sqrt{\frac{\log(2/\delta)}{2n}},
\end{aligned}
\]
which implies the desired generalization gap bound.

In the following, we derive the generalization bounds in expectation. From the decomposition $F(w_{S,k}) - F_S(w_{S,k}) = A_1+A_2+ A_3$ we can show that
\[
\begin{aligned}
&\left|\mathbb{E}_S[F(w_{S,k}) - F_S(w_{S,k})]\right| \\
=& \left|\mathbb{E}_S[F(w_{S,k}) - F(\bar w)] + \mathbb{E}_S[F(\bar w) - F_S(\bar w) ] + \mathbb{E}_S[F_S(\bar w) - F_S(w_{S,k})]\right| \\
\overset{\zeta_1}{=}& \left|\mathbb{E}_S[F(w_{S,k}) - F(\bar w)] + \mathbb{E}_S[F_S(\bar w) - F_S(w_{S,k})]\right| \\
\le& \mathbb{E}_S[\left|F(w_{S,k}) - F(\bar w)\right|] + \mathbb{E}_S[\left|F_S(\bar w) - F_S(w_{S,k})\right|] \\
\overset{\zeta_2}{\le}& \frac{L}{2}\mathbb{E}_S\left[\|w_{S,k} - \bar w\|^2\right] + \frac{k}{L}\mathbb{E}_S\left[\|\nabla F_S(\bar w)\|^2_\infty\right] + L \mathbb{E}_S\left[\|w_{S,k} - \bar w\|^2\right] \\
\le& \frac{3L}{2}\mathbb{E}_S\left[\|w_{S,k} - \bar w\|^2\right] + \frac{k}{L}\mathbb{E}_S\left[\|\nabla F_S(\bar w)\|^2_\infty\right]\\
\overset{\zeta_3}{\le}& \frac{3L}{2} \left(\frac{16k}{\mu^2_{2k}}+1\right)\delta_n + \frac{k}{L} \delta_n = \left(\frac{(24L^2 + \mu_{2k}^2)k}{L\mu^2_{2k}} + \frac{3L}{2}\right)\delta_n \le \left(\frac{25Lk}{\mu^2_{2k}} + \frac{3L}{2}\right)\delta_n,
\end{aligned}
\]
where in ``$\zeta_1$'' we have used $\mathbb{E}_S[F_S(\bar w)] = F(\bar w)$, $\zeta_2$ follows from the bounding results for $A_1, A_3$, and in ``$\zeta_3$'' we have used Lemma~\ref{lemma:sub_Gaussian_bounds} and Lemma~\ref{lemma:L2_error_expectation}. This proves the  generalization gap bound in expectation. The excess risk in expectation can then be bounded according to the basic fact $\mathbb{E}_S \left[F(w_{S,k}) - F(\bar w)\right] \le  \mathbb{E}_S \left[\Delta_{S,k} \right]$. This completes the proof.
\end{proof}

\subsection{Proof of Corollary~\ref{corol:generalization_barw_linearreg}}
\label{apdsect:proof_corollary_linear}

Here we provide a detailed proof of Corollary~\ref{corol:generalization_barw_linearreg} as restated below.
\ERMWhiteBoxBoundsLinaer*

The following lemma useful in our analysis.
\begin{lemma}\label{lemma:strong_convexity_subgaussian}
Suppose $x_i$ are drawn i.i.d. from a zero-mean sub-Gaussian distribution with covariance matrix $\Sigma\succ 0$. Let $X = [x_1,...,x_n]\in \mathbb{R}^{d\times n}$. Assume that $\Sigma_{jj}\le \sigma^2$. Then there exist universal positive constants $c_0$ and $c_1$ such that for all $w\in \mathbb{R}^p$
\[
\begin{aligned}
\frac{\|X^\top w\|^2}{n} \ge& \frac{1}{2}\|\Sigma^{1/2}w\|^2 - c_1\frac{\sigma^2\log (p)}{n} \|w\|^2_1
\end{aligned}
\]
holds with probability at least $1- \exp\{-c_0 n\}$.
\end{lemma}
The lemma follows immediately from~\cite[Lemma 6]{agarwal2012fast}. Based on this lemma and the fact $\|w\|_1\le \sqrt{k}\|w\|$ when $\|w\|_0\le k$, it holds with probability at least $1- \exp\{-c_0 n\}$ that $F_S(w)$ is $\mu_{2k}$-strongly convex with
\[
\mu_{2k}= \frac{1}{2}\lambda_{\min}(\Sigma) - \frac{k c_1\log (p)}{n}.
\]
Provided that $n \ge \frac{4kc_1 \log (p)}{\lambda_{\min}(\Sigma)}$, we have $\mu_{2k}\ge \frac{1}{4}\lambda_{\min}(\Sigma)$ holds with probability at least $1- \exp\{-c_0 n\}$. Now we are ready to prove Corollary~\ref{corol:generalization_barw_linearreg}.

\begin{proof}[Proof of Corollary~\ref{thrm:generalization_barw}]
Let $\xi=\{x,\varepsilon\}$ in which $x$ is zero-mean sub-Gaussian with covariance matrix $\Sigma \succ 0$ and $\varepsilon$ is zero-mean $\sigma^2$-sub-Gaussian. Obviously $\nabla F(\bar w)=\mathbb{E}_{\xi}\left[\nabla \ell(\bar w;\xi)\right]=\mathbb{E}_{\varepsilon,x}\left[-\varepsilon x\right]=0$. Given that $\Sigma_{jj} \le 1$, it can be shown that $\nabla_j \ell(\bar w;\xi_i) = -\varepsilon_i [x_i]_j$ are zero-mean $\sigma^2$-sub-Gaussian variables, which indicates that Assumption~\ref{assump:gradient_sub_gaussian} holds. By assuming $\|x_i\|\le 1$, we get that $\ell(w;\xi_i)$ is $L$-strongly smooth with $L=1$. Lemma~\ref{lemma:strong_convexity_subgaussian} implies that if $n \ge \frac{4kc_1 \log (p)}{\lambda_{\min}(\Sigma)}$, then it holds with probability at least $1- \exp\{-c_0 n\}$ that $F_S(w)$ is $\mu_{2k}$-strongly convex with $\mu_{2k}\ge \frac{1}{4}\lambda_{\min}(\Sigma)$. By applying the high probability bound in Theorem~\ref{thrm:generalization_barw} we obtain that with probability at least $1-\delta-\exp\{-c_0 n\}$,
\[
F(w_{S,k}) - F(\bar w) \le \mathcal{O}\left(\frac{1}{\lambda^2_{\min}(\Sigma)}\left(\frac{k\sigma^2\log (p/\delta)}{n}\right)\right).
\]
When sample size $n$ is sufficiently large such that $\exp\{-c_0 n\}\le \min\left\{0.5, \frac{\delta_n}{4R^2}\right\}$, it follows from the expected generalization bounds in Theorem~\ref{thrm:generalization_barw} that
\[
\mathbb{E}_S \left[F(w_{S,k}) - F(\bar w)\right] \le \mathbb{E}_S \left[F(w_{S,k}) - F_S(w_{S,k})\right] \le \mathcal{O}\left(\frac{1}{\lambda^2_{\min}(\Sigma)}\left(\frac{k\sigma^2\log (p/\delta)}{n}\right)\right).
\]
This implies the desired bounds.
\end{proof}

\newpage

\subsection{Proof of Corollary~\ref{corol:generalization_barw_logisticreg}}
\label{apdsect:proof_corollary_logistic}
In this subsection, we prove Corollary~\ref{corol:generalization_barw_logisticreg} as restated below.
\ERMWhiteBoxBoundsLogistic*

\begin{proof}[Proof of Corollary~\ref{corol:generalization_barw_logisticreg}]
Let $\xi=\{x,y\}$ in which $x$ is zero-mean sub-Gaussian with covariance matrix $\Sigma \succ 0$ and $y\in\{-1,1\}$ is generated by $\mathbb{P}(y|x; \bar w) = \frac{\exp (2y\bar w^\top x)}{1+\exp (2y\bar w^\top x)}$. The logistic loss function at $\xi_i$ is given by $\ell(w;\xi_i) = \log(1+\exp(-2y_i w^\top x_i))$. We first show that $\nabla F(\bar w)=\mathbb{E}_{\xi}\left[\nabla \ell(\bar w;\xi)\right]=0$. Indeed,
\[
\begin{aligned}
&\mathbb{E}_{\xi}\left[\nabla \ell(\bar w;\xi)\right] \\
=& \mathbb{E}_{x,y}\left[ \nabla \log(1+\exp(-2y \bar w^\top x))\right] = \mathbb{E}_{x} \left[\mathbb{E}_{y\mid x}\left[\nabla \log(1+\exp(-2y \bar w^\top x))\mid x\right]\right]\\
=& \mathbb{E}_{x} \left[\mathbb{P}(y=1\mid x)\nabla \log(1+\exp(-2 \bar w^\top x)) + \mathbb{P}(y=-1\mid x)\nabla \log(1+\exp(2 \bar w^\top x)) \right]\\
=& \mathbb{E}_{x} \left[\frac{\exp (2\bar w^\top x)}{1+\exp (2\bar w^\top x)}\frac{-2x \exp(-2 \bar w^\top x)}{1+\exp(-2 \bar w^\top x)} + \frac{1}{1+\exp (2\bar w^\top x)}\frac{2x \exp(2 \bar w^\top x)}{1+\exp(2 \bar w^\top x)} \right]=0.
\end{aligned}
\]
Next we show that $\nabla_j \ell(\bar w;\xi)=\frac{-2 y [x]_j\exp(-2y \bar w^\top x)}{1+\exp(-2y \bar w^\top x)}$ is a zero-mean sub-Gaussian random variable. Clearly, $\mathbb{E}[\nabla_j \ell(\bar w;\xi)]=0$. Since $y\in\{-1,1\}$ and $[x]_j$ is $\frac{\sigma^2}{32}$-sub-Gaussian, we can show the following
\[
\begin{aligned}
&\mathbb{P}\left( |\nabla_j \ell(\bar w;\xi)| \ge t\right) \\
=& \mathbb{P}\left( \frac{2 |[x]_j|\exp(-2y \bar w^\top x)}{1+\exp(-2y \bar w^\top x)} \ge t\right) \le \mathbb{P}\left(|[x]_j| \ge \frac{t}{2}\right) \le  2\exp\left(-\frac{4t^2}{\sigma^2}\right).
\end{aligned}
\]
Then based on the result~\cite[Lemma 1.5]{rigollet201518}  we know that for any $\lambda>0$,
\[
\mathbb{E}_{\xi}\left[\exp(\lambda\nabla_j \ell(\bar w;\xi))\right]\le \exp\left(\frac{\lambda^2\sigma^2}{2}\right),
\]
which shows that $\nabla_j \ell(\bar w;\xi)$ is $\sigma^2$-sub-Gaussian. This verifies the validness of Assumption~\ref{assump:gradient_sub_gaussian}.

Given that $\|x_i\|\le 1$, we have $\ell(w;\xi_i)$ is $L$-strongly smooth with $L\le4s(2y_iw^\top x_i)(1-s(2y_iw^\top x_i))\le1$. Since $\|w\|\le R$ and $\|x_i\|\le 1$, we must have $|y_iw^\top x_i|\le R$ and thus $[\Lambda(w)]_{ii}=4s(2y_iw^\top x_i)(1-s(2y_iw^\top x_i))\ge \frac{4}{(1+\exp(2R))^2}\ge\frac{1}{\exp(4R)}$. It follows that
\[
\nabla^2 F_S(w) = \frac{1}{n}X \Lambda(w) X^\top \succeq \frac{1}{n\exp(4R)} XX^\top.
\]
By invoking Lemma~\ref{lemma:strong_convexity_subgaussian} we obtain that if $n \ge \frac{4\sigma^2kc_1 \log (p)}{\lambda_{\min}(\Sigma)}$, then it holds with probability at least $1- \exp\{-c_0 n\}$ that $F_S(w)$ is $\mu_{2k}$-strongly convex with $\mu_{2k}\ge \frac{\lambda_{\min}(\Sigma)}{\exp(4R)}$. By applying the high probability bound in Theorem~\ref{thrm:generalization_barw} we obtain that with probability at least $1-\delta-\exp\{-c_0 n\}$,
\[
F(w_{S,k}) - F(\bar w) \le \mathcal{O}\left(\frac{\exp(8R)}{\lambda^2_{\min}(\Sigma)}\left(\frac{k \sigma^2 \log (p/\delta)}{n}\right)\right).
\]
When sample size $n$ is sufficiently large such that $\exp\{-c_0 n\}\le \min\left\{0.5, \frac{\delta_n}{4R^2}\right\}$, it follows from the expected generalization bounds in Theorem~\ref{thrm:generalization_barw} that
\[
\mathbb{E}_S \left[F(w_{S,k}) - F(\bar w)\right]\le \mathbb{E}_S \left[F(w_{S,k}) - F_S(w_{S,k})\right] \le \mathcal{O}\left(\frac{\exp(8R)}{\lambda^2_{\min}(\Sigma)}\left(\frac{k \sigma^2 \log (p/\delta)}{n}\right)\right).
\]
The concludes the proof.
\end{proof}

\subsection{Proof of Theorem~\ref{thrm:universe_generalizaion}}
\label{apdsect:proof_universe_generalizaion}

Here we prove Theorem~\ref{thrm:universe_generalizaion} as restated below.
\ERMBlackBoxBoundsUniform*

We need the following lemma which guarantees uniform convergence of $F_S(w)$ towards $F(w)$ for all $w$ when the loss function is Lipschitz continuous and the optimization is limited on a bounded domain.
\begin{lemma}\label{lemma:universe_support}
Assume that the domain of interest $\mathcal{W}\subset \mathbb{R}^p$ is bounded by $R$ and the loss function $\ell(w;\xi)$ is $G$-Lipschitz continuous with respect to $w$. Then for any $\delta\in(0,1)$ and $m\ge5$, there exists a universal constant $c_0$ such that the following bound holds with probability at least $1-\delta$ over the random draw of sample set $S$ for all $w\in \mathcal{W}$,
\[
\left|F(w) - F_S(w)\right| \le \mathcal{O}\left(\sqrt{\frac{GR(\log(c_0/\delta)+ m p\log(p))}{n}}\right),
\]
provided that
\[
\frac{4G(\log(c_0/\delta)+mp\log(p))}{R} \le n \le \frac{p^m(\log(c_0/\delta)+mp\log(p))}{GR}.
\]
\end{lemma}
\begin{proof}
As a subset of an $\ell_2$-sphere, for all $\epsilon \le R/2$, based on the result in~\cite{boroczky2003covering} we can bound the covering number of $\mathcal{W}$ with respect to the $\ell_2$-distance as
\[
\mathcal{N}(\epsilon, \mathcal{W}, \ell_2) = \mathcal{O}\left(p^{3/2}\log(p) \left(\frac{R}{\epsilon}\right)^p\right).
\]
Since the loss function $\ell(w;\xi)$ is $G$-Lipschitz continuous with respect to $w$, it can be verified that the covering number of the class of functions $\mathcal{L}=\left\{\xi \mapsto \ell(w;\xi)\mid w \in \mathcal{W}\right\}$ with respect to $\ell_\infty$-distance $\ell_\infty(\ell(w_1;\cdot),\ell(w_2;\cdot)):=\sup_{\xi} |\ell(w_1;\xi) - \ell(w_2;\xi)|$ is given by
\[
\mathcal{N}(\epsilon, \mathcal{L}, \ell_\infty) \le \mathcal{N}(\epsilon/G, \mathcal{L}, \ell_2) = \mathcal{O}\left(p^{3/2}\log(p) \left(\frac{GR}{\epsilon}\right)^p\right).
\]
Then based on a uniform bound from~\cite{pollard2012convergence} we know that
\[
\begin{aligned}
&\mathbb{P}\left(\sup_{w\in \mathcal{W}}|F(w)-F_S(w)|\ge \epsilon\right) \\
\le& \mathcal{O}\left(\mathcal{N}(\epsilon, \mathcal{L}, \ell_\infty)\exp\left(-\frac{n\epsilon^2}{GR}\right)\right) \le c_0 p^{3/2}\log(p)\left(\frac{GR}{\epsilon}\right)^p\exp\left(-\frac{n\epsilon^2}{GR}\right),
\end{aligned}
\]
where $c_0$ is a universal constant. To guarantee $\mathbb{P}\left(\sup_{w\in \mathcal{W}}|F(w)-F_S(w)|\ge \epsilon\right) \le\delta$, we need
\[
p^{3/2}\log(p)\left(\frac{GR}{\epsilon}\right)^p\exp\left(-\frac{n\epsilon^2}{GR}\right) \le \delta/c_0,
\]
or equivalently
\[
\frac{3}{2} \log(p) + \log\log(p) + p \log(GR) - p\log(\epsilon) - \frac{n\epsilon^2}{GR} \le \log\left(\frac{\delta}{c_0}\right).
\]
Setting $\epsilon= \sqrt{\frac{GR(\log(c_0/\delta)+ m p\log(p))}{n}}$, the above inequality holds if
\[
\begin{aligned}
& 3 \log(p) /2+ \log\log(p) + p \log(GR) \\
&- \frac{p}{2}\log\left(\frac{GR(\log(c_0/\delta)+ m p\log(p))}{n}\right) -mp\log(p) \le 0 \\
\overset{p\ge1}{\Leftarrow} & 2.5\log(p) + p \log(GR) - \frac{p}{2}\log\left(\frac{GR(\log(c_0/\delta)+ m p\log(p))}{n}\right) -mp\log(p) \le 0 \\
\overset{m\ge5}{\Leftarrow} & p \log(GR) - \frac{p}{2}\log\left(\frac{GR(\log(c_0/\delta)+ m p\log(p))}{n}\right) -\frac{mp}{2}\log(p) \le 0 \\
\Leftrightarrow & GR \le p^{m/2}\sqrt{\frac{GR(\log(c_0/\delta)+ m p\log(p))}{n}}  \Leftrightarrow n \le \frac{p^m(\log(c_0/\delta)+mp\log(p))}{GR}.
\end{aligned}
\]
The condition $\epsilon \le R/2$ leads to the requirement
\[
n\ge \frac{4G(\log(c_0/\delta)+mp\log(p))}{R}.
\]
This proves the desired result in the lemma.
\end{proof}

Based on this lemma, we can readily prove the main result in the theorem.
\begin{proof}[Proof of Theorem~\ref{thrm:universe_generalizaion}]

Let $\mathcal{J}=\{J \subseteq \{1,...,p\}: |J|= k\}$ be the set of index set of cardinality $k$. For any fixed supporting set $J \in \mathcal{J}$, by applying Lemma~\ref{lemma:universe_support} with $m=10$ we obtain that the following uniform convergence bound holds for all $w$ with $\supp(w) \subseteq J$ with probability at least $1-\delta$ over $S$:
\[
\left|F(w) - F_S(w)\right| \le \mathcal{O}\left(\sqrt{\frac{GR(\log(c_0/\delta) + 10k\log(k))}{n}}\right).
\]
For any $k$-sparse vector $w$ we always have $\supp(w) \in \mathcal{J}$. Then by union probability we get that with probability at least $1-\delta$, the following bound holds for all $w\in \mathcal{W}$ with $\|w\|_0\le k$:
\[
\left|F(w)  - F_S(w)\right| \le  \mathcal{O}\left(\sqrt{\frac{GR(10k\log(k)+ \log(|\mathcal{J}|) + \log(c_0/\delta))}{n}}\right).
\]
It remains to bound the cardinality $|\mathcal{J}|$. From~\cite[Lemma 2.7]{rigollet201518} we know $|\mathcal{J}|=\binom{p}{k}\le \left(\frac{ep}{k}\right)^k$, which then implies the desired generalization bound.
%To prove the excess risk bound, we bound $F(w_{S,k}) - F(\bar w)$ as
%\[
%\begin{aligned}
%F(w_{S,k}) - F(\bar w) =& F(w_{S,k}) - F_S(w_{S,k}) + F_S(w_{S,k}) -F_S(\bar w) + F_S(\bar w) - F(\bar w) \\
%\le& F(w_{S,k}) - F_S(w_{S,k}) + F_S(\bar w) - F(\bar w),
%\end{aligned}
%\]
%where we have used $F_S(w_{S,k}) \le F_S(\bar w)$. Since $\ell(w;\xi)\le M$, from Hoeffding inequality we know that with probability at least $1-\delta/2$,
%\[
%F_S(\bar w) - F(\bar w) \le M\sqrt{\frac{\log(2/\delta)}{2n}}.
%\]
%The first uniform generalization gap bound suggests that the following bound holds with probability at least $1-\delta/2$
%\[
%F(w_{S,k}) - F_S(w_{S,k}) \le  \mathcal{O}\left(\sqrt{\frac{GR(k\log(ep)+ \log(1/\delta))}{n}}\right).
%\]
%By union probability we get with probability at least $1-\delta$
%\[
%\begin{aligned}
%&F(w_{S,k}) - F(\bar w) \le F(w_{S,k}) - F_S(w_{S,k}) + F_S(\bar w) - F(\bar w) \\
%\le& \mathcal{O}\left(\sqrt{\frac{GR(k\log(ep)+ \log(1/\delta))}{n}} +  M\sqrt{\frac{\log(1/\delta)}{n}} \right).
%\end{aligned}
%\]
%This completes the proof.
\end{proof}

\newpage

\subsection{Proof of Proposition~\ref{thrm:uniform_stability}}
\label{apdsect:proof_uniform_stability}

In this subsection, we prove Proposition~\ref{thrm:uniform_stability} as restated below.
\ERMBlackBoxBoundsStability*

To prove this theorem, we need the following lemma from~\cite[Theorem 1.1]{feldman2019high} which gives a nearly tight generalization bound for uniformly stable learning algorithms.
\begin{lemma}\label{lemma:uniform_stability}
Let $A: \mathcal{X}^n \mapsto \mathcal{W}$ be a learning algorithm that has uniform stability $\gamma$  with respect to a loss function $\ell$. Then for any $\delta\in(0,1)$, the following generalization error bound holds with probability at least $1-\delta$ over $S$:
\[
\mathbb{E}_{\xi}\left[\ell(A(S);\xi)\right] \le \frac{1}{n}\sum_{i=1}^n \ell(A(S), \xi_i) + \mathcal{O}\left(\gamma\log(n)\log(n/\delta) + \sqrt{\frac{\log(1/\delta)}{n}}\right).
\]
\end{lemma}

For a given index set $J\subseteq \{1,...,p\}$, we consider the following restrictive estimator over $J$:
\[
w_{S\mid J} = \argmin_{w\in \mathcal{W}, \supp(w)\subseteq J} F_S(w).
\]
The following result is about the generalization gap of $w_{S\mid J}$ at any fixed $J$ with $|J|=k$.
\begin{lemma}\label{lemma:support_stability}
Assume that the loss function $\ell$ is smooth and $G$-Lipschitz continuous with respect to its first argument and $0\le\ell(w;\xi)\le M$ for all $w,\xi$. Suppose that $F_S$ is $\mu_k$-strongly convex with probability at least $1- \delta'_n$ over the random draw of $S$. Then for any fixed index set $J$ with cardinality $k$ and $\delta\in(0,1-\delta'_n)$, and for any $\lambda>0$, the following bound holds with probability at least $1-\delta-\delta'_n$ over the random draw of $S$,
\[
F(w_{S\mid J}) - F_S(w_{S\mid J}) \le \mathcal{O}\left(\frac{G^2}{\lambda n}\log(n)\log(n/\delta) + \sqrt{\frac{\log(1/\delta)}{n}} + \frac{\lambda G \sqrt{M}}{\mu_k\sqrt{\mu_k}}\right).
\]
\end{lemma}
\begin{proof}
For any given $\lambda>0$, let us consider the following defined $\ell_2$-regularized $\ell_0$-ERM estimator:
\[
w_{\lambda,S\mid J} := \argmin_{w\in \mathcal{W}, \supp(w)\subseteq J} \left\{F_{\lambda,S}(w):= F_S(w) + \frac{\lambda}{2}\|w\|^2\right\}.
\]
The reason for introducing the additional $\ell_2$-regularization term is to guarantee uniform stability of the hypothetical estimator $w_{\lambda,S\mid J}$. Based on the standard argument (see, e.g.,~\cite{shalev2009stochastic}) we can show that the optimal model $w_{\lambda,S\mid J}$ has uniform stability $\gamma=\frac{4G^2}{\lambda n}$. Here we choose to provide the proof details in order to make our analysis self-contained. Let $S^{(i)}$ be a sample set that is identical to $S$ except that one of the $\xi_i$ is replaced by another random sample $\xi'_i$. Then we can show that
\[
\begin{aligned}
&F_{\lambda,S}(w_{\lambda, S^{(i)}\mid J}) - F_{\lambda,S}(w_{\lambda,S\mid J}) \\
=& \frac{1}{n}\sum_{j\neq i} \left(\ell(w_{\lambda, S^{(i)}\mid J};\xi_j) - \ell(w_{\lambda,S\mid J};\xi_j) \right) + \frac{1}{n} \left(\ell(w_{\lambda, S^{(i)}\mid J};\xi_i) - \ell( w_{\lambda,S\mid J};\xi_i)\right)\\
& + \frac{\lambda}{2}\|w_{\lambda, S^{(i)}\mid J}\|^2 -  \frac{\lambda}{2}\|w_{\lambda,S\mid J}\|^2\\
=& F_{\lambda,S^{(i)}}(w_{\lambda, S^{(i)}\mid J}) - F_{\lambda,S^{(i)}}(w_{\lambda,S\mid J}) + \frac{1}{n} \left(\ell(w_{\lambda, S^{(i)}\mid J};\xi_i) - \ell(w_{\lambda,S\mid J};\xi_i) \right) \\
& - \frac{1}{n} \left(\ell(w_{\lambda, S^{(i)}\mid J};\xi'_i) - \ell(w_{\lambda,S\mid J};\xi'_i) \right) \\
\le& \frac{1}{n} \left|\ell(w_{\lambda, S^{(i)}\mid J};\xi_i) - \ell(w_{\lambda,S\mid J};\xi_i) \right| + \frac{1}{n} \left|\ell(w_{\lambda, S^{(i)}\mid J};\xi'_i) - \ell(w_{\lambda,S\mid J};\xi'_i)\right| \\
\le& \frac{2G}{n} \|w_{\lambda, S^{(i)}\mid J} - w_{\lambda,S\mid J}\|,
\end{aligned}
\]
where we have used the optimality of $w_{\lambda, S^{(i)}\mid J}$ with respect to $F_{\lambda,S^{(i)}}(w)$ and the Lipschitz continuity of the loss function $\ell(w;\xi)$. Since $F_{\lambda,S}$ is $\lambda$-strongly convex and $w_{\lambda,S\mid J}$ is optimal for $F_{\lambda,S}(w)$ over the supporting set $J$, we have
\[
F_{\lambda,S}(w_{\lambda, S^{(i)}\mid J}) \ge F_{\lambda,S}(w_{\lambda,S\mid J}) + \frac{\lambda}{2}\|w_{\lambda, S^{(i)}\mid J} - w_{\lambda,S\mid J}\|^2.
\]
By combing the preceding two inequalities we arrive at
\[
\|w_{\lambda, S^{(i)}\mid J} - w_{\lambda,S\mid J}\| \le \frac{4G}{\lambda n}.
\]
Consequently from the Lipschitz continuity of $\ell$ we have that for any sample $\xi$
\[
|\ell(w_{\lambda, S^{(i)}\mid J};\xi ) - \ell(w_{\lambda,S\mid J};\xi) | \le G\|w_{\lambda,S\mid J}^{(i)} - w_{\lambda,S\mid J}\| \le \frac{4G^2}{\lambda n}.
\]
This confirms that the optimal model $w_{\lambda,S\mid J}$ has uniform stability $\gamma=\frac{4G^2}{\lambda n}$. By invoking Lemma~\ref{lemma:uniform_stability} we obtain that with probability at least $1-\delta$ over random draw of $S$,
\begin{equation}\label{inequat:lemma_support_stability_key_1}
F(w_{\lambda, S\mid J}) - F_S(w_{\lambda, S \mid J}) \le \mathcal{O}\left(\frac{G^2}{\lambda n}\log(n)\log(n/\delta) + \sqrt{\frac{\log(1/\delta)}{n}}\right).
\end{equation}
Next, we show how to bound the estimator difference $\|w_{S\mid J} - w_{\lambda, S\mid J}\|$. The strong convexity assumption of $F_S$ implies that  the following bound holds with probability at least $1-\delta'_n$ over $S$:
\[
\|\nabla_J F_{\lambda, S}(w_{S\mid J}) - \nabla_J F_{\lambda, S}(w_{\lambda, S\mid J})\| \ge (\mu_k+\lambda)\|w_{S\mid J} - w_{\lambda, S\mid J}\|,
\]
where the notation $\nabla_J F$ denotes the restriction of gradient $\nabla F$ over $J$. The optimality of $w_{\lambda,S\mid J}$ and $w_{S\mid J}$ over $J$ implies that
\[
\nabla_J F_{\lambda, S}(w_{\lambda, S\mid J})=0, \quad \nabla_J F_{\lambda, S}(w_{S\mid J}) = \nabla_J F_S(w_{S\mid J}) + \lambda w_{S\mid J}=\lambda w_{S\mid J}.
\]
In the meanwhile, since $\ell(\cdot;\cdot)\le M$, we must have the following holds with probability at least $1-\delta'_n$ over $S$:
\[
2M \ge F_S(0) - F_S(w_{S\mid J}) \ge \frac{\mu_{k}}{2}\|w_{S\mid J}\|^2,
\]
which leads to $\|w_{S\mid J}\|\le 2\sqrt{M/\mu_{k}}$. Then it follows readily from the previous two bounds that
\[
\|w_{S\mid J} - w_{\lambda, S\mid J}\| \le \frac{\lambda}{\mu_k + \lambda}\|w_{S \mid J}\| \le \frac{2\lambda \sqrt{M}}{\sqrt{\mu_k}(\mu_k + \lambda)}\le \frac{2\lambda \sqrt{M}}{\mu_k\sqrt{\mu_k}}.
\]
Since the loss function is $G$-Lipschitz continuous, the following is then valid with probability at least $1-\delta'_n$ over the random draw of $S$:
\[
\begin{aligned}
& F(w_{S\mid J}) - F_S(w_{S \mid J}) \\
\le & F(w_{\lambda, S\mid J}) - F_S(w_{\lambda, S \mid J}) + |F_S(w_{S\mid J})-F_S(w_{\lambda, S\mid J})| + |F(w_{S\mid J})-F(w_{\lambda, S\mid J})| \\
\le& F(w_{\lambda, S\mid J}) - F_S(w_{\lambda, S \mid J}) + 2G\|w_{S\mid J} - w_{\lambda, S\mid J}\| \\
\le& F(w_{\lambda, S\mid J}) - F_S(w_{\lambda, S \mid J}) + \frac{4\lambda G \sqrt{M}}{\mu_k\sqrt{\mu_k}}.
\end{aligned}
\]
In view of the above bound and the bound in~\eqref{inequat:lemma_support_stability_key_1} we get that for any fixed index set $J$, with probability at least $1-\delta-\delta'_n$ over $S$,
\[
F(w_{S\mid J}) - F_S(w_{S \mid J}) \le \mathcal{O}\left(\frac{G^2}{\lambda n}\log(n)\log(n/\delta) + \sqrt{\frac{\log(1/\delta)}{n}} + \frac{\lambda G \sqrt{M}}{\mu_k\sqrt{\mu_k}}\right).
\]
The proof is concluded.
\end{proof}
Now we are in the position to prove the main theorem.
\begin{proof}[Proof of Proposition~\ref{thrm:uniform_stability}]
Let $\mathcal{J}=\{J \subseteq \{1,...,p\}: |J|= k\}$ be the set of index set of cardinality $k$. It is standard to know $|\mathcal{J}|=\binom{p}{k}\le \left(\frac{ep}{k}\right)^k$ (see, e.g.,~\cite[Lemma 2.7]{rigollet201518}). For any random sample set $S$, from the optimality of $w_{S,k}$ we always have $w_{S,k} \in \{w_{S\mid J}: J\in \mathcal{J}\}$. By assumption $\delta'_n \le \frac{\delta}{2|\mathcal{J}|}$. For each $J\in \mathcal{J}$, by applying Lemma~\ref{lemma:support_stability} we can show that with probability at least $1-\frac{\delta}{|\mathcal{J}|}$ over $S$, the generalization gap satisfies $F(w_{S\mid J})- F_S(w_{S\mid J})\le$
\[
 \mathcal{O}\left( \frac{G^2}{\lambda n}\log(n)(\log(n/\delta)+\log(|\mathcal{J}|))+\sqrt{\frac{\log(1/\delta)+ \log(|\mathcal{J}|)}{n}} + \frac{\lambda G \sqrt{M}}{\mu_k\sqrt{\mu_k}} \right).
\]
Then by union probability we get that with probability at least $1-\delta$, $F(w_{S,k})- F_S(w_{S,k}) \le $
\[
\mathcal{O}\left( \frac{G^2}{\lambda n}\log(n)(\log(n/\delta)+k\log(p/k))+\sqrt{\frac{\log(1/\delta)+ k\log(p/k)}{n}}+ \frac{\lambda G \sqrt{M}}{\mu_k\sqrt{\mu_k}}\right).
\]
By setting $\lambda=\sqrt{\frac{\mu_k^{1.5}\log(n)(\log(n/\delta)+k\log(p/k))}{nM^{0.5}}}$ we obtain the first desired bound.

To prove the excess risk bound, we bound $F(w_{S,k}) - F(\bar w)$ as
\[
\begin{aligned}
F(w_{S,k}) - F(\bar w) =& F(w_{S,k}) - F_S(w_{S,k}) + F_S(w_{S,k}) -F_S(\bar w) + F_S(\bar w) - F(\bar w) \\
\le& F(w_{S,k}) - F_S(w_{S,k}) + F_S(\bar w) - F(\bar w) ,
\end{aligned}
\]
where we have used $F_S(w_{S,k}) \le F_S(\bar w) $. Since $\ell(w;\xi)\le M$, from Hoeffding inequality we know that with probability at least $1-\delta/2$,
\[
F_S(\bar w) - F(\bar w) \le M\sqrt{\frac{\log(2/\delta)}{2n}}.
\]
Based on the previous generalization gap bound and by union probability we get with probability at least $1-\delta$
\[
\begin{aligned}
&F(w_{S,k}) - F(\bar w) \le F(w_{S,k}) - F_S(w_{S,k}) + F_S(\bar w) - F(\bar w) \le\\
&  \mathcal{O}\left(\frac{G^{3/2}M^{1/4}}{\mu_k^{3/4}}\sqrt{\frac{\log(n)(\log(n/\delta)+k\log(p/k))}{n}} + M\sqrt{\frac{\log(1/\delta)}{n}} \right).
\end{aligned}
\]
This completes the proof.
\end{proof}

\subsection{Proof of Theorem~\ref{thrm:generalizaion_support_recovery_high}}
\label{apdsect:proof_theorem_generalizaion_support_recovery_high}

In this subsection, we prove Theorem~\ref{thrm:generalizaion_support_recovery_high} as restated below.
\ERMBlackBoxBoundsUniformSupport*

To prove the main result in the theorem, we first need to prove the following lemma which basically provides a sufficient condition to guarantee the support recovery stability of $\ell_0$-ERM.
\begin{lemma}\label{lemma:support_recovery_bar_w}
Suppose that $F_S$ is $\mu_{2k}$-strongly convex with probability at least $1- \delta'_n$. Assume that $\|\nabla \ell(w;\cdot)\|\le G$. Suppose that there exists a $k$-sparse vector $\tilde w$ such that
\[
\tilde w_{\min}> \frac{2\sqrt{2k}\|\nabla F(\tilde w)\|_\infty}{\mu_{2k}} + \frac{2G}{\mu_{2k}} \sqrt{\frac{2k\log(p/\delta)}{2n}}
\]
for some $\delta\in (0, 1-\delta'_n)$. Then the support recovery $\supp(w_{S,k}) = \supp(\tilde w)$ holds with probability at least $1-\delta - \delta'_n$.
\end{lemma}
\begin{proof}
Let us consider a fixed $\tilde w$. Since $\|\nabla \ell(w;\cdot)\|\le G$, from Hoeffding concentration bound we know that with probability at least $1-\delta$ over $S$,
\[
\|\nabla F_S(\tilde w) - \nabla F(\tilde w)\| \le G \sqrt{\frac{\log(p/\delta)}{2n}}.
\]
Then with probability at least $1-\delta$,
\[
\begin{aligned}
&\|\nabla F_S(\tilde w)\|_\infty \le \|\nabla F(\tilde w)\|_\infty + \|\nabla F_S(\tilde w) - \nabla F(\tilde w)\|_\infty \\
\le& \|\nabla F(\tilde w)\|_\infty + \|\nabla F_S(\tilde w) - \nabla F(\tilde w)\|\le \|\nabla F(\tilde w)\|_\infty + G \sqrt{\frac{\log(1/\delta)}{2n}}.
\end{aligned}
\]
Consequently from the condition in the theorem we can show that with probability at least $1-\delta$,
\[
\begin{aligned}
\tilde w_{\min} >& \frac{2\sqrt{2k}\|\nabla F(\tilde w)\|_\infty}{\mu_{2k}} + \frac{2G}{\mu_{2k}} \sqrt{\frac{2k\log(1/\delta)}{2n}}\\
\ge& \frac{2\sqrt{2k}\|\nabla F_S(\tilde w)\|_\infty}{\mu_{2k}} - \frac{2G}{\mu_{2k}} \sqrt{\frac{2k\log(1/\delta)}{2n}} + \frac{2G}{\mu_{2k}} \sqrt{\frac{2k\log(1/\delta)}{2n}}\\
=& \frac{2\sqrt{2k}\|\nabla F_S(\tilde w)\|_\infty}{\mu_{2k}}.
\end{aligned}
\]
Finally, since $F_S(w_{S,k}) \le F_S(\tilde w)$ and $F_S$ is $\mu_{2k}$-strongly convex with probability at least $1- \delta'_n$, by invoking Lemma~\ref{lemma:estimation_error} with $w=w_{S,k}, w'=\tilde w$ and $\epsilon = 0$, and using union probability argument we get that $\supp(w_{S,k}) = \supp(\tilde w)$ (note that $w_{S,k}$ and $\tilde w$ are both $k$-sparse vectors) holds with probability at least $1-\delta-\delta'_n$.
\end{proof}
\begin{remark}
The main message conveyed by this lemma is that with additional conditions imposed on the signal strength of certain underlying sparse vector $\tilde w$, the supporting set recovered by $\ell_0$-ERM is exactly that of $\tilde w$ with high probability. This simple result essentially guarantees the stability of support recovery for $\ell_0$-ERM.
\end{remark}

Now we are ready to prove the main result.
\begin{proof}[Proof of Theorem~\ref{thrm:generalizaion_support_recovery_high}]
\emph{Part(a):} Let us consider $S^{(i)}$ which is identical to $S$ except that one of the $\xi_i$ is replaced by another random sample $\xi'_i$. Then we can show that
\[
\begin{aligned}
&F_{S}(w_{S^{(i)},k}) - F_{S}(w_{S,k}) \\
=& \frac{1}{n}\sum_{j\neq i} \left(\ell(w_{S^{(i)},k};\xi_j) - \ell(w_{S,k};\xi_j) \right) + \frac{1}{n} \left(\ell(w_{S^{(i)},k};\xi_i) - \ell( w_{S,k};\xi_i)\right) \\
=& F_{S^{(i)}}(w_{S^{(i)},k}) - F_{S^{(i)}}(w_{S,k}) + \frac{1}{n} \left(\ell(w_{S^{(i)},k};\xi_i) - \ell(w_{S,k};\xi_i) \right) \\
&- \frac{1}{n} \left(\ell(w_{S^{(i)},k};\xi'_i) - \ell(w_{S,k};\xi'_i) \right) \\
\le& \frac{1}{n} \left|\ell(w_{S^{(i)},k};\xi_i) - \ell(w_{S,k};\xi_i) \right| + \frac{1}{n} \left|\ell(w_{S^{(i)},k};\xi'_i) - \ell(w_{S,k};\xi'_i)\right| \\
\le& \frac{2G}{n} \|w_{S^{(i)},k} - w_{S,k}\|,
\end{aligned}
\]
where we have used the optimality of $w_{S^{(i)},k}$ with respect to $F_{S^{(i)}}(w)$ and the Lipschitz continuity of the loss function $\ell$ with respect to its first argument. Note that $S$ and $S^{(i)}$ are both i.i.d. samples. Then based on the condition on $\tilde w_{\min}$ and Lemma~\ref{lemma:support_recovery_bar_w} we know that $\supp(w_{S,k}) = \supp(\tilde w)$ and $\supp(w_{S^{(i)},k}) = \supp(\tilde w)$ hold (separately) with probability at least $1-\delta/2-\delta'_n$. Therefore the following event occurs with probability at least $1-\delta-2\delta'_n$ over $\{S, \xi'_i\}$:
\[
\text{$F_S$ is $\mu_{2k}$-strongly convex}, \ \ \ \supp(w_{S,k}) = \supp(w_{S^{(i)},k}).
\]
Now we assume that the above event occurs. Since $w_{S,k}$ is optimal for $F_{S}(w)$ over the supporting set of $S$, we have
\[
F_{S}(w_{S^{(i)},k}) \ge F_{S}(w_{S,k}) + \frac{\mu_{2k}}{2}\|w_{S^{(i)},k} - w_{S,k}\|^2.
\]
By combing the preceding two inequalities we arrive at
\[
\|w_{S^{(i)},k} - w_{S,k}\| \le \frac{4G}{\mu_{2k} n}.
\]
Let us consider a random variable defined by
\[
\gamma^{(i)}_{S}(\xi_i):=  |\ell(w_{S^{(i)},k};\xi_i) - \ell(w_{S,k};\xi_i)|
\]
Consequently from the Lipschitz continuity of the loss function $\ell$ we have that the following holds with probability at least $1-\delta-2\delta'_n$ over $\{S, \xi'_i\}$:
\[
\gamma^{(i)}_{S}(\xi_i) = |\ell(w_{S^{(i)},k};\xi_i) - \ell(w_{S,k};\xi_i) | \le G\|w_{S^{(i)},k} - w_{S,k}\| \le \frac{4G^2}{\mu_{2k} n}.
\]
In the meanwhile, the bounding assumption on $\ell$ implies
\[
\gamma^{(i)}_{S}(\xi_i) \le |\ell(w_{S^{(i)},k};\xi_i)| + |\ell(w_{S,k};\xi_i) | \le 2M.
\]
Let us consider a random indication variable $Y^{(i)}_S(\xi_i)$ defined by $Y^{(i)}_S(\xi_i)=1$ if $\gamma^{(i)}_{S}(\xi_i)\le \frac{4G^2}{\mu_{2k} n}$, and $Y^{(i)}_S(\xi_i)=0$ otherwise. Then we can bound the expected value of $\gamma^{(i)}_S(\xi_i)$ as follows:
\begin{equation}\label{inequat:uniform_strong_thrm_key1}
\begin{aligned}
&\mathbb{E}_{S\cup \{\xi'_i\}}\left[\gamma^{(i)}_S(\xi_i)\right] = \mathbb{E}_{Y^{(i)}_S(\xi_i)}\mathbb{E}_{S\cup \{\xi'_i\}}\left[\gamma^{(i)}_S(\xi_i)\mid Y^{(i)}_S(\xi_i)\right] \\
=& \mathbb{E}_{S\cup \{\xi'_i\}}\left[\gamma^{(i)}_S(\xi_i)\mid Y^{(i)}_S(\xi_i)=1 \right] \mathbb{P}\left(Y^{(i)}_S(\xi_i)=1 \right) \\
& + \mathbb{E}_{S\cup \{\xi'_i\}}\left[\gamma^{(i)}_S(\xi_i)\mid Y^{(i)}_S(\xi_i)=0 \right] \mathbb{P}\left(Y^{(i)}_S(\xi_i)=0 \right) \\
\le& \frac{4G^2}{\mu_{2k} n} + 2M \mathbb{P}\left(Y^{(i)}_S(\xi_i)=0 \right) \le \frac{4G^2}{\mu_{2k} n} + 2M(\delta + 2\delta'_n),
\end{aligned}
\end{equation}
where in the last inequality we have used $\mathbb{P}\left(Y^{(i)}_S(\xi)=0 \right)\le \delta + 2\delta'_n$.

Note again that $S$ and $S^{(i)}$ are both i.i.d. samples of the data distribution $D$. It follows that
\[
\mathbb{E}_{S}\left[ F(w_{S,k})\right] = \mathbb{E}_{S^{(i)}}\left[ F(w_{S^{(i)},k})\right] = \mathbb{E}_{S^{(i)}\cup \{\xi_i\}}\left[ \ell(w_{S^{(i)},k}; \xi_i)\right].
\]
Since the above holds for all $i=1,...,n$, we can show that
\[
\mathbb{E}_{S}\left[ F(w_{S,k}) \right] = \frac{1}{n}\sum_{i=1}^n \mathbb{E}_{S^{(i)}\cup \{\xi_i\}}\left[ \ell(w_{S^{(i)},k};\xi_i)\right] = \frac{1}{n}\sum_{i=1}^n \mathbb{E}_{S\cup \{\xi'_i\}}\left[ \ell(w_{S^{(i)},k};\xi_i)\right].
\]
Concerning the empirical case, we can see that
\[
\mathbb{E}_{S}\left[ F_S(w_{S,k})\right] =  \frac{1}{n}\sum_{i=1}^n \mathbb{E}_S\left[ \ell(w_{S,k};\xi_i)\right]  = \frac{1}{n}\sum_{i=1}^n \mathbb{E}_{S\cup \{\xi'_i\}}\left[ \ell(w_{S,k};\xi_i)\right].
\]
By combining the above two inequalities we get
\[
\begin{aligned}
&\left|\mathbb{E}_{S}\left[ F(w_{S,k}) - F_S(w_{S,k}) \right] \right| = \left| \frac{1}{n}\sum_{i=1}^n \mathbb{E}_{S\cup \{\xi'_i\}} \left[\ell(w_{S^{(i)},k}; \xi_i) - \ell(w_{S,k}; \xi_i)\right] \right| \\
\le& \frac{1}{n}\sum_{i=1}^n \mathbb{E}_{S\cup \{\xi'_i\}} \left[\left|\ell(w_{S^{(i)},k};\xi_i) - \ell(w_{S,k}; \xi_i)\right|\right] = \frac{1}{n}\sum_{i=1}^n \mathbb{E}_{S\cup \{\xi'_i\}} \left[\gamma^{(i)}_S(\xi_i)\right] \\
\le& \frac{4G^2}{\mu_{2k} n} + 2M(\delta + 2\delta'_n),
\end{aligned}
\]
where in the last inequality we have used~\eqref{inequat:uniform_strong_thrm_key1}.

In order to prove the excess risk bound, we bound $F(w_{S,k}) - F(\bar w)$ as
\[
\begin{aligned}
F(w_{S,k}) - F(\bar w) =& F(w_{S,k}) - F_S(w_{S,k}) + F_S(w_{S,k}) -F_S(\bar w) + F_S(\bar w) - F(\bar w) \\
\le& F(w_{S,k}) - F_S(w_{S,k}) + F_S(\bar w) - F(\bar w),
\end{aligned}
\]
where we have used $F_S(w_{S,k}) \le F_S(\bar w)$. Thus
\[
\begin{aligned}
\mathbb{E}_{S}\left[F(w_{S,k}) - F(\bar w)\right] \le & \mathbb{E}_{S}\left[F(w_{S,k}) - F_S(w_{S,k}) + F_S(\bar w) - F(\bar w)\right] \\
=& \mathbb{E}_{S}\left[F(w_{S,k}) - F_S(w_{S,k}) \right] \le \frac{4G^2}{\mu_{2k} n} + 2M(\delta + 2\delta'_n).
\end{aligned}
\]
This completes the proof of part(a) .

\emph{Part(b):} To ease notation, we abbreviate
\[
C_{\lambda,n}:=\frac{G^2}{\lambda n}\log(n)\log(n/\delta) + \sqrt{\frac{\log(1/\delta)}{n}} + \frac{\lambda G \sqrt{M}}{\mu_k\sqrt{\mu_k}}.
\]
Denote $\tilde J=\supp(\tilde w)$ and $w_{S\mid \tilde J} = \argmin_{w\in \mathcal{W}, \supp(w)\subseteq \tilde J} F_S(w)$. Let us define the following three events associated with the sample set $S$:
\[
\mathcal{E}_1:\left\{S \in \mathcal{X}^n: F(w_{S,k})  - F_S(w_{S,k})\le \mathcal{O}\left(C_{\lambda,n}\right)\right\},
\]
\[
\mathcal{E}_2:\left\{S \in \mathcal{X}^n: F(w_{S\mid \tilde J})  - F_S(w_{S\mid \tilde J})\le \mathcal{O}\left(C_{\lambda,n}\right), \right\},
\]
and
\[
\mathcal{E}_3:=\left\{S \in \mathcal{X}^n: \supp(w_{S,k}) = \supp(\tilde w)\right\}.
\]
It is easy to verify that $\mathcal{E}_1 \cap \mathcal{E}_3 \supseteq \mathcal{E}_2 \cap \mathcal{E}_3$. Indeed, for any $S\in \mathcal{E}_2 \cap \mathcal{E}_3$, it follows immediately that $S\in \mathcal{E}_1$ and thus $S\in \mathcal{E}_1 \cap \mathcal{E}_3$. Since the loss function $\ell$ is differentiable and $G$-Lipschitz continuous with its first argument, we have that $\|\nabla \ell(w;\cdot)\|\le G$. Then based on the condition on $\bar w_{\min}$ and Lemma~\ref{lemma:support_recovery_bar_w} we can show that the following probability bound holds:
\[
\mathbb{P}\left(\mathcal{E}_3 \right) \ge 1 - \delta/2 - \delta'_n.
\]
In the meanwhile, by invoking Lemma~\ref{lemma:support_stability} over the supporting set $\supp(\bar w)$ we obtain that
\[
\mathbb{P}\left(\mathcal{E}_2 \right) \ge 1 - \delta/2 .
\]
Then, we can derive
\[
\mathbb{P}(\mathcal{E}_1) \ge\mathbb{P}(\mathcal{E}_1 \cap \mathcal{E}_3) \ge \mathbb{P}(\mathcal{E}_2 \cap \mathcal{E}_3) \ge 1 - \mathbb{P}(\overline{\mathcal{E}}_2) - \mathbb{P}(\overline{\mathcal{E}}_3) \ge 1 - \delta - \delta_n'.
\]
The desired generalization gap bound follows by setting $\lambda=\sqrt{\frac{\mu_{2k}^{1.5}\log(n)\log(n/\delta)}{n\sqrt{M}}}$ in $C_{\lambda,n}$.

To prove the excess risk bound, we bound $F(w_{S,k}) - F(\bar w)$ as
\[
\begin{aligned}
F(w_{S,k}) - F(\bar w) =& F(w_{S,k}) - F_S(w_{S,k}) + F_S(w_{S,k}) -F_S(\bar w) + F_S(\bar w) - F(\bar w) \\
\le& F(w_{S,k}) - F_S(w_{S,k}) + F_S(\bar w) - F(\bar w),
\end{aligned}
\]
where we have used $F_S(w_{S,k}) \le F_S(\bar w)$. Since $\ell(w;\xi)\le M$, from Hoeffding inequality we know that with probability at least $1-\delta/2$,
\[
F_S(\bar w) - F(\bar w) \le M\sqrt{\frac{\log(2/\delta)}{2n}}.
\]
The established generalization gap bound suggests that the following bound holds with probability at least $1-\delta/2-\delta'_n$
\[
F(w_{S,k}) - F_S(w_{S,k}) \le  \mathcal{O}\left(C_{\lambda,n}\right).
\]
By union probability we get with probability at least $1-\delta-\delta'_n$
\[
F(w_{S,k}) - F(\bar w) \le F(w_{S,k}) - F_S(w_{S,k}) + F_S(\bar w) - F(\bar w) \le \mathcal{O}\left(C_{\lambda,n} +  M\sqrt{\frac{\log(1/\delta)}{n}} \right).
\]
This concludes the proof.
\end{proof}

\section{Proof of the Generalization Results for the IHT Algorithm}
\label{append:proof_iht}

\subsection{Proof of Corollary~\ref{corol:generalization_barw_iht}}
\label{apdsect:proof_generalization_barw_iht}

In this subsection, we prove Corollary~\ref{corol:generalization_barw_iht} as restated below.
\IHTWhiteBoxBounds*
\begin{proof}
The proof largely mimics that of Theorem~\ref{thrm:generalization_barw}, with proper adaptation to the excess risk bounds of IHT. Here we still present the full proof for completeness purpose. From Lemma~\ref{lemma:sub_Gaussian_bounds} we know that it holds with probability at least $1-\delta$ that
\[
\|\nabla F_S(\bar w)\|_\infty \le \sigma\sqrt{\frac{2\log(p/\delta)}{n}}.
\]
Since by assumption $F_S(w)$ is $L$-strongly smooth and $\mu_{3k}$-strongly convex with probability at least $1- \delta'_n$, Lemma~\ref{lemma:convergence_iht} then shows that $F(w^{(t)}_{S,k}) - F_S(\bar w) \le \epsilon$ with probability at least $1- \delta'_n$ provided that $t\ge\mathcal{O}(L/\mu_{3k}\log(1/\epsilon))$. By invoking Lemma~\ref{lemma:estimation_error} we obtain that with probability at least $1- \delta'_n$,
\[
\|w^{(t)}_{S,k} - \bar w\|^2 \le \frac{16k\|\nabla F_S(\bar w)\|^2_\infty}{\mu_{2k}^2} + \frac{4\epsilon}{\mu_{2k}}\le \frac{16k\|\nabla F_S(\bar w)\|^2_\infty}{\mu_{3k}^2} + \frac{4\epsilon}{\mu_{3k}}.
\]
Based on the smoothness of $F$ and $\nabla F(\bar w)=0$, by union probability we get that with probability at least $1-\delta - \delta'_n$
\[
\begin{aligned}
&F(w^{(t)}_{S,k}) - F(\bar w) \\
\le& \frac{L}{2}\|w^{(t)}_{S,k} - \bar w\|^2 \le \frac{8kL\|\nabla F_S(\bar w)\|^2_\infty}{\mu_{3k}^2} + \frac{2L\epsilon}{\mu_{3k}} \le \frac{16L}{\mu_{3k}^2}\left(\frac{k\sigma^2\log (p/\delta)}{n}\right) + \frac{2L\epsilon}{\mu_{3k}}.
\end{aligned}
\]
By setting $\epsilon = \frac{1}{\mu_{3k}}\left(\frac{k\sigma^2\log (p/\delta)}{n}\right)$ we obtain the desired high probability bound for excess risk.

Next, we bound the generalization gap in expectation. Let us consider the following decomposition of generalization gap:
\[
F(w^{(t)}_{S,k}) - F_S(w^{(t)}_{S,k}) = \underbrace{F(w^{(t)}_{S,k}) - F(\bar w)}_{A_1} + \underbrace{F(\bar w) - F_S(\bar w)}_{A_2} + \underbrace{F_S(\bar w) - F_S(w^{(t)}_{S,k})}_{A_3}.
\]
We first bound the term $A_1$. From the smoothness of $\ell(\cdot,\cdot)$ and $\nabla F(\bar w) =0$ we get
\[
\begin{aligned}
\left|F(w^{(t)}_{S,k}) - F(\bar w)\right| \le& \left|\langle \nabla F(\bar w), w^{(t)}_{S,k} - \bar w  \rangle\right| + \frac{L}{2}\|w^{(t)}_{S,k} - \bar w\|^2 = \frac{L}{2}\|w^{(t)}_{S,k} - \bar w\|^2.
\end{aligned}
\]
Similarly we can bound the term $A_3$. Indeed, let $\bar I = \supp(\bar w)$ and $I^{(t)}_S=\supp(w^{(t)}_{S,k})$. Again, from the smoothness of $\ell(w,\cdot)$ we get
\[
\begin{aligned}
&\left|F_S(\bar w) - F_S(w^{(t)}_{S,k})\right|\\
=& \left|F_S(w^{(t)}_{S,k}) - F_S(\bar w) \right|\le \left|\langle \nabla F_S(\bar w), w^{(t)}_{S,k} - \bar w  \rangle\right| + \frac{L}{2}\|w^{(t)}_{S,k} - \bar w\|^2 \\
\le& \|\nabla_{\bar I \cup I^{(t)}_S} F_S(\bar w)\|\|w^{(t)}_{S,k} - \bar w\| + \frac{L}{2}\|w^{(t)}_{S,k} - \bar w\|^2 \\
\le& \sqrt{2k}\|\nabla F_S(\bar w)\|_\infty \|\|w^{(t)}_{S,k} - \bar w\| + \frac{L}{2}\|w^{(t)}_{S,k} - \bar w\|^2 \\
\le& \frac{k}{L_{2k}}\|\nabla F_S(\bar w)\|^2_\infty + L \|w^{(t)}_{S,k} - \bar w\|^2.
\end{aligned}
\]
From Lemma~\ref{lemma:sub_Gaussian_bounds} we have $\mathbb{E}_S\left[\|\nabla F_S(\bar w)\|^2_\infty\right]\le \delta_n$. Therefore,
\[
\begin{aligned}
&\left|\mathbb{E}_S[F(w^{(t)}_{S,k}) - F_S(w^{(t)}_{S,k})]\right| \\
=& \left|\mathbb{E}_S[F(w^{(t)}_{S,k}) - F(\bar w)] + \mathbb{E}_S[F(\bar w) - F_S(\bar w) ] + \mathbb{E}_S[F_S(\bar w) - F_S(w^{(t)}_{S,k})]\right| \\
\overset{\zeta_1}{=}& \left|\mathbb{E}_S[F(w^{(t)}_{S,k}) - F(\bar w)] + \mathbb{E}_S[F_S(\bar w) - F_S(w^{(t)}_{S,k})]\right| \\
\le& \mathbb{E}_S\left[\left|F(w^{(t)}_{S,k}) - F(\bar w)\right|\right] + \mathbb{E}_S\left[\left|F_S(\bar w) - F_S(w^{(t)}_{S,k})\right|\right] \\
\le& \frac{L}{2}\mathbb{E}_S\left[\|w^{(t)}_{S,k} - \bar w\|^2\right] + \frac{k}{L}\mathbb{E}_S\left[\|\nabla F_S(\bar w)\|^2_\infty\right] + L \mathbb{E}_S\left[\|w^{(t)}_{S,k} - \bar w\|^2\right] \\
\le& \frac{3L}{2}\mathbb{E}_S\left[\|w^{(t)}_{S,k} - \bar w\|^2\right] + \frac{k}{L}\mathbb{E}_S\left[\|\nabla F_S(\bar w)\|^2_\infty\right]\\
\overset{\zeta_2}{\le}& \frac{3L}{2} \left(\frac{32k}{\mu^2_{2k}}+1\right)\delta_n+ \frac{6L\epsilon}{\mu_{2k}}  + \frac{k}{L} \delta_n = \left(\frac{(48L^2 + \mu_{2k}^2)k}{L\mu^2_{2k}} + \frac{3L}{2}\right)\delta_n + \frac{6L\epsilon}{\mu_{2k}},
\end{aligned}
\]
where in ``$\zeta_1$'' we have used $\mathbb{E}_S[F_S(\bar w)] = F(\bar w)$ and ``$\zeta_2$'' we have used Lemma~\ref{lemma:L2_error_expectation}. Similarly, we can bound the excess risk in expectation as:
\[
\mathbb{E}_S \left[F(w^{(t)}_{S,k}) - F(\bar w)\right] \le  \frac{L}{2}\mathbb{E}_S\left[\|w^{(t)}_{S,k} - \bar w\|^2\right] \le \frac{L}{2} \left(\frac{32k}{\mu^2_{2k}}+1\right)\delta_n + \frac{2L\epsilon}{\mu_{2k}}.
\]
The desired bounds in expectation follow immediately by noting $\mu_{2k}\ge\mu_{3k}$ and setting $\epsilon = \frac{1}{\mu_{3k}}\left(\frac{k\sigma^2\log (p/\delta)}{n}\right)$. This completes the proof.
\end{proof}

\newpage

\subsection{Proof of Corollary~\ref{corol:uniform_stability_iht}}
\label{apdsect:proof_uniform_stability_iht}

Here we prove Corollary~\ref{corol:uniform_stability_iht} as restated below.
\IHTBlackBoxBoundsStability*

\begin{proof}
Let $\mathcal{J}=\{J \subseteq \{1,...,p\}: |J|= k\}$ be the set of index set of cardinality $k$. By definition it always holds that $\tilde w^{(t)}_{S,k} \in \{w_{S\mid J}: J\in \mathcal{J}\}$. Using the identical proof arguments of Proposition~\ref{thrm:uniform_stability} we can show that with probability at least $1-\delta$ over $S$, the generalization gap $F(\tilde w^{(t)}_{S,k}) - F_S(\tilde w^{(t)}_{S,k})$ is upper bounded by $\mathcal{O}\left(\frac{G^{3/2}M^{1/4}}{\mu_k^{3/4}}\sqrt{\frac{\log(n)(\log(n/\delta)+k\log(p/k))}{n}}\right)$ which implies the first desired result.

For any $\epsilon>0$, given that $t=\mathcal{O}\left(\frac{L}{\mu_{3k}}\log \left(\frac{F_S(w^{(0)})}{\epsilon}\right)\right)$ is sufficiently large, we can bound the excess risk $F(\tilde w^{(t)}_{S}) - F(\bar w)$ as
\[
\begin{aligned}
F(\tilde w^{(t)}_{S,k}) - F(\bar w) =& F(\tilde w^{(t)}_{S,k}) - F_S(\tilde w^{(t)}_{S,k}) + F_S(\tilde w^{(t)}_{S,k}) -F_S(\bar w) + F_S(\bar w) - F(\bar w) \\
\le& F(\tilde w^{(t)}_{S,k}) - F_S(\tilde w^{(t)}_{S,k}) + F_S(\bar w) - F(\bar w) + \epsilon,
\end{aligned}
\]
where in the last inequality we have used the bound $F_S(\tilde w^{(t)}_{S,k}) \le F_S(w^{(t)}_{S,k}) \le F_S(\bar w) + \epsilon$ which is implied by the definition of $\tilde w_{S,k}^{(t)}$ and Lemma~\ref{lemma:convergence_iht}. Since $\ell(w;\xi)\le M$, from Hoeffding inequality we know that $F_S(\bar w) - F(\bar w) \le M\sqrt{\frac{\log(2/\delta)}{2n}}$ is valid with probability at least $1-\delta/2$. Then based on the previous generalization gap bound and union probability we get that with probability at least $1-\delta$, $F(\tilde w^{(t)}_{S,k}) - F(\bar w) \le$
\[
\mathcal{O}\left(\frac{G^{3/2}M^{1/4}}{\mu_k^{3/4}}\sqrt{\frac{\log(n)(\log(n/\delta)+k\log(p/k))}{n}} + M\sqrt{\frac{\log(1/\delta)}{n}}+ \epsilon\right).
\]
By setting $\epsilon=\mathcal{O}(\sqrt{k\log(n)\log(p/k)/n})$ we obtain the desired bound. This completes the proof.
\end{proof}

\subsection{Proof of Theorem~\ref{thrm:uniform_stability_strong_iht}}
\label{apdsect:proof_uniform_stability_strong_iht}

In this subsection, we present a detailed proof Theorem~\ref{thrm:uniform_stability_strong_iht} as restated below.
\IHTBlackBoxBoundsStabilityStrong*

In the following proof arguments, we will frequently use the operator $\mathrm{H}_J(w)$ which is defined as the restriction of $w$ over an index set $J$. We also will use the abbreviation $\mathrm{H}_J(\nabla F(w))=\nabla_J F(w)$ for the sake of notation simplicity. The following standard lemma is useful in our analysis. Its proof is provided here to make the entire proof arguments self-contained.
\begin{lemma}\label{lemma:strong_smooth}
Assume that a differentiable function $f$ is $\mu_s$-strongly convex and $L$-strongly smooth. For any index set $J$ with cardinality $|J| \le s$ and any $w,w'$
with $\supp(w)\cup \supp(w')\subseteq J$, if $\eta\in(0, 2 / (L+\mu_s) )$, then
\[
\|w - w' - \eta \nabla_J f(w) + \eta \nabla_J f(w')\| \le \left(1 - \frac{\eta L\mu_s}{L+\mu_s}\right)\|w - w'\|.
\]
\end{lemma}
\begin{proof}
Since $f$ is $\mu_s$-strongly convex over $J$, we have that $g(w)=f(w) - \frac{\mu_s\|w\|^2}{2}$ is convex and $(L-\mu_s)$-strongly smooth when restricted to $J$. Then based on the co-coercivity of $\nabla g$ we know that
\[
\langle \nabla_J g(w) - \nabla_J g(w'), w - w'\rangle \ge \frac{1}{L-\mu_s}\| \nabla_J g(w) - \nabla_J g(w')\|^2,
\]
which then yields
\[
\langle \nabla_J f(w) - \nabla_J f(w'), w - w'\rangle \ge \frac{1}{L+\mu_s}\|\nabla_J f(w) - \nabla_J f(w')\|^2 + \frac{L\mu_s}{L+\mu_s}\|w-w'\|^2.
\]
Based on this inequality we can show
\[
\begin{aligned}
&\|w - w' - \eta \nabla_J f(w) + \eta \nabla_J f(w')\|^2 \\
=& \|w-w'\|^2 - 2\eta \langle \nabla_J f(w) - \nabla_J f(w'), w - w'\rangle + \eta^2 \|\nabla_J f(w) - \nabla_J f(w')\|^2\\
\le& \left(1 - \frac{2\eta L\mu_s}{L+\mu_s}\right)\|w - w'\|^2 - \left(\frac{2\eta}{L+\mu_s} - \eta^2\right)\| \nabla_J f(w) - \nabla_J f(w')\|^2 \\
\le& \left(1 - \frac{2\eta L\mu_s}{L+\mu_s}\right)\|w - w'\|^2 \le \left(1 - \frac{\eta L\mu_s}{L+\mu_s}\right)^2\|w - w'\|^2,
\end{aligned}
\]
where in the last but one inequality we have used the assumption on $\eta$ and the last inequality follows from $1-2a\le (1-a)^2$. This readily implies the desired bound.
\end{proof}

The following key lemma shows that if the population function $F$ is IHT stable, then the supporting set of the sparse solution returned by IHT invoked on the empirical risk $F_S$ is also unique provided that $F_S$ is close enough to $F$ along the solution path of IHT.
\begin{lemma}\label{lemma:emp_iht_stability}
For a fixed data sample $S$, assume that $F_S$ is $\mu_{4k}$-strongly convex and $L$-strongly smooth. Let $\{w^{(t)}\}_{t=1}^T$ and $\{w_{S,k}^{(t)}\}_{t=1}^T$ respectively be the sequence generated by invoking IHT on $F$ and $F_S$ with step-size $\eta= \frac{2}{3L}$ and initialization $w^{(0)}$. Suppose that the population risk function $F$ is $(\varepsilon_{k},\eta, T, w^{(0)})$-IHT stable and $\left\|\nabla F_S(w^{(t)}) - \nabla F(w^{(t)})\right\|\le \frac{L\mu_{4k}\varepsilon_k}{2(L+\mu_{4k})}$, $\forall t\in [T]$. Then we have
\[
\|w^{(t)}_{S,k} - w^{(t)}\| < \varepsilon_k/2, \ \ \supp\left(w^{(t)}_{S,k}\right) = \supp\left(w^{(t)}\right), \  \ \forall t\in [T].
\]
\end{lemma}
\begin{proof}
We show by induction that $\forall t\in \{0\} \cup[T]$, $\|w^{(t)}_{S,k} - w^{(t)}\| < \varepsilon_k/2$ and $\supp(w_{S,k}^{(t)})=\supp(w^{(t)})$. The base case $t=0$ holds trivially as  $w_{S,k}^{(0)}=w^{(0)}$. Suppose that the claim holds for some $t\ge 0$. Now consider the case $t+1$. Denote $J_S^{(\tau)}=\supp(w_{S,k}^{(\tau)})$ and $J^{(\tau)}=\supp(w^{(\tau)})$ for $\tau=t,t+1$ and $J=\cup_{\tau=t}^{t+1} \left(J_S^{(\tau)}\cup J^{(\tau)}\right)$. Then we must have $|J|\le 4k$. Let us consider the following pair of vectors:
\[
\hat w_{S,k}^{(t+1)}:= \mathrm{H}_J \left(w_{S,k}^{(t)} - \eta \nabla F_S(w_{S,k}^{(t)})\right), \quad\hat w^{(t+1)}:=\mathrm{H}_J \left( w^{(t)} - \eta \nabla F(w^{(t)})\right).
\]
We can show that
\[
\begin{aligned}
&\left\|\hat w_{S,k}^{(t+1)} - \hat w^{(t+1)}\right\| = \left\| w_{S,k}^{(t)} - \eta \nabla_J F_S(w_{S,k}^{(t)}) - w^{(t)} + \eta \nabla_J F(w^{(t)})\right\| \\
=&  \left\| w_{S,k}^{(t)} - w^{(t)} - \eta \nabla_J F_S(w_{S,k}^{(t)}) + \eta \nabla_J F_S(w^{(t)}) - \eta \nabla_J F_S(w^{(t)}) + \eta \nabla_J F(w^{(t)})\right\| \\
\le& \left\| w_{S,k}^{(t)} - w^{(t)} - \eta \nabla_J F_S(w_{S,k}^{(t)}) + \eta \nabla_J F_S(w^{(t)})\right\| + \eta \left\|\nabla F_S(w^{(t)}) - \nabla F(w^{(t)})\right\| \\
\overset{\zeta_1}{\le}& \left(1-\frac{2\mu_{4k}}{3(L+\mu_{4k})}\right)\left\|w_{S,k}^{(t)} - w^{(t)}\right\| + \frac{2}{3L}\left\|\nabla F_S(w^{(t)}) - \nabla F(w^{(t)})\right\| \\
\overset{\zeta_2}{<}&\left(1-\frac{2\mu_{4k}}{3(L+\mu_{4k})}\right) \frac{\varepsilon_k}{2} + \frac{2\mu_{4k}}{3(L+\mu_{4k})} \frac{\varepsilon_k}{2} = \frac{\varepsilon_k}{2},
\end{aligned}
\]
where in ``$\zeta_1$'' we have used Lemma~\ref{lemma:strong_smooth} with $\eta= \mu_{4k}/L^2$, and ``$\zeta_2$'' follows from the induction assumption and the bound on $\left\|\nabla F_S(w^{(t)}) - \nabla F(w^{(t)})\right\|$. By definition $J^{(t+1)}\subseteq J$, and thus it holds trivially that $J^{(t+1)}$ also uniquely contains the top $k$ (in magnitude) entries of $\hat w^{(t+1)}$ as a restriction of $w^{(t+1)}$ over $J$. Based on this observation, since $F$ is $(\varepsilon_{k},\eta, T, w^{(0)})$-IHT stable, $w^{(t)} - \eta \nabla F(w^{(t)})$ must be $\varepsilon_{k}$-hard-thresholding stable which then implies that $\hat w^{(t+1)}$ is also $\varepsilon_{k}$-hard-thresholding stable. Therefore, the preceding inequality readily indicates that $\hat w_{S,k}^{(t+1)}$ and $\hat w^{(t+1)}$ share the identical top $k$ entries, and thus $J_S^{(t+1)}=J^{(t+1)}$. Consequently, based on the preceding inequality we can show that
\[
\left\|w_{S,k}^{(t+1)} - w^{(t+1)}\right\| \le \left\|\hat w_{S,k}^{(t+1)} - \hat w^{(t+1)}\right\| < \varepsilon_k/2.
\]
This shows that the claim holds for $t+1$ and the proof is concluded.
\end{proof}
Now we are in the position to prove the main result.
\begin{proof}[Proof of Theorem~\ref{thrm:uniform_stability_strong_iht}]
Let us define an oracle sequence $\{w^{(t)}\}_{t=1}^T$ generated by applying $T$ rounds of IHT iteration to the population risk $F$ with the considered fixed initialization $w^{(0)}$ and step-size $\eta$. Since $F$ is $(\varepsilon_{k},\eta, T, w^{(0)})$-IHT stable, the sequence $\{w^{(t)}\}_{t=1}^T$ is unique. Since $\|\nabla \ell(w;\cdot)\|\le G$, from Hoeffding concentration bound and union probability we know that with probability at least $1-\delta$ over $S$, $\forall t\in [T]$,
\[
\|\nabla F_S(w^{(t)}) - \nabla F(w^{(t)})\| \le G \sqrt{\frac{\log(pT/\delta)}{2n}} \le \frac{L\mu_{4k}\varepsilon_k}{2(L+\mu_{4k})},
\]
where the last inequality is due to the condition on sample size $n$. The assumptions in the theorem imply that $F_S$ is $L$-strongly smooth and $\mu_{4k}$-strongly convex with probability at least $1-\delta'_n$ over $S$. Therefore, by invoking Lemma~\ref{lemma:emp_iht_stability} and union probability we know that with  probability at least $1-\delta'_n -\delta$ over $S$,
\[
 \supp\left(w^{(T)}_S\right) = \supp\left(w^{(T)}\right).
\]
Since $\supp\left(w^{(T)}\right)$ is a fixed deterministic index set of size $k$, from Lemma~\ref{lemma:support_stability} (keep in mind that $\mu_{4k}\le \mu_k$) we obtain the following bound:
\[
F(\tilde w^{(T)}_S) - F_S(\tilde w^{(T)}_S) \le \mathcal{O}\left(\frac{G^2}{\lambda n}\log(n)\log(n/\delta) + \sqrt{\frac{\log(1/\delta)}{n}} + \frac{\lambda G \sqrt{M}}{\mu_{4k}\sqrt{\mu_{4k}}}\right).
\]
The desired generalization gap bound then follows immediately by setting $\lambda=\sqrt{\frac{\mu_{4k}^{1.5}\log(n)\log(n/\delta)}{nM^{0.5}}}$ in the above bound.

The proof of the excess risk bound is almost identical to that of Corollary~\ref{corol:uniform_stability_iht}. Here we still present some key ingredients for the sake of completeness. If $T=\mathcal{O}\left(\frac{L}{\mu_{3k}}\log \left(\frac{F_S(w^{(0)})}{\epsilon}\right)\right)$ is sufficiently large, then using Lemma~\ref{lemma:convergence_iht} we can bound $F(\tilde w^{(T)}_{S}) - F(\bar w)$ as
\[
F(\tilde w^{(T)}_S) - F(\bar w) \le F(\tilde w^{(T)}_S) - F_S(\tilde w^{(T)}_S) + F_S(\bar w) - F(\bar w) + \epsilon.
\]
Since $\ell(w;\xi)\le M$, from Hoeffding inequality we know that $F_S(\bar w) - F(\bar w) \le M\sqrt{\frac{\log(2/\delta)}{2n}}$ is valid with probability at least $1-\delta/2$.
In view of the previous generalization gap bound and by union probability we get that with probability at least $1-\delta-\delta'_n$,
\[
\begin{aligned}
F(\tilde w^{(t)}_{S,k}) - F(\bar w) \le \mathcal{O}\left(\frac{G^{3/2}M^{1/4}}{\mu_{2k}^{3/4}}\sqrt{\frac{\log(n)\log(n/\delta)}{n}} + M\sqrt{\frac{\log(1/\delta)}{n}} + \epsilon\right).
\end{aligned}
\]
By setting $\epsilon=\mathcal{O}(\sqrt{\log(1/\delta)/n})$ we obtain the desired bound. This completes the proof.
\end{proof}

\section{Proofs of Auxiliary Lemmas}\label{ProofforAuxiliaryLemmas}

This appendix section collects the technical proofs of the auxiliary lemmas presented in Appendix~\ref{append:lemmas}.

\subsection{Proof of Lemma~\ref{lemma:estimation_error}}
\begin{proof}
Since $f$ is $\mu_s$-strongly convex and $f(w)\le f(w') +\epsilon$, we have
\[
\begin{aligned}
f(w)  \ge& f(w') + \langle \nabla f(w'), w - w'\rangle + \frac{\mu_s}{2} \|w - w'\|^2 \\
\ge& f(w') - \|\nabla_{I \cup I'} f(w')\|\|w - w'\| + \frac{\mu_s}{2} \|w - w'\|^2\\
\ge& f(w) - \epsilon - \|\nabla_{I \cup I'} f(w')\|\|w - w'\| + \frac{\mu_s}{2} \|w - w'\|^2,
\end{aligned}
\]
where the second inequality follows from Cauchy-Schwarz inequality. The above inequality then implies
\[
\|w - w'\| \le \frac{2\|\nabla_{I \cup I'} f(w')\| + \sqrt{2 \mu_s\epsilon}}{\mu_s}\le \frac{2\sqrt{s}\|\nabla f(w')\|_\infty}{\mu_s} + \sqrt{\frac{2\epsilon}{\mu_s}}.
\]
This shows the desired bound. Next we show that $\supp(w')\subseteq \supp(w)$ given that $w'_{\min}> \frac{2\sqrt{s}\|\nabla f(w')\|_\infty}{\mu_s}+ \sqrt{\frac{2\epsilon}{\mu_s}}$. Assume otherwise $\supp(w')\nsubseteq \supp(w)$. Then the previous bound implies that
\[
w'_{\min} \le \|w - w'\| \le \frac{2\sqrt{s}\|\nabla f(w')\|_\infty}{\mu_s} + \sqrt{\frac{2\epsilon}{\mu_s}},
\]
which contradicts the assumption. Therefore, it must hold that $\supp(w')\subseteq \supp(w)$.
\end{proof}
\begin{proof}
Let us define an \emph{oracle} sequence $\{w^{(t)}\}_{t=1}^T$ generated by applying $T$ rounds of IHT iteration to the population risk $F$ with the considered initialization $w^{(0)}$ and step-size $\eta$. Since $F$ is $(\varepsilon_{k},\eta, T, w^{(0)})$-IHT stable, the sequence $\{w^{(t)}\}_{t=1}^T$ is unique and unique. Since $\|\nabla \ell(w;\cdot)\|\le G$, from Hoeffding concentration bound and union probability we know that with probability at least $1-\delta$ over $S$, $\forall t\in [T]$,
\[
\|\nabla F_S(w^{(t)}) - \nabla F(w^{(t)})\| \le G \sqrt{\frac{\log(pT/\delta)}{2n}} \le \frac{L^2\varepsilon_k}{2\mu_{4k}\left(1-\sqrt{1-\mu^2_{4k}/L^2}\right)},
\]
where the last inequality is due to the condition on sample size $n$. The assumptions in the theorem imply that $F_S$ is $L$-strongly smooth and $\mu_{4k}$-strongly convex with probability at least $1-\delta'_n$ over $S$. Therefore, by invoking Lemma~\ref{lemma:emp_iht_stability} and union probability we know that with  probability at least $1-\delta'_n -\delta$ over $S$,
\[
 \supp\left(w^{(T)}_S\right) = \supp\left(w^{(T)}\right).
\]
Since $\supp\left(w^{(T)}\right)$ is a fixed deterministic index set of size $k$, from Lemma~\ref{lemma:support_stability} (keep in mind that $\mu_{4k}\le \mu_k$) we obtain the following bound:
\[
F(\tilde w^{(T)}_S) - F_S(\tilde w^{(T)}_S) \le \mathcal{O}\left(\frac{G^2\log(n)\log(n/\delta)}{\lambda n} + \sqrt{\frac{\log(1/\delta)}{n}} + \frac{\lambda G R}{\mu_{4k}+\lambda}\right).
\]
Finally, the desired bound follows immediately by setting $\lambda = \mathcal{O}(1/\sqrt{n})$ in the above bound.
\end{proof}

\subsection{Proof of Lemma~\ref{lemma:max_subgaussian_bound}}
\begin{proof}
The proof of the first bound follows directly from~\cite[Lemma 1.3]{lugosi2002pattern} and thus is omitted here. We just prove the second bound. Fix $j\in\{1,...,p\}$. Since $X_j$ is zero-mean $\sigma^2$-sub-Gaussian, it is standard to know that $ \mathbb{E}[X^2_j] \le 4\sigma^2$ and $X^2_j$ is $16\sigma^2$-sub-exponential, i.e.,
\[
\mathbb{E}\left[\exp\left\{\lambda (X^2_j - \mathbb{E}[X^2_j])\right\}\right]\le \exp\left\{128\lambda^2\sigma^4\right\}, \quad |\lambda| \le \frac{1}{16\sigma^2}.
\]
It follows that
\[
\mathbb{E}\left[\exp\left\{\lambda X^2_j\right\}\right]\le \exp\left\{\lambda\mathbb{E}[X^2_j] + 128\lambda^2\sigma^4\right\} \le \exp\left\{1/4 + 128\lambda^2\sigma^4\right\}, \quad |\lambda| \le \frac{1}{16\sigma^2}.
\]
Then, for any $\lambda \in(0,\frac{1}{16\sigma^2}]$,
\[
\begin{aligned}
&\mathbb{E}\left[\max_{1\le j \le p} X_j^2\right] = \frac{1}{\lambda} \mathbb{E}\left[\log\left(\exp\left\{\lambda \max_{1\le j \le p} X_j^2\right\}\right)\right] \\
\overset{\zeta_1}{\le}& \frac{1}{\lambda} \log\left(\mathbb{E}\left[\exp\left\{\lambda \max_{1\le j \le p} X_j^2\right\}\right]\right) =  \frac{1}{\lambda} \log\left(\mathbb{E}\left[\max_{1\le j \le p} \exp\left\{\lambda X_j^2\right\}\right]\right) \\
\le& \frac{1}{\lambda} \log\left(\sum_{1\le j \le p}\mathbb{E}\left[ \exp\left\{\lambda X_j^2\right\}\right]\right) \overset{\zeta_2}{\le} \frac{1}{\lambda} \log\left(\sum_{1\le j \le p}\exp\left\{1/4 +  128\lambda^2\sigma^4\right\}\right) \\
=&  \frac{1}{\lambda} \log\left(p\exp\left\{1/4 +  128\lambda^2\sigma^4\right\}\right) = \frac{\log p}{\lambda} + \frac{1}{4\lambda} + 128\lambda\sigma^4,
\end{aligned}
\]
where ``$\zeta_1$'' is due to Jensen's inequality and ``$\zeta_2$'' follows from the previous inequality. By setting $\lambda = \frac{1}{16\sigma^2}$ we obtain
\[
\mathbb{E}\left[\max_{1\le j \le p} X_j^2\right]  \le 72\sigma^2 + 16\sigma^2\log p.
\]
This proves the desired bound.
\end{proof}

\subsection{Proof of Lemma~\ref{lemma:sub_Gaussian_bounds}}

\begin{proof}
Fix $j \in \{1,...,p\}$. Since $\nabla_j \ell(\bar w;\xi)$ are assumed to be $\sigma^2$-sub-Gaussian and $\nabla F(\bar w) = \mathbb{E}_{\xi}\left[\nabla \ell(\bar w;\xi)\right] = 0$, we must have $\nabla_j \ell(\bar w;\xi)$ are zero-mean $\sigma^2$-sub-Gaussian. Thus it is known from the Hoeffding bound (see, e.g., ~\cite{vershynin2010introduction}) that for any $\varepsilon >0 $,
\[
\mathbb{P}\left(\left|\nabla_j F_S(\bar w)\right| > \varepsilon \right) = \mathbb{P}\left(\left|\frac{1}{n}\sum_{\xi_i\in S}\nabla_j \ell(\bar w; \xi_i)\right| > \varepsilon \right)\le \exp\left\{-\frac{n\varepsilon^2}{2\sigma^2}\right\}.
\]
By the union  bound we have
\[
\mathbb{P}(\|\nabla F_S(\bar w)\|_\infty > \varepsilon) \le
p\exp\left\{-\frac{n\varepsilon^2}{2\sigma^2}\right\} \nonumber.
\]
By choosing $\varepsilon = \sqrt{\frac{2\sigma^2\log(p/\delta)}{n}}$ in the above inequality we obtain that
with probability at least $1-\delta$,
\[
\|\nabla F_S(\bar w)\|_\infty \le \sqrt{\frac{2\sigma^2\log(p/\delta)}{n}}.
\]
This proves the first high probability bound.

Next, we prove the concentration bound in expectation. For each $j\in\{1,...,p\}$, let $X_j=\sum_{\xi_i\in S}\nabla_j \ell(\bar w; \xi_i)$. Since $\nabla_j \ell(\bar w;\xi_i)$ are assumed to be $\sigma^2$-sub-Gaussian, we have $X_j$ are all $n\sigma^2$-sub-Gaussian. By invoking Lemma~\ref{lemma:max_subgaussian_bound} we obtain
\[
\mathbb{E}\left[\max_{1\le j\le p }|X_j|\right] \le \sigma\sqrt{2n\log(2p)}.
\]
It follows that
\[
\mathbb{E}_S\left[\|\nabla F_S(\bar w)\|_\infty\right]  = \frac{1}{n}\mathbb{E}\left[\max_{1\le j\le p }|X_j|\right] \le \sigma\sqrt{\frac{2\log(2p)}{n}}.
\]
Moreover, since $X_j$ are all $n\sigma^2$-sub-Gaussian, it follows from Lemma~\ref{lemma:max_subgaussian_bound} that
\[
\quad \mathbb{E}\left[\max_{1\le j\le p }X^2_j\right] \le  72n\sigma^2 + 16n\sigma^2\log p.
\]
Finally,
\[
\mathbb{E}_S\left[\|\nabla F_S(\bar w)\|^2_\infty\right]  = \frac{1}{n^2}\mathbb{E}\left[\max_{1\le j\le p }X^2_j\right] \le \frac{\sigma^2(72 + 16\log p)}{n}.
\]
This completes the proof.
\end{proof}

\bibliographystyle{acmtrans-ims}
\bibliography{ijcai17}

\begin{thebibliography}{}
\ifx \url   \undefined \def \url#1{#1}   \fi

\bibitem{abramovich2018high}
\textsc{Abramovich, F.} \textsc{and} \textsc{Grinshtein, V.} (2019).
\newblock High-dimensional classification by sparse logistic regression.
\newblock \emph{IEEE Transactions on Information Theory\/}~\textbf{65},~5,
  3068--3079.

\bibitem{agarwal2012fast}
\textsc{Agarwal, A.}, \textsc{Negahban, S.}, \textsc{Wainwright, M.~J.},
  \textsc{and} \textsc{others}. (2012).
\newblock Fast global convergence of gradient methods for high-dimensional
  statistical recovery.
\newblock \emph{The Annals of Statistics\/}~\textbf{40},~5, 2452--2482.

\bibitem{arora2018stronger}
\textsc{Arora, S.}, \textsc{Ge, R.}, \textsc{Neyshabur, B.}, \textsc{and}
  \textsc{Zhang, Y.} (2018).
\newblock Stronger generalization bounds for deep nets via a compression
  approach.
\newblock In \emph{International Conference on Machine Learning}. 254--263.

\bibitem{bach2012optimization}
\textsc{Bach, F.}, \textsc{Jenatton, R.}, \textsc{Mairal, J.}, \textsc{and}
  \textsc{Obozinski, G.} (2012).
\newblock Optimization with sparsity-inducing penalties.
\newblock \emph{Foundations and Trends{\textregistered} in Machine
  Learning\/}~\textbf{4},~1, 1--106.

\bibitem{bahmani2013greedy}
\textsc{Bahmani, S.}, \textsc{Raj, B.}, \textsc{and} \textsc{Boufounos, P.~T.}
  (2013).
\newblock Greedy sparsity-constrained optimization.
\newblock \emph{Journal of Machine Learning Research\/}~\textbf{14},~Mar,
  807--841.

\bibitem{bartlett2006convexity}
\textsc{Bartlett, P.~L.}, \textsc{Jordan, M.~I.}, \textsc{and}
  \textsc{McAuliffe, J.~D.} (2006).
\newblock Convexity, classification, and risk bounds.
\newblock \emph{Journal of the American Statistical
  Association\/}~\textbf{101},~473, 138--156.

\bibitem{bellec2018slope}
\textsc{Bellec, P.~C.}, \textsc{Lecu{\'e}, G.}, \textsc{Tsybakov, A.~B.},
  \textsc{and} \textsc{others}. (2018).
\newblock Slope meets lasso: improved oracle bounds and optimality.
\newblock \emph{The Annals of Statistics\/}~\textbf{46},~6B, 3603--3642.

\bibitem{blumensath2009iterative}
\textsc{Blumensath, T.} \textsc{and} \textsc{Davies, M.~E.} (2009).
\newblock Iterative hard thresholding for compressed sensing.
\newblock \emph{Applied and Computational Harmonic Analysis\/}~\textbf{27},~3,
  265--274.

\bibitem{boroczky2003covering}
\textsc{B{\"o}r{\"o}czky, K.} \textsc{and} \textsc{Wintsche, G.} (2003).
\newblock Covering the sphere by equal spherical balls.
\newblock In \emph{Discrete and computational geometry}. Springer, 235--251.

\bibitem{bottou2008tradeoffs}
\textsc{Bottou, L.} \textsc{and} \textsc{Bousquet, O.} (2008).
\newblock The tradeoffs of large scale learning.
\newblock In \emph{Advances in Neural Information Processing Systems}.
  161--168.

\bibitem{bousquet2002stability}
\textsc{Bousquet, O.} \textsc{and} \textsc{Elisseeff, A.} (2002).
\newblock Stability and generalization.
\newblock \emph{Journal of Machine Learning Research\/}~\textbf{2},~Mar,
  499--526.

\bibitem{cesa2006prediction}
\textsc{Cesa-Bianchi, N.} \textsc{and} \textsc{Lugosi, G.} (2006).
\newblock \emph{Prediction, learning, and games}.
\newblock Cambridge University Press.

\bibitem{charles2018stability}
\textsc{Charles, Z.} \textsc{and} \textsc{Papailiopoulos, D.} (2018).
\newblock Stability and generalization of learning algorithms that converge to
  global optima.
\newblock In \emph{International Conference on Machine Learning}. 744--753.

\bibitem{chatterjee2013assumptionless}
\textsc{Chatterjee, S.} (2013).
\newblock Assumptionless consistency of the lasso.
\newblock \emph{arXiv preprint arXiv:1303.5817\/}.

\bibitem{chen2018best}
\textsc{Chen, L.-Y.} \textsc{and} \textsc{Lee, S.} (2018a).
\newblock Best subset binary prediction.
\newblock \emph{Journal of Econometrics\/}~\textbf{206},~1, 39--56.

\bibitem{chen2018high}
\textsc{Chen, L.-Y.} \textsc{and} \textsc{Lee, S.} (2018b).
\newblock High dimensional classification through $\ell_0$-penalized empirical
  risk minimization.
\newblock \emph{arXiv preprint arXiv:1811.09540\/}.

\bibitem{donoho2006compressed}
\textsc{Donoho, D.~L.} (2006).
\newblock Compressed sensing.
\newblock \emph{IEEE Transactions on Information Theory\/}~\textbf{52},~4,
  1289--1306.

\bibitem{fan2001variable}
\textsc{Fan, J.} \textsc{and} \textsc{Li, R.} (2001).
\newblock Variable selection via nonconcave penalized likelihood and its oracle
  properties.
\newblock \emph{Journal of the American statistical
  Association\/}~\textbf{96},~456, 1348--1360.

\bibitem{fan2011nonconcave}
\textsc{Fan, J.} \textsc{and} \textsc{Lv, J.} (2011).
\newblock Nonconcave penalized likelihood with np-dimensionality.
\newblock \emph{IEEE Transactions on Information Theory\/}~\textbf{57},~8,
  5467--5484.

\bibitem{fan2004nonconcave}
\textsc{Fan, J.} \textsc{and} \textsc{Peng, H.} (2004).
\newblock Nonconcave penalized likelihood with a diverging number of
  parameters.
\newblock \emph{The Annals of Statistics\/}~\textbf{32},~3, 928--961.

\bibitem{fan2014strong}
\textsc{Fan, J.}, \textsc{Xue, L.}, \textsc{Zou, H.}, \textsc{and}
  \textsc{others}. (2014).
\newblock Strong oracle optimality of folded concave penalized estimation.
\newblock \emph{The Annals of Statistics\/}~\textbf{42},~3, 819--849.

\bibitem{feldman2018generalization}
\textsc{Feldman, V.} \textsc{and} \textsc{Vondrak, J.} (2018).
\newblock Generalization bounds for uniformly stable algorithms.
\newblock In \emph{Advances in Neural Information Processing Systems}.
  9747--9757.

\bibitem{feldman2019high}
\textsc{Feldman, V.} \textsc{and} \textsc{Vondrak, J.} (2019).
\newblock High probability generalization bounds for uniformly stable
  algorithms with nearly optimal rate.
\newblock In \emph{Proceedings of the Thirty-Second Conference on Learning
  Theory}. 1270--1279.

\bibitem{foucart2011hard}
\textsc{Foucart, S.} (2011).
\newblock Hard thresholding pursuit: an algorithm for compressive sensing.
\newblock \emph{SIAM Journal on Numerical Analysis\/}~\textbf{49},~6,
  2543--2563.

\bibitem{foucart2017mathematical}
\textsc{Foucart, S.} \textsc{and} \textsc{Rauhut, H.} (2017).
\newblock A mathematical introduction to compressive sensing.
\newblock \emph{Bull. Am. Math\/}~\emph{54}, 151--165.

\bibitem{frankle2019lottery}
\textsc{Frankle, J.} \textsc{and} \textsc{Carbin, M.} (2019).
\newblock The lottery ticket hypothesis: Finding sparse, trainable neural
  networks.
\newblock In \emph{International Conference on Learning Representations}.

\bibitem{garg2009gradient}
\textsc{Garg, R.} \textsc{and} \textsc{Khandekar, R.} (2009).
\newblock Gradient descent with sparsification: an iterative algorithm for
  sparse recovery with restricted isometry property.
\newblock In \emph{International Conference on Machine Learning}.
  Vol.~\textbf{9}. 337--344.

\bibitem{han2016deep}
\textsc{Han, S.}, \textsc{Mao, H.}, \textsc{and} \textsc{Dally, W.~J.} (2016).
\newblock Deep compression: Compressing deep neural networks with pruning,
  trained quantization and huffman coding.
\newblock In \emph{International Conference on Learning Representations}.

\bibitem{hardt2016train}
\textsc{Hardt, M.}, \textsc{Recht, B.}, \textsc{and} \textsc{Singer, Y.}
  (2016).
\newblock Train faster, generalize better: Stability of stochastic gradient
  descent.
\newblock In \emph{International Conference on Machine Learning}. 1225--1234.

\bibitem{hastie2015statistical}
\textsc{Hastie, T.}, \textsc{Tibshirani, R.}, \textsc{and} \textsc{Wainwright,
  M.} (2015).
\newblock \emph{Statistical learning with sparsity}.
\newblock CRC press.

\bibitem{jain2014iterative}
\textsc{Jain, P.}, \textsc{Tewari, A.}, \textsc{and} \textsc{Kar, P.} (2014).
\newblock On iterative hard thresholding methods for high-dimensional
  m-estimation.
\newblock In \emph{Advances in Neural Information Processing Systems}.
  685--693.

\bibitem{jin2016training}
\textsc{Jin, X.}, \textsc{Yuan, X.-T.}, \textsc{Feng, J.}, \textsc{and}
  \textsc{Yan, S.} (2016).
\newblock Training skinny deep neural networks with iterative hard thresholding
  methods.
\newblock \emph{arXiv preprint arXiv:1607.05423\/}.

\bibitem{kakade2009complexity}
\textsc{Kakade, S.~M.}, \textsc{Sridharan, K.}, \textsc{and} \textsc{Tewari,
  A.} (2009).
\newblock On the complexity of linear prediction: Risk bounds, margin bounds,
  and regularization.
\newblock In \emph{Advances in Neural Information Processing Systems}.
  793--800.

\bibitem{kuzborskij2018data}
\textsc{Kuzborskij, I.} \textsc{and} \textsc{Lampert, C.} (2018).
\newblock Data-dependent stability of stochastic gradient descent.
\newblock In \emph{International Conference on Machine Learning}. 2820--2829.

\bibitem{li2015sparsistency}
\textsc{Li, Y.-H.}, \textsc{Scarlett, J.}, \textsc{Ravikumar, P.}, \textsc{and}
  \textsc{Cevher, V.} (2015).
\newblock Sparsistency of $\ell_1$-regularized m-estimators.
\newblock In \emph{Artificial Intelligence and Statistics}. 644--652.

\bibitem{loh2012high}
\textsc{Loh, P.-L.}, \textsc{Wainwright, M.~J.}, \textsc{and} \textsc{others}.
  (2012).
\newblock High-dimensional regression with noisy and missing data: Provable
  guarantees with nonconvexity.
\newblock \emph{The Annals of Statistics\/}~\textbf{40},~3, 1637--1664.

\bibitem{lugosi2002pattern}
\textsc{Lugosi, G.} (2002).
\newblock Pattern classification and learning theory.
\newblock In \emph{Principles of Nonparametric Learning}. Springer, 1--56.

\bibitem{maurer2012structured}
\textsc{Maurer, A.} \textsc{and} \textsc{Pontil, M.} (2012).
\newblock Structured sparsity and generalization.
\newblock \emph{Journal of Machine Learning Research\/}~\textbf{13},~Mar,
  671--690.

\bibitem{mei2018landscape}
\textsc{Mei, S.}, \textsc{Bai, Y.}, \textsc{Montanari, A.}, \textsc{and}
  \textsc{others}. (2018).
\newblock The landscape of empirical risk for nonconvex losses.
\newblock \emph{The Annals of Statistics\/}~\textbf{46},~6A, 2747--2774.

\bibitem{meinshausen2009lasso}
\textsc{Meinshausen, N.}, \textsc{Yu, B.}, \textsc{and} \textsc{others}.
  (2009).
\newblock Lasso-type recovery of sparse representations for high-dimensional
  data.
\newblock \emph{The Annals of Statistics\/}~\textbf{37},~1, 246--270.

\bibitem{mukherjee2006learning}
\textsc{Mukherjee, S.}, \textsc{Niyogi, P.}, \textsc{Poggio, T.}, \textsc{and}
  \textsc{Rifkin, R.} (2006).
\newblock Learning theory: stability is sufficient for generalization and
  necessary and sufficient for consistency of empirical risk minimization.
\newblock \emph{Advances in Computational Mathematics\/}~\textbf{25},~1-3,
  161--193.

\bibitem{natarajan1995sparse}
\textsc{Natarajan, B.~K.} (1995).
\newblock Sparse approximate solutions to linear systems.
\newblock \emph{SIAM Journal on Computing\/}~\textbf{24},~2, 227--234.

\bibitem{pollard2012convergence}
\textsc{Pollard, D.} (2012).
\newblock \emph{Convergence of stochastic processes}.
\newblock Springer Science \& Business Media.

\bibitem{ravikumar2011high}
\textsc{Ravikumar, P.}, \textsc{Wainwright, M.~J.}, \textsc{Raskutti, G.},
  \textsc{Yu, B.}, \textsc{and} \textsc{others}. (2011).
\newblock High-dimensional covariance estimation by minimizing
  $\ell_1$-penalized log-determinant divergence.
\newblock \emph{Electronic Journal of Statistics\/}~\emph{5}, 935--980.

\bibitem{rigollet201518}
\textsc{Rigollet, P.} (2015).
\newblock 18. s997: High dimensional statistics.
\newblock \emph{Lecture Notes, Cambridge, MA, USA: MIT Open-CourseWare\/}.

\bibitem{shalev2009stochastic}
\textsc{Shalev-Shwartz, S.}, \textsc{Shamir, O.}, \textsc{Srebro, N.},
  \textsc{and} \textsc{Sridharan, K.} (2009).
\newblock Stochastic convex optimization.
\newblock In \emph{COLT}.

\bibitem{shalev2010learnability}
\textsc{Shalev-Shwartz, S.}, \textsc{Shamir, O.}, \textsc{Srebro, N.},
  \textsc{and} \textsc{Sridharan, K.} (2010).
\newblock Learnability, stability and uniform convergence.
\newblock \emph{Journal of Machine Learning Research\/}~\textbf{11},~Oct,
  2635--2670.

\bibitem{shen2017iteration}
\textsc{Shen, J.} \textsc{and} \textsc{Li, P.} (2017).
\newblock On the iteration complexity of support recovery via hard thresholding
  pursuit.
\newblock In \emph{Proceedings of the 34th International Conference on Machine
  Learning-Volume 70}. 3115--3124.

\bibitem{sun2019optimization}
\textsc{Sun, R.} (2019).
\newblock Optimization for deep learning: theory and algorithms.
\newblock \emph{arXiv preprint arXiv:1912.08957\/}.

\bibitem{talagrand1994sharper}
\textsc{Talagrand, M.} (1994).
\newblock Sharper bounds for gaussian and empirical processes.
\newblock \emph{The Annals of Probability\/}~\textbf{22},~1, 28--76.

\bibitem{tibshirani1996regression}
\textsc{Tibshirani, R.} (1996).
\newblock Regression shrinkage and selection via the lasso.
\newblock \emph{Journal of the Royal Statistical Society. Series B
  (Methodological)\/}, 267--288.

\bibitem{vapnik2013nature}
\textsc{Vapnik, V.} (2013).
\newblock \emph{The nature of statistical learning theory}.
\newblock Springer science \& business media.

\bibitem{vershynin2010introduction}
\textsc{Vershynin, R.} (2010).
\newblock Introduction to the non-asymptotic analysis of random matrices.
\newblock \emph{arXiv preprint arXiv:1011.3027\/}.

\bibitem{wainwright2009sharp}
\textsc{Wainwright, M.~J.} (2009).
\newblock Sharp thresholds for high-dimensional and noisy sparsity recovery
  using $\ell_1$-constrained quadratic programming (lasso).
\newblock \emph{IEEE Transactions on Information Theory\/}~\textbf{55},~5,
  2183--2202.

\bibitem{yuan2014gradient}
\textsc{Yuan, X.-T.}, \textsc{Li, P.}, \textsc{and} \textsc{Zhang, T.} (2014).
\newblock Gradient hard thresholding pursuit for sparsity-constrained
  optimization.
\newblock In \emph{International Conference on Machine Learning}. 127--135.

\bibitem{yuan2016exact}
\textsc{Yuan, X.-T.}, \textsc{Li, P.}, \textsc{and} \textsc{Zhang, T.} (2016).
\newblock Exact recovery of hard thresholding pursuit.
\newblock In \emph{Advances in Neural Information Processing Systems}.
  3558--3566.

\bibitem{yuan2018gradient}
\textsc{Yuan, X.-T.}, \textsc{Li, P.}, \textsc{and} \textsc{Zhang, T.} (2018).
\newblock Gradient hard thresholding pursuit.
\newblock \emph{Journal of Machine Learning Research\/}~\emph{18}, 1--43.

\bibitem{zhang2010nearly}
\textsc{Zhang, C.-H.} (2010).
\newblock Nearly unbiased variable selection under minimax concave penalty.
\newblock \emph{The Annals of Statistics\/}~\textbf{38},~2, 894--942.

\bibitem{zhang2012general}
\textsc{Zhang, C.-H.} \textsc{and} \textsc{Zhang, T.} (2012).
\newblock A general theory of concave regularization for high-dimensional
  sparse estimation problems.
\newblock \emph{Statistical Science\/}~\textbf{27},~4, 576--593.

\bibitem{zhou2018efficient}
\textsc{Zhou, P.}, \textsc{Yuan, X.-T.}, \textsc{and} \textsc{Feng, J.} (2018).
\newblock Efficient stochastic gradient hard thresholding.
\newblock In \emph{Advances in Neural Information Processing Systems}.
  1984--1993.

\end{thebibliography}

%We denote $\mathrm{H}_k(x)$ as a truncation operator which preserves the top $k$ (in magnitude) entries of vector $x$ and forces the remaining to be zero. The notation $\supp(x)$ represents the index set of nonzero entries of $x$. We conventionally define $\|x\|_\infty = \max_{i}|[x]_i|$ and define $x_{\min}=\min_{i \in \supp(x)} |[x]_i|$. For an index set $S$, we define $[x]_S$ and $[A]_{SS}$ as the restriction of $x$ to $S$ and the restriction of rows and columns of $A$ to $S$, respectively.

\end{document}